\documentclass[11pt]{amsart}
\usepackage{hyperref}
\usepackage{subfiles}
\usepackage{tikz-cd}
\usepackage[utf8]{inputenc}
\usepackage[english]{babel}

\usepackage{geometry} 
\usepackage{verbatim} 
\usepackage{color}
\usepackage{amsmath}
\usepackage{amsfonts, amssymb, amsthm,  latexsym}
\usepackage{multicol}
\usepackage{graphicx}
\usepackage[all,cmtip,curve, knot,frame]{xy} 
\usepackage[at]{easylist}
\usepackage{subfiles}
\usepackage{enumitem}
\usepackage{setspace}
\usepackage{todonotes}
\usepackage{scalefnt}
\usepackage{tikz}
\usepackage{centernot}
\usepackage{csquotes} 

\usepackage{xcolor}
\hypersetup{
    colorlinks,
    linkcolor={red!50!black},
    citecolor={blue!50!black},
    urlcolor={blue!80!black}
}

\usetikzlibrary{knots}

\DeclareUnicodeCharacter{0301}{'}

\newcommand{\mvhide}[1]{}

\newcommand{\Y}{\mathcal{Y}}

\newcommand{\X}{\mathcal{X}}
\newcommand{\HKuniv}{\HKop}
\newcommand{\DD}{\mathbb{D}^2}

\newcommand{\ellen}{\ell en}

\newcommand{\q}{\mathtt{q}}
\newcommand{\Sq}{\boldsymbol{q}} 
\newcommand{\St}{{t}} 
\newcommand{\Spi}{\overline{\pi}} 
\newcommand{\Q}{\mathbb{Q}} 
\newcommand{\K}{\mathcal{K}}

\newcommand{\AffSym}{\widehat{\mathfrak{S}}_n}

\newcommand{\Sym}{{\mathfrak{S}}}
\newcommand{\Zprod}{\boldsymbol{Z}} 
\newcommand{\Sn}{{\Sym}_n}
\newcommand{\SN}{{\Sym}_{N}}
\newcommand{\Sm}[1]{{\Sym}_{{#1}}}
\newcommand{\id}{\operatorname{Id}}

\newcommand{\cA}{\mathcal{A}}

\newcommand{\SkCat}{\mathrm{SkCat}}
\newcommand{\SkMod}{\mathrm{SkMod}}
\newcommand{\SkAlg}{\mathrm{SkAlg}}
\newcommand{\SkAlgN}[1]{\mathrm{SkAlg}_{#1}}
\newcommand{\SkAlgNn}[2]{\mathrm{SkAlg}_{#1,#2}}

\newcommand{\DAHA}[2]{\mathrm{\mathbb{H}_{#2}}}
\newcommand{\DAHAN}[1]{\mathrm{\mathbb{H}_{#1}}(N)}
\newcommand{\DN}{\HH_N} 

\newcommand{\SLat}{\Lambda_{\SLN}}  
\newcommand{\GLat}{\Lambda_{\GLN}} 

\newcommand{\SymX}{S(\X)} 
\newcommand{\SymY}{S(\Y)}

\newcommand{\stab}[2]{Z_{{#1}}(#2)} 
\DeclareMathOperator{\HHz}{HH_0}
\newcommand{\dl}{g_\lambda}

\newcommand{\SWfin}{\mathrm{SW}}
\newcommand{\SWaff}{\dot{\mathrm{SW}}}
\newcommand{\SWdaff}{\ddot{\mathrm{SW}}}

\newcommand{\Uw}{U_\sigma}
\newcommand{\Uwp}{\Uw^{\perp}}

\renewcommand{\Im}{\mathrm{Im}} 

\newcommand{\ZZ}{\mathbb{Z}}

\newcommand{\cD}{\mathcal{D}}
\newcommand{\Dq}{{\cD_\q(G)}}

\newcommand{\DqG}[1]{{\cD_\q(#1)}}

\newcommand{\DqH}{\DqG{H}}

\newcommand{\Uq}{U_\q(\g)}
\newcommand{\Uqgl}{U_\q(\glN)}
\newcommand{\Uqsl}{U_\q(\slN)}

\newcommand{\g}{\mathfrak{g}}
\newcommand{\into}{\to} 
\newcommand{\HH}{\mathbb{H}}
\newcommand{\HG}{\HH_n^{\GL}({q,t})}
\newcommand{\HS}{\HH_n^{\SL}(\Sq,t)}
\newcommand{\HStwo}{\HH_2^{\SL}(\Sq,t)}
\newcommand{\HSweak}{\widetilde \HH_n^{\SL}(\Sq,t)}

\newcommand{\Schmaha}[1]{\widetilde\HH_{#1}^{\SL}(\Sq,t)}
\newcommand{\SchmahaN}[1]{\widetilde\HH_{#1}(N)}
\newcommand{\Schmahan}{\widetilde\HH_{n}}
\newcommand{\Schmaham}[1]{\widetilde\HH_{#1}}
\newcommand{\LCM}{\mathtt{LCM}}
\renewcommand{\H}{\operatorname{H}} 
\newcommand{\Hf}{\operatorname{H}^{\mathrm{fin}}}  
\newcommand{\Hfn}{\Hf_n}

\newcommand{\Hfm}{\Hf_m}
 
\newcommand{\Rep}{\operatorname{Rep}}
\newcommand{\Ann}{\operatorname{Ann}}
\newcommand{\ot}{\otimes}
\newcommand{\rt}[1]{\underset{#1}{\otimes}}

\newcommand{\modu}{\textrm{-mod}}
\newcommand{\Hom}{\operatorname{Hom}}
\newcommand{\SL}{\mathrm{SL}}
\newcommand{\SLN}{\SL_N}
\newcommand{\SLn}{\SL_n}
\newcommand{\GL}{\mathrm{GL}}
\newcommand{\GLN}{\GL_N}
\newcommand{\GLn}{\GL_n}
\newcommand{\slN}{\mathfrak{sl}_N}
\newcommand{\glN}{\mathfrak{gl}_N}

\DeclareMathOperator{\Res}{Res}
\DeclareMathOperator{\Ind}{Ind} 
\DeclareMathOperator{\End}{End}

\newcommand{\HKe}{\HKop(\epsilon)}
\newcommand{\HKop}{\operatorname{HK}}

\newcommand{\Dist}{\operatorname{Dist}}

\newcommand{\Repq}{\mathrm{Rep}_{\q}}

\newcommand{\sgna}[1]{\sgni(#1)}

\newcommand{\lambdatup}{{\underline{{\mathbf \lambda}} }}
\newcommand{\ek}[1]{\mathrm{e}_{{#1}}} 
\newcommand{\eN}{\ek{N}}
\newcommand{\en}{\ek{n}}
\newcommand{\eNj}{\ek{N^j}}
\newcommand{\eNk}{\ek{N^k}}
\newcommand{\emj}{\ek{m^j}}
\newcommand{\eml}{\ek{m^{\ell}}}

\newcommand{\ee}{\mathrm{e}} 
\newcommand{\etrivn}{\etriv{n}}
\newcommand{\etriv}[1]{\mathrm{e}^{+}_{#1}}
\newcommand{\Coinv}{\K_\omega[\Lambda \oplus \Lambda]^{\SN}}  
 
\newcommand{\SmashD}{\K_\omega[\Lambda \oplus \Lambda] \# \SN}

\newcommand{\KLam}{\K_\omega[\Lambda \oplus \Lambda]}

\DeclareMathOperator{\rev}{rev}
\DeclareMathOperator{\sgni}{\mathbf{sgn}}
\DeclareMathOperator{\trivial}{\operatorname{triv}}

\newcommand{\Dqstr}[1]{\mathcal{D}_\q({#1}){\modu}^{\hspace{0.5pt} {#1}}}

\newcommand{\Vect}{\operatorname{Vect}}

\DeclareMathOperator{\Id}{Id}

\newcommand{\seg}[2]{[#1;#2]}

\newcommand{\HX}{\H(\cX)}
\newcommand{\HY}{\H(\Y)}

\renewcommand{\r}{\mathfrak{r}}

\newcommand{\cC}{\mathcal C}

\newcommand{\cK}{\mathcal K}
\newcommand{\cP}{\mathcal P}

\newcommand{\cX}{\mathcal X}

\newcommand{\cZ}{\mathcal Z}

\newcommand{\B}{\mathrm B}
\newcommand{\D}{\mathbb D}

\newcommand{\Z}{\mathbb Z}

\makeatletter
\newcommand*\leftdash{\rotatebox[origin=c]{-45}{$\dabar@\dabar@\dabar@$}}
\newcommand*\rightdash{\rotatebox[origin=c]{45}{$\dabar@\dabar@\dabar@$}}
\makeatother

\newcommand{\omitt}[1]{ }
\newcommand{\Conj}[1]{[#1]}  

\newcommand{\tSpecht}{S}
\newcommand{\Slam}{\tSpecht^\lambda(t)}
\newcommand{\Sblah}[1]{\tSpecht^{#1}(t)}

\tikzstyle{V}=[draw, fill =black, circle, inner sep=0pt, minimum size=1.
5pt]
\tikzstyle{bV}=[draw, fill =black, circle, inner sep=0pt, minimum size=4
pt]
\tikzstyle{over}=[draw=white,double=black, double distance=1.75pt]
\tikzstyle{diagram}=[line width=.75pt, scale=\SCALE]

\def\Over[#1,#2][#3,#4]{ 
        \draw[over]   (#1,#2) .. controls ++(0,#4*.5-#2*.5) and ++(0,-#4*.5+#2*.5) .. (#3,#4);}
\def\Under[#1,#2][#3,#4]{ 
        \draw  (#1,#2) .. controls ++(0,#4*.5-#2*.5) and ++(0,-#4*.5+#2*.5) .. (#3,#4);}
\def\Cross[#1,#2][#3,#4]{
        \Under[#3,#2][#1,#4]\Over[#1,#2][#3,#4]}

\newcommand{\cupcap}{    \begin{tikzpicture}
\begin{scope}[shift={(0,0)}, scale=.5, line width=2pt]
        \draw[line width=2pt] (0,1) .. controls (.15,.55) and (.35, .55)   .. (0.5,1);
        \draw[line width=2pt] (0,0) .. controls (0.15,0.45) and (0.35, 0.45)  .. (0.5,0);
        \node[V] at (.5,1){}; \node[V] at (0.5,0){}; 
        \node[V] at (0,1){}; \node[V] at (0,0){}; 
\end{scope}
\end{tikzpicture} }

\numberwithin{equation}{section}
\newtheorem*{theorem*}{Theorem}
\newtheorem{theorem}{Theorem}[section]
\newtheorem{theorem/def}[theorem]{Theorem/Definition}
\newtheorem{lemma}[theorem]{Lemma}
\newtheorem{proposition}[theorem]{Proposition}
\newtheorem{claim}[theorem]{Claim}
\newtheorem{conjecture}[theorem]{Conjecture}
\newtheorem{corollary}[theorem]{Corollary}
\theoremstyle{definition}
\newtheorem{definition}[theorem]{Definition}
\newtheorem{remark}[theorem]{Remark}
\newtheorem{notation}[theorem]{Notation}
\newtheorem{example}[theorem]{Example}
\newtheorem{def/prop}[theorem]{Definition/Proposition}
\newtheorem{question}[theorem]{Question}

\newcommand{\defterm}[1]{\textbf{#1}}
\newcommand{\dom}{\unlhd} 
\DeclareMathOperator{\sort}{sort}

        \newcommand{\SLGL}{({\color{blue}\diamond})} 

\usepackage[style=alphabetic]{biblatex}
\title[Skeins on tori]{Skeins on tori}
\author{Sam Gunningham, David Jordan, Monica Vazirani}
\date{\today}
\begin{document}

\begin{abstract}
We analyze the $G$-skein theory invariants of the 3-torus $T^3$ and the two-torus $T^2$, for the groups $G = \GLN, \SLN$ and for generic quantum parameter.   We obtain formulas for the dimension of the skein module of $T^3$, and we describe the algebraic structure of the skein category of $T^2$ -- namely of the $n$-point relative skein algebras.


The Schur-Weyl case $n=N$ is special in our analysis. We construct an isomorphism between the $N$-point relative skein algebra and the double affine Hecke algebra at specialized parameters. As a consequence, we prove that all tangles in the relative $N$-point skein algebra are in fact equivalent to linear combinations of braids, modulo skein relations.  More generally for $n$ an integer multiple of $N$, we construct a surjective homomorphism from an appropriate DAHA to the $n$-point relative skein algebra. 

In the case $G=\SL_2$ corresponding to the Kauffman bracket we give proofs directly using skein relations.  Our analysis of skein categories in higher rank hinges instead on 
the combinatorics of multisegment representations when restricting from DAHA to AHA and nonvanishing
properties of parabolic sign idempotents upon them. 
\end{abstract}

 \maketitle
 \tableofcontents
 
\section{Introduction}
This paper contains several contributions to the study of skein theory.  Skein modules of oriented 3-manifolds, as well as skein categories and skein algebras of oriented 2-manifolds, are graphically defined invariants which ``integrate" the local algebraic structure of a quantum group around the global topology of the 3-manifold or surface, respectively.

Because skein theory invariants are defined fully locally, they fit naturally into the framework of topological quantum  field theory (TQFT) and higher Morita theory.  Within quantum topology, skein theory describes both the 4D Crane--Yetter \cite{crane1993evaluating,crane1997state,barrett2007observables} and 3D Witten--Reshetikhin--Turaev TQFT's \cite{witten1989quantum,reshetikhin1991invariants}, and also underpins the quantum $A$-polynomial \cite{FGL,Garoufalidis,garoufalidis2005colored,Sikora2005}.    Distinct variants of the skein theories we consider have recently been related to counting of holomorphic curves \cite{skeins-on-branes} and to the theory of vanishing cycles \cite{Gunningham-Safronov}.  In the categorified setting, so-called skein-lasagna modules \cite{MWW-skein-lasagna},\cite{MN-skein-lasagna},\cite{MWW2-skein-lasagna} have been used to distinguish exotic 4-manifolds \cite{Ren-Willis}.

In physics, skein modules and skein categories describe the space of topological line operators in 3-dimensions, and the categories of surface operators, respectively, in the Marcus/Kapustin--Witten twist of $\mathcal{N}=4, d=4$ gauge theory with group $G$ \cite{Marcus1995,Kapustin-Witten}.  Dually, skein algebras arise in the study of $\mathcal{N}=2$ theories of Class S \cite{gaiotto2023commuting,tachikawa2015skein}.  These appearances echo the original appearance of skein relations in physics, in Witten's seminal paper \cite{witten1989quantum} on quantum Chern-Simons invariants.

By a conjecture of Witten, proved in \cite{GJS}, the $G$-skein module of a closed oriented 3-manifold $M$ is finite-dimensional, for any $G$ and for generic quantum parameter $\q$.  This fact begs the question: what precisely are these dimensions for various 3-manifolds, and what is the deeper enumerative significance of the resulting dimension formulas?  Very little is known, however, at this level of precision about skein modules of 3-manifolds, and even less about skein categories of surfaces.  Where concrete formulas for dimension have been obtained -- namely, in the cases $S^2\times S^1$ \cite{Hoste-Przytycki}, $T^3$ \cite{Carrega,Gilmer}, and $\Sigma_g\times S^1$ \cite{Gilmer-Masbaum,Detcherry-Wolff} --
such formulas are only applicable to the group $G=\SL_2$, corresponding to the Kauffman skein relations.\footnote{While our focus in the paper is to push beyond the case of the Kauffman bracket skein relations, readers interested in this case specifically may see Section \ref{sec:SL2} for some new results there, which bypass the type $A$ combinatorics needed for the rest of the paper.}

In this paper we focus on the $\GLN$- and $\SLN$-skein invariants, of the  tori $T^2$ and $T^3$
at generic value of the quantum parameter $\q$. Specifically, we prove a finiteness property -- a categorical analogue of Witten's finiteness conjecture -- yielding a finite presentation of $\SLN$- and $\GLN$- skein categories of the two-torus $T^2$.  Passing to Hochschild homology, we obtain formulas for dimensions of $\GLN$- and $\SLN$-skein modules of $T^3$; to our knowledge these are the first known dimensions beyond the case $G=\SL_2$ of the Kauffman bracket skein module. The interesting combinatorics which appear in our formulas suggest many generalisations and connections to type $A$ algebraic combinatorics.

Our main technical tool is a form of ``elliptic" Schur-Weyl duality:  the natural homomorphism 
from the relevant braid group of the torus to the relative  skein algebra (where each strand is labeled by the defining representation of $G=\GLN$ or $\SLN$) factors through a suitable version of the double affine Hecke algebra (DAHA), and in some cases this yields an isomorphism between the skein algebra and the DAHA at appropriate parameters \cite{J2008,Jordan-Vazirani,GJV}.

\medskip

We turn now to an enumeration of our main results, which we have ordered starting with the most concrete/elementary, moving onto the most abstract/fundamental results from which the former are derived.

\subsection{The skein algebra of the two-torus $T^2$}
Our first main result is a complete computation of the zeroeth Hochschild homology of the $\GLN$ and $\SLN$ skein algebras of the two-torus.  In section \ref{sec:HH skein algebras} we prove:

\begin{theorem} \label{thm: HHz of GL skein algebra}
Suppose that the quantum parameter $\q$ is generic. 
The dimensions of the zeroeth Hochschild homology of the skein algebras for $G =\GLN$ is given by
$$\dim \HHz(\SkAlgN{\GLN}(T^2)) = \cP(N).
$$
For $N=1, 2,3,4,\ldots$ this yields the dimensions:
\[
1,
2,
 3,
 5,
 7,
 11,
 15,
 22,
 30,
 42,
 56,
 77,
 101,
 135,
 176,
 231,
 297,
 \cdots
\]
\end{theorem}

\begin{theorem} \label{thm: HHz of SL skein algebra}
Suppose that the quantum parameter $\q$ is generic. 
The dimensions of the zeroeth Hochschild homology of the skein algebras for $G =\SLN$ is given by
$$ \dim \HHz(\SkAlgN{\SLN}(T^2)) = (\cP \star J_2)(N) = \sum_{\lambda \vdash N} \gcd(\lambda)^2 
= \sum_{v
\in (\Z/N\Z)^{\oplus 2}} \cP(\gcd(v, N)).
$$
For $N=2,3,4,\ldots$ this yields the dimensions:
\[
5, 11, 23, 31, 60, 63, 109, 126, 183, 176, 330, 269, 420, 496, 645,
585, 
,\ldots
\]
\end{theorem}

In the formulas above, $\mathcal{P}(N)$ denotes the partition number of $N$, i.e. the number of ways to write $N$ as a sum of positive integers, and the $\gcd$ of a partition $\lambda$ of $N$ is defined as the greatest common divisor  of  its parts. 
For $(v_1, \cdots, v_k ) \in(\Z/N\Z)^{\oplus k} $  we write $\gcd(v, N)$ for $\gcd(v_1, \cdots, v_k, N)$.
Given two functions $F$ and $G$ of positive integers,
$F\star G$ denotes their multiplicative convolution, \[(F\star G)(N) = \sum_{d|N}F(d)G(N/d),\] and finally $J_k$ denotes the $k$th Jordan totient function,
\[
J_k(d) = d^k \prod_{\textrm{prime } p | d}(1-\frac{1}{p^k}),
\]
which counts the order $d$ elements in $(\mathbb{Z}/N\mathbb{Z})^{\oplus k}$, for any multiple $N$ of $d$,
or equivalently the size of the  $\GL_k(\mathbb{Z}/N\mathbb{Z})$ orbit of the vector $(\frac Nd,0,\ldots,0) \in (\mathbb{Z}/N\mathbb{Z})^{\oplus k}$.

The proof of Theorem \ref{thm: HHz of GL skein algebra} and \ref{thm: HHz of SL skein algebra} are elementary once we invoke a Morita equivalence of the skein algebra with the smash product of a quantum torus and the symmetric group; and we compute the latter directly using a simple combinatorial reformulation. 

\subsection{The skein module of the three-torus $T^3$}
Our next main results are as follows.
\begin{theorem}\label{mainthm:dimensionsGLN}
Suppose that the quantum parameter $\q$ is generic. 
The $\GLN$-skein module of $T^3$ has dimension:
\[
\dim \SkMod_{\GLN}(T^3) = \mathcal{P}(N).
\]
For $N=1, 2,3,4,\ldots$ this yields the dimensions:
\[
1,
2,
 3,
 5,
 7,
 11,
 15,
 22,
 30,
 42,
 56,
 77,
 101,
 135,
 176,
 231,
 297,
 \cdots
\]
\end{theorem}
\begin{theorem}\label{mainthm:dimensionsSLN}
Suppose that the quantum parameter $\q$ is generic. 
The $\SLN$-skein module of $T^3$ has dimension:
\[
\dim \SkMod_{\SLN}(T^3) = (\mathcal{P}\star J_3)(N)  = \sum_{\lambda \vdash N} \gcd(\lambda)^3
= \sum_{v \in (\Z/N\Z)^{\oplus 3}} \cP(\gcd(v,N))\]
For $N=2,3,4,\ldots$ this yields the dimensions:
\[
9, 29, 75, 131, 266, 357, 617, 810, 1207, 1386, 2272,
2297, 3318, 3954, 5145, 5209, 7745, 7348,
\ldots
\]

\end{theorem}

Our proofs of Theorems \ref{mainthm:dimensionsGLN}, \ref{mainthm:dimensionsSLN}  rely on an identification of the skein module of $T^3$ with the (zeroeth) Hochschild homology of the skein \emph{category} of $T^2$.  Whereas Theorems \ref{thm: HHz of GL skein algebra} and \ref{thm: HHz of SL skein algebra} involves similar formulas for skein \emph{algebras} of $T^2$, this is not enough, and instead we need to understand the full structure of the skein category, as laid out in the following subsection.

\subsection{The skein category of the two-torus $T^2$}
Our most important general result in the paper is the following description of $\GLN$- and $\SLN$- skein categories at generic parameters:

\begin{theorem}\label{mainthm:generators} Suppose that the quantum parameter $\q$ is generic.  Then the skein categories $\SkCat_{\GLN}(T^2)$ and $\SkCat_{\SLN}(T^2)$ are each Morita equivalent to finitely presented algebras.  Specifically:
\begin{enumerate}  
    \item The abelian category $\SkCat_{\GLN}(T^2)\modu$ is generated by the compact projective object $\Dist$, and hence $\SkCat_{\GLN}(T^2)$ is Morita equivalent to $\SkAlg_{\GLN}(T^2)$.
    \item The abelian category $\SkCat_{\SLN}(T^2)\modu$ is generated by the compact projective object
    \[\Dist \oplus \Dist_{V} \oplus \ldots \oplus \Dist_{V^{\otimes (N-1)}}\]
    where $V$ is the defining representation of $\SLN$.  The summands are orthogonal for the $\Hom$ pairing, hence $\SkCat_{\SLN}(T^2)$ is Morita equivalent to $\prod_{m=0}^{N-1}\End(\Dist_{V^{\ot m}})$.
\end{enumerate}
\end{theorem}

Here, given a finite-dimensional representation $X$ of $\Uqgl$ or of $\Uqsl$, we denote by $\Dist_X$ the compact projective object of the skein category obtained by coloring some framed disk 
on the surface by finite-dimensional representation $X$ of the quantum group $\Uq$, and write simply $\Dist$ in case $X$ is the monoidal unit (see Section \ref{sec: skein categories} for details).   We note that later in the paper $\End(\Dist_{V^{\ot n}})$ is referred to as the relative skein algebra and denoted $\SkAlgNn{G}{n}$.

Recall that the existence of a compact projective generator $P$ of some category $\cC$ gives rise to an equivalence of categories $\cC\simeq \End_\cC(P)\modu$.
{\it A priori}, skein categories of surfaces admit a \emph{countably infinite} collection of compact projective objects which \emph{collectively} generate, namely those obtained from a single framed disk on the surface colored with an arbitrary finite-dimensional representations of the quantum group; however the infinite sum of compact objects is not compact.  The content of the above assertions is that for $\GLN$ it suffices to take only the trivial representation (equivalently, we may simply take the empty set), while for $\SLN$ it suffices to take finitely many tensor powers $V^{\ot 0},\ldots,V^{\ot (N-1)}$, of the defining representation $V$ (equivalently, to take $0,\ldots, N-1$ 
framed disks each colored by $V$).

\begin{remark}
We emphasize that Theorem \ref{mainthm:generators} holds only at generic parameter.  For example, the first statement says that the $\GLN$-skein category of the torus is simply Morita equivalent to  its skein algebra. At $\q=1$ this statement would read that the $\GLN$-character stack is affine, non-stacky, and in fact coincident with the $\GLN$-character \emph{variety} -- this is clearly false.  Compact generation is a purely generic quantum phenomenon.
\end{remark}


Theorem \ref{mainthm:generators} has the following interesting corollary, which immediately follows using the co-end identity for composition of cobordisms in the skein theory TQFT (see, e.g. Theorem 2.5 and Lemma 4.5 from \cite{GJS} for precise statements, both due to Kevin Walker \cite{Walker}).

\begin{corollary}
Suppose that the quantum parameter $\q$ is generic, and fix inside the oriented 3-manifold $M$ some embedded torus $T^2\subset M$. Then:
\begin{enumerate}
    \item The $\SL_2$-skein module $\SkMod_{\SL_2}(M)$ is spanned by tangles intersecting $T^2\subset M$ at most once.  In particular, if $T^2\subset M$ is separating, then $\SkMod_{\SL_2}(M)$ is spanned by skeins which do not intersect $T^2\subset M$.
    \item The $\SLN$-skein module $\SkMod_{\SLN}(M)$ is spanned by $\SLN$-spiders intersecting $T^2\subset M$ at most $N-1$ times.
    \item The $\GLN$-skein module $\SkMod_{\GLN}(M)$ is spanned by $\GLN$-spiders which do not intersect $T^2\subset M$.

\end{enumerate}

\end{corollary}

Motivated by Theorem \ref{mainthm:generators}, and by the finite-dimensionality of skein modules of oriented 3-manifolds, we propose the following categorification of Witten's conjecture:

\begin{conjecture}[Categorical Finiteness Conjecture]\label{conj:skein-cpt-gens}
Let $G$ be a reductive group, let $\Sigma$ be a closed compact surface, and let $\SkCat_G(\Sigma)$ denote the $G$-skein category at a generic value of the quantization parameter.  Then the abelian category $\SkCat_G(\Sigma)\modu$ admits a compact projective generator.
\end{conjecture}

Given that finite-dimensionality of skein modules was proven using general tools in deformation quantization theory, we are led to ask more generally:

\begin{question}\label{qn:defquant}
What conditions on a Poisson stack ensure that a generic deformation quantization -- as a presentable abelian category, admits a compact projective generator?
\end{question}

For example, we expect that this holds when the Poisson stack arises as the quasi-Hamiltonion reduction of a smooth symplectic variety acted upon by a reductive algebraic group (as is the case for the $G$-character stack of a closed surface).

We note that in the setting of algebraic $\cD$-modules (that is, when the stack in question is a cotangent bundle), Question \ref{qn:defquant} has already been posed and studied in \cite{BGR}.

\begin{remark}
Renaud Detcherry and the second author have recently confirmed Conjecture \ref{conj:skein-cpt-gens} for all genus $g$ surfaces $\Sigma_g$ in the case that the group $G$ is $\SL_{2}$.
\end{remark}

In the case $\Sigma=T^2$, a refinement of Conjecture \ref{conj:skein-cpt-gens} appeared in Conjecture 1.10 of \cite{GJV}, namely we proposed a precise description of the $G$-skein category of $T^2$ in terms of elliptic cuspidal data for $G$. In the case $G=\SLN$ where the cuspidal data are well-understood, this leads to the following prediction:

\begin{conjecture}\label{conj:SL_N generalized Springer}    
    We have an equivalence of categories
    \[
\SkCat_{\SLN}(T^2) \simeq \bigoplus_{d \mid N} \SkAlg_{\SL_{N/d}}\modu^{\, \oplus d^2 J_1(d)} .
    \]
\end{conjecture}
Moreover, each of the $\SL_{N/d}$-skein algebras has an explicit description in terms of the anti-spherical DAHA,
or equivalently, symmetric group invariants in a quantum torus (see Section \ref{sec:DAHA intro}). As a consistency check, we note that the the formula for $\dim \HHz (\SkCat_{\SLN}(T^2))$ arising from Conjecture \ref{conj:SL_N generalized Springer} lead to the following formula:
\[
\sum_{d \mid N} d^2\phi(d) \dim \HHz (\SkAlg_{\SLN}(T^2)).
\]
This reduces to the formula in Theorem \ref{mainthm:dimensionsSLN} using Theorem \ref{thm: HHz of SL skein algebra} and the combinatorial identity $\Id_2 \star J_2=J_3$ (where $\Id_2$ denotes the function $n\mapsto n^2$).

\subsection{The relative skein algebra of the two-torus $T^2$}

Let $\SkAlgNn{G}{n}$ denote either the relative skein algebra of the torus $T^2$ with $n$ strands
for $G=\SLN$, or the relative skein algebra of the torus with $n$ standard strands and a strand carrying $\det^{\otimes \frac{-n}{N}}$  for $G=\GLN$.  
In either case, let $\ee\in \End(V^{\ot N})$ denote the anti-symmetrizing idempotent, projecting onto the trivial (respectively, the determinant) representation of $\Uq$ appearing in the $N$th tensor power of the defining representation $V$.
When $n \ge N$, one can identify $\ee \SkAlgNn{G}{n} \ee$  with $ \SkAlgNn{G}{n-N}$.
By abuse of notation, we write $\ee$, or sometimes $\eN$, to denote this (sign)  idempotent both as an element of the finite Hecke algebra (via Schur-Weyl duality) and as an element of the skein algebra involving the first $N$ of the $n$ strands when $N \le n$. 

\begin{theorem}\label{thm:relskeins}
Suppose that the quantum parameter $\q$ is generic.  Then for $n\geq N$ we have
\[
\SkAlgNn{G}{n}\cdot \eN \cdot \SkAlgNn{G}{n} = \SkAlgNn{G}{n}.
\]
Consequently, we obtain a Morita equivalence between $\SkAlgNn{G}{n}$ and $\eN \cdot \SkAlgNn{G}{n} \cdot \eN \cong \SkAlgNn{G}{n-N}$ 
 for $G= \GLN$ or $\SLN$. 
\end{theorem}

\subsection{The double affine Hecke algebras}\label{sec:DAHA intro}
Our approach to the skein category of $T^2$ is to reduce questions and computations there to questions about double affine Hecke algebras (DAHAs).  Let $G=\GLN$ or $\SLN$.  In either case, throughout the remainder of the introduction let us denote by $\DAHAN{n}$ the relevant specialisation of Cherednik's double affine Hecke algebra in type $\mathrm{A}$. See Section \ref{sec:AHA} for a presentation, and details about the specialisation of parameters, which we may express in terms of the quantum parameter $\q$, necessary to define a homomorphism to the skein algebra. We introduce an intermediate algebra $\SchmahaN{n}$ that has $\DAHAN{n}$ for $G=\SL$ as a quotient.  When $N \nmid n$, we need to work with $\SchmahaN{n}$ rather than $\DAHAN{n}$. 

In our conventions we express $\DAHAN{n}$ as a quotient of the group algebra of the $n$th braid group of the torus $T^2$ (or the $n,1$ elliptic braid group for $G=\GLN$).
The natural homomorphism from the
braid group 
to the relative skein algebra descends to  homomorphisms,
\begin{gather} \label{eqn:SW}
   \SWdaff: \DAHAN{n} \longrightarrow \SkAlgNn{\GLN}{n}(T^2) \notag \\ 
   \SWdaff: \SchmahaN{n} \longrightarrow \SkAlgNn{\SLN}{n}(T^2). 
\end{gather}

Further if $n=kN$ for $k \in \Z_{>0}$, the map from $\SchmahaN{n}$ factors through $\DAHAN{kN}$:
\begin{gather*}
   \SWdaff: \DAHAN{kN} \longrightarrow \SkAlgNn{\SLN}{kN}(T^2).
\end{gather*}

At $k=1$ we recall the following theorem from \cite{GJV} -- see \cite{Morton-Samuelson}, \cite{MS-Daha-skein},  and \cite{Mellit-and-co} for closely related results at different specialisations, and see \cite[Section 1.7]{GJV} for a longer discussion how the various results inter-relate. 

\begin{theorem} \label{thm:GJV eHe=Sk} Upon restriction to the antispherical subalgebra, we obtain an isomorphism,
\[
\ee\SWdaff\ee: \eN\cdot \DAHAN{N} \cdot \eN \overset{\sim}{\longrightarrow} \SkAlgNn{G}{0} (T^2)
\]
between the skein algebra and the antispherical subalgebra $\eN\cdot \HH_N(N) \cdot \eN$ of the double affine Hecke algebra for $G= \GLN$ or $\SLN$.
\end{theorem}

In this paper we prove the following key result:

\begin{theorem}\label{mainthm:morita-DAHA} Suppose the quantum parameter $\q$ is not a root of unity. 
Then we have   
\[
\DAHAN{n}\cdot \eN\cdot \DAHAN{n} = \DAHAN{n} \text{ and }
\SchmahaN{n}\cdot \eN\cdot \SchmahaN{n} = \SchmahaN{n}.
\]
 Consequently, we  obtain  a Morita equivalence between $\DAHAN{n}$ and 
 $\eN \cdot  \DAHAN{n} \cdot \eN$  as well as between $\SchmahaN{n}$ and $\eN \cdot \SchmahaN{n} \cdot \eN$.
\end{theorem}

Theorem \ref{mainthm:morita-DAHA} is a special case of Theorem \ref{mainthm:morita-DAHA-eNj} in Section \ref{sec:apps to skeins}, which in turn holds at one particular specialization of parameters 
in 
Theorem  \ref{thm:idempotent}, which is more general and which we prove  in Section \ref{sec:DAHA stuff}. 
Theorem  \ref{thm:idempotent} (along with Proposition \ref{prop:signn neq 0}) is stated for a vast array of specializations as well as a wider family of idempotents.  The proof involves first reducing the claim, following an argument of \cite{Bezrukavnikov-Etingof}, to showing that $\eN$ is conservative on the subcategory of $\Y$-finite modules (i.e. that $\eN \cdot M \neq 0$ for any nonzero $\Y$-finite DAHA-module $M$).  For this, we restrict to the AHA $\HY$ action, and employ the rich combinatorics of multisegments to find an AHA composition factor not killed by $\eN$.  Note this $\eN$ is a partial or parabolic sign idempotent involving just $N$, not $n$, strands. 

\medskip

We obtain the following corollaries of Theorem \ref{mainthm:morita-DAHA}, which extend the isomorphism of Theorem \ref{thm:GJV eHe=Sk} to the full DAHA.  

\begin{corollary}\label{cor:SkAlgN}
Let $G=\GLN$ or $\SLN$.   Then the Schur-Weyl map,
\[
\SWdaff: \DAHAN{N} \longrightarrow \SkAlgNn{G}{N} (T^2),
\]
is an isomorphism.
\end{corollary}

\begin{remark}
We note that a very similar statement to Corollary \ref{cor:SkAlgN}, involving the HOMFLYPT skein algebra, appeared as Conjecture 1.5 from \cite{Morton-Samuelson}, and was proved in \cite{Mellit-and-co}. Recall that the $\GLN$-DAHA depends on two scalar parameters, which in \cite{Mellit-and-co} are assumed to be algebraically independent.  In our result, in place of the HOMFLYPT skein algbra we have the $\GLN$- or $\SLN$-skein algebra, and correspondingly we work with a specialisation, in which both parameters of the DAHA are a power of the quantum parameter $\q$ (however we do assume the quantum parameter is generic). 
We emphasize that this is a genuine rather than merely technical distinction between the two types of result: $\SkAlgNn{\GLN}{n}$ is in fact the zero algebra for $n$ not a multiple of $N$ (see Lemma \ref{lem:GLdeg}).
\end{remark}

Theorem \ref{mainthm:morita-DAHA} implies Theorem \ref{thm:relskeins} by standard arguments from Morita theory, and also implies the following corollary, which is of independent interest:

\begin{corollary} \label{cor:SkAlgGLNn} Let $G=\GLN$ or $\SLN$, and let $n =kN$, for some $k\in\Z_{>0}$.  Then the Schur-Weyl map,
\[
\SWdaff: \DAHAN{n} \longrightarrow \SkAlgNn{G}{n}(T^2),
\]
is a surjection.
\end{corollary}
The content of this corollary is the assertion that every relative tangle on the torus can be expressed as a linear combination of braid words,
modulo the skein relations.  Despite its elementary formulation we are not aware of an elementary proof.

Our proof of Theorems \ref{thm: HHz of GL skein algebra}, \ref{thm: HHz of SL skein algebra} hence of Theorems \ref{mainthm:dimensionsGLN}, \ref{mainthm:dimensionsSLN}, also hinge on facts about DAHA's.  Specifically, in order to compute the Hochschild homology of the skein algebra, we recall from \cite{GJV} that Theorem \ref{thm:GJV eHe=Sk} along with an application of the shift isomorphism \cite{Marshall1999} yields an isomorphism,
\begin{gather} \label{eq:Sk = KLamW}
   \SkAlgN{G}(T^2) \cong  \DqH^W.
\end{gather}
where $W$ is the Weyl group.
In Section \ref{sec:HH skein algebras} we rely on explicit presentation of  the quantum torus $\DqH$, denoted $\KLam$, built from the weight lattice $\Lambda$ of $G$, and $\KLam^W$ denotes its invariants for the natural $W$-action by algebra automorphisms.  By \cite[Theorem 2.4]{Montgomery1980}, this is further Morita equivalent to the smash product $\KLam\# W$.  We have an explicit decomposition   of $\HHz(\KLam\# W)$ 
(see \cite{Shepler-Witherspoon},\cite{Stefan},\cite{AFLS})
which we compute combinatorially in Section \ref{sec:HH skein algebras} in  order to yield the dimensions asserted in Theorem \ref{mainthm:dimensionsGLN}.

\subsection{Outline}
We now outline the structure of the paper, and enumerate where to find proofs of each theorem stated in the introduction.  Section \ref{sec:skeins} recalls basic definitions from skein theory and a few standard results from the Morita theory of linear categories.  Section \ref{sec:AHA} defines the various Hecke algebras appearing in our constructions, and some well-known representation theory of affine Hecke algebras which we will require.  The first new results appear in Section \ref{sec:HH skein algebras}, where we compute the Hochschild homology of skein algebras, in particular proving Theorems \ref{thm: HHz of GL skein algebra} and \ref{thm: HHz of SL skein algebra}.  Section \ref{sec:DAHA stuff} is the technical core of the paper, and involves a combinatorial proof of Theorem \ref{mainthm:morita-DAHA}.  Section \ref{sec:SL2} derives the complete structure of the Kauffman bracket skein category of $T^2$, corresponding to the group $G=\SL_2$.  Section \ref{sec:apps to skeins} brings together the results of the previous sections to prove all the remaining theorems stated in the introduction.
Section \ref{sec:quantum character} interprets the results of the present paper in the language of quantum Springer theory, and strengthens some results proved in \cite{GJV}. 

\subsection{Acknowledgements}
We would like to thank L{\'e}a Bittmann, Alex Chandler, Renaud Detcherry, Patrick Kinnear, Anton Mellit, Chiara Novarini, Peter Samuelson, Jos\'e Simental, and Kevin Walker for helpful discussions.  We are especially grateful to Haiping Yang, who first suggested the formulas in Theorems \ref{mainthm:dimensionsGLN} and \ref{mainthm:dimensionsSLN} and proposed a proof based on computing the top de Rham cohomology of the symplectic resolution of the character variety.

The first author was partially supported by NSF grant DMS-2202363. The second author was partially supported by ERC Starting Grant no. 637618, and by the Simons Foundation award 888988 as part of the Simons Collaboration on Global Categorical Symmetry.
The third author was partially supported by Simons Foundation Collaboration Grants 319233 and 707426. 
The authors are grateful to the International Centre for Mathematical Sciences Research in Groups Programme, to Aspen Center for Physics, which is supported by National Science Foundation grant PHY-1607611, and to Mathematical Sciences Research Institute (MSRI), now named the Simons Laufer Mathematical Sciences Institute (SLMath), which is supported by the National Science Foundation (Grant No. DMS-1928930), all of whom hosted research visits during which parts of this work was undertaken.

\section{Preliminaries on skein theory}\label{sec:skeins}
In this section we recall basic definitions from skein theory, and the special case of $\SLN$- and $\GLN$-spiders (in particular for $\SL_2$ it recovers the Kauffman bracket skein category).  We also recall some basic notions from enriched category theory and how they apply to skein theory.

\subsection{Skein modules}

\begin{definition}[See {\cite[Section 4.2]{Cooke}} for a more detailed definition.]
Fix a $\cK$-linear ribbon braided tensor category $\cA$, where $\cK$ is a commutative ring.

\begin{itemize}
\item Let $\Sigma$ be an oriented surface. 
 An \defterm{$\cA$-labelling of $\Sigma$} is the data, $X$, of an oriented embedding of finitely many (possibly zero) disjoint disks $x_1, \dots, x_n\colon \DD\rightarrow \Sigma$ labeled by objects $V_1, \dots, V_n$ of $\cA$.  We denote by $\vec{x_i}$ the $x$ axis sitting inside each disk $x_i$, and denote $\vec{X}=\cup_i\vec{x_i}$.

\item Let $M$ be an oriented 3-manifold, possibly with boundary $\partial M = \rev \Sigma_{\mathrm{in}}\sqcup \Sigma_{\mathrm{out}}$.  A ribbon graph $\Gamma$ in $M$ has ``ribbons" connected at ``coupons."  As topological spaces, ribbons and coupons are simply embedded rectangles $I\times I$, 
however, we require that ribbons begin and end at either the top ``outgoing," or bottom ``incoming," boundary interval of some coupon, or else at $\partial M$.
\item An \defterm{$\cA$-coloring} of a ribbon graph is a labelling of each ribbon by an object of $\cA$, and of each coupon by a morphism from the (ordered) tensor product of incoming edge-labels to the (ordered) tensor product of outgoing edge-labels.

\item We say that an $\cA$-colored ribbon graph $\Gamma$ is \defterm{compatible with an $\cA$-labelling} $X$ 
of $\partial M$
if $\partial \Gamma={X}$, by which we mean $\partial\Gamma= \vec{X}$ and the corresponding object labels agree.
We denote by $\mathrm{Rib}_\cA(M,X)$ the free $\K$-module with basis the $\cA$-colored ribbon graphs on $M$ compatible with $X$.

\end{itemize}
\end{definition}

Consider the $3$-ball $\DD\times I$, where $I=[0,1]$ and consider a labelling $X\cup Y$ with disks $X=\{(x_1,V_1),\ldots,(x_n,V_n)\}$ embedded in $\DD\times\{0\}$ and $Y=\{(y_1,W_1),\ldots (y_m,W_m)\}$ embedded in $\DD\times \{1\}$.  Then we have a well-defined surjection, 
\[\mathrm{Rib}_\cA(\DD\times I,X\cup Y) \to \Hom_\cA(V_1\otimes \cdots \otimes V_n,W_1\otimes \cdots \otimes W_m),\]
see \cite{TuraevBook}. We will call the kernel of this map the \defterm{skein relations} between $X$ and $Y$.

\begin{definition}
Let $M$ be an oriented $3$-manifold possibly with boundary $\partial M\cong \rev{\Sigma}_{\mathrm{in}}\sqcup \Sigma_{\mathrm{out}}$, and $\cA$-labellings $X_{\mathrm{in}}$ of $\Sigma_{\mathrm{in}}$ and $X_{\mathrm{out}}$ of $\Sigma_{\mathrm{out}}$.  The \defterm{relative $\cA$-skein module} $\SkMod_\cA(M, X_{\mathrm{in}}, X_{\mathrm{out}})$ is the $\K$-module spanned by isotopy classes of $\cA$-colored ribbon graphs in $M$ compatible with $X_{\mathrm{in}}\cup X_{\mathrm{out}}$, taken modulo the skein relations determined by any oriented ball $\DD\times I\subset M$. In the case $\partial M=\emptyset$, we write $\SkMod_\cA(M)$ and call this the \defterm{$\cA$-skein module} of $M$.
\end{definition}

\subsection{Skein categories} \label{sec: skein categories}

\begin{definition}
Let $\Sigma$ be an oriented surface. The \defterm{skein category} $\SkCat_\cA(\Sigma)$ of $\Sigma$ has:
\begin{itemize}
\item As its objects, $\cA$-labellings of $\Sigma$.

\item As the 1-morphisms from $X$ to $Y$ the relative $\cA$-skein module $\SkMod_\cA(\Sigma\times I,X,Y)$.
\end{itemize}
\end{definition}

\begin{example}
In the case $\cA=\Repq(\SL_2)$ (with suitable choice of ribbon structure), the constructions above recover the Kauffman bracket skein theories. 
\end{example}

\begin{example}
In the case $\cA=\Repq(\SLN)$ (with suitable choice of ribbon structure), the constructions above recover the $\SLN$-spider constructions \cite{CKM} (see also \cite{Sikora2005} and \cite{Poudel-comparison}).  In the case $\cA=\Repq(\GLN)$ see \cite{tubbenhauer}
for the corresponding notion of $\GLN$-spider.
\end{example}

\begin{notation}
Let us fix once and for all an embedding $\DD \into \Sigma$.  Then, given an object $W\in \cA$, we have the \defterm{distinguished object} $\Dist_W\in\SkCat_\cA(\Sigma)$, obtained by labelling the $x$-axis of $\DD$ by $W$. Every object of the skein category is isomorphic to $\Dist_W$ for some $W$, with such an isomorphism obtained by any tangle collecting all the disjoint disks 
$x_1,\ldots, x_n$ into $\DD$, and accordingly taking $W=V_1\otimes \cdots \otimes V_n$.

An important special case is when $W=\mathbf{1}\in \cA$ is the unit, and in this case we denote the resulting  by  $\Dist$.  This object is even more canonical than the distinguished objects $\Dist_W$: while $\Dist_W$ is independent up to isomorphism of the chosen disk $\DD\into \Sigma$, clearly the objects $\Dist$ obtained from any two choices of $\DD\into \Sigma$ are canonically isomorphic. 

When the ribbon tensor category $\cA$ is clear from context (typically $\cA=\Repq(\GLN)$ or $\Repq(\SLN)$), we will abbreviate the Hom-functor $\Hom_{\SkCat_G(\Sigma)}$ simply by $\Hom_\Sigma$, likewise $\End_\Sigma$.  
\end{notation}

\begin{notation}
Let $G=\SLN$, and fix a non-negative integer $n$. We denote by $V$ the defining representation of dimension $N$. The \defterm{relative skein algebra} $\SkAlgNn{N}{n}(\Sigma)$ of $\Sigma$ is the endomormorphism algebra,
\[\SkAlgNn{N}{n}(\Sigma) = \End_{\Sigma}(\Dist_{V^{\ot n}}).\]
For $G=\GLN$ we assume further that $n=kN$ for a non-negative integer $k$, and similarly we write,
\[
\SkAlgNn{N}{n}(\Sigma) = \End_{\Sigma}(\Dist_{V^{\ot n}\otimes (\Lambda_q^N(V)^*)^{\ot k}}).
\]
Often in the paper $\Sigma=T^2$ is clear from context and we will write simply $\SkAlgNn{N}{n}$.
\end{notation}

\subsection{Modules, Morita equivalence, Hochschild homology}

Skein categories defined in this way are linear but not abelian categories, and so are more akin to ``many pointed algebras" (for this reason skein categories also sometimes called ``skein algebroids").  Given such a small linear category $\cC$, one studies the abelian category,
\[
\cC\modu := \operatorname{Fun}(\cC^{op},\Vect),
\]
which is sometimes called the \defterm{category of $\cC$-modules}, or alternatively the \defterm{free co-completion} of $\cC$.  Two small categories with equivalent free co-completions are said to be \defterm{Morita equivalent}; in fact a Morita equivalence occurs between two small categories if, and only if, they have equivalent idempotent completions.  

We recall that a subcategory $\cD$ in an abelian category $\cC$ is said to \defterm{generate} $\cC$
if the restriction to $\cD$ of the contravariant Hom functor in $\cC$ is conservative, i.e. if for any non-zero object $c\in\cC$, there exists some $d \in \cD$ with $\Hom(d,c)\neq 0$.  We say that $\cC$ \defterm{admits a compact projective generator} if such a $\cD$ can be chosen to consist of a single compact projective object $X$, where we recall that $X$ is \defterm{compact projective} if its covariant Hom functor $\Hom(X,-)$ is colimit-preserving. Given a compact projective generator $X$, we obtain an equivalence of categories $\cD\simeq \End(X)\modu$.

Recall that the (zeroeth) \defterm{Hochschild homology} $\HHz(\cC)$ of a small $\K$-linear category $\cC$ is the vector space,
\[
\HHz(\cC) = \bigoplus_{P\in \cC} \Hom_\cC(P,P) \,\,\Big{/}\,\, \left\{f\circ g - g\circ f \,\,\big{|}\,\, f:P\to Q,\,\, g: Q\to P\right\}.
\]
Here the sum is taken over the objects of $\cC$, and the relations are taken over all (doubly) composable pairs of morphisms $f,g$, noting that in general the two compositions will appear in different direct summands.  We note that the Hochschild homology factors through the free cocompletion: if two categories $\cC$ and $\cD$ are Morita equivalent then they have isomorphic Hochschild homologies.  We note that in the case $\cC$ has a unique object $X$, then $\HHz(\cC)$ coincides with $\HHz(\End(X))$.

\begin{remark}
$\K$-linear categories with enough compact projectives define dualizable objects in the 2-category of locally presentable $\K$-linear categories.  The value of the associated 1-dimensional topological field theory on the circle $S^1$ is precisely the Hochschild homology as defined above.  
\end{remark}

\begin{proposition}\label{prop:HH0}
Let $\Sigma$ be a closed and oriented surface.  We have a canonical isomorphism of vector spaces,
\[
\SkMod_\cA(\Sigma\times S^1) = \HHz(\SkCat(\Sigma)).
\]
\end{proposition}

\begin{remark}
Proposition \ref{prop:HH0} is a special case of the fact that skein modules, and skein categories combine to give a categorified (3,2)-TQFT, see \cite{Walker, JF, GJS}).
\end{remark}
\begin{remark}
As illustrated by Proposition \ref{prop:HH0}, it is not the skein category but rather its free co-completion (equivalently, its Morita class) which plays an essential role in TQFT.
\end{remark}

For later use we recall the following standard propositions concerning Morita equivalences of $\K$-algebras.

\begin{proposition}\label{prop:Morita}
Suppose that $e$ is an idempotent in some unital $\K$-algebra $A$, and consider the $eAe$-$A$-bimodule $eA$.  The following are equivalent:
\begin{enumerate}
    \item The bimodule $eA$ defines a Morita equivalence.
    \item The functor $e:A\modu \to eAe\modu$ is conservative, i.e. $eM=0 \implies M=0$.
    \item We have an equality $AeA=A$.
\end{enumerate}
\end{proposition}
\begin{proof}
    This is straightforward.
\end{proof}
 \begin{proposition}\label{prop:idems}
     Let $A$, $B$ be unital  $\K$-algebras with an idempotent $e^2=e \in A$ and an algebra homomorphism 
     $\varphi: A \to B$.  Write $\bar e = \varphi(e)$. If $A=AeA$ then $B = B \bar e B$.
 \end{proposition}
 \begin{proof}
 
 Since $1_A \in AeA$ we have $1_B = \varphi(1_A) \in \varphi(AeA) \subseteq B \varphi(e) B = B \bar e B $. 
 Thus we conclude $B = B \bar e B$.
     
 \end{proof}
  \begin{proposition}\label{prop:central}
     Let $A$ be a unital  $\K$-algebra with  $e \in A$ such that $A=AeA$ and central $z \in Z(A)$.  Then $ze=0$ implies $z=0$.
 \end{proposition}
 \begin{proof}
 Since $z$ is central the set $\{ a \in A \mid za=0\}$ is an ideal containing $e$ and hence the whole algebra $A$. Since $A$ is unital this implies $z=0.$
  \end{proof}

\section{Preliminaries on Hecke algebras} \label{sec:AHA}

In this section we recall the definitions of finite, affine and double affine Hecke algebras, as well as a variant which we call the Schmaha\footnote{We are grateful to Kevin Walker for suggesting to us this name.}.
We establish some properties of (parabolic) sign idempotents. We also recall some well-known combinatorics of multisegments, which index simple representations of affine Hecke algebras.

\begin{notation}
In this subsection, we work over ground field $\K$ of characteristic zero containing elements $q^{\frac12}$, $t^{\frac1N}$, and $\Sq$ none of which is a root of unity, and we abbreviate $q=(q^{\frac12})^2, t=(t^{\frac1N})^N$.  
\end{notation}

\subsection{Finite Hecke algebra and idempotents}

\begin{definition} The finite Hecke algebra $\Hf_n(t)$  is the $\K$-algebra with  generators $T_1,\ldots T_{n-1}$, and relations:
\[
T_iT_{i+1}T_i = T_{i+1}T_iT_{i+1}, \textrm{ for $i=1,\ldots, n-2$},\]
\[\quad T_iT_j=T_jT_i, \textrm{ if $|i-j|\geq 2$ },\qquad
(T_i-t)(T_i+t^{-1})=0, \textrm{ for $i=1,\ldots n-1$}.
\]
\end{definition}

We denote the symmetric group by $\Sn$ and the extended affine symmetric group by $\AffSym$. 
Given a reduced word $w = s_{i_1} \cdots s_{i_m} \in \Sn$ we have a corresponding well-defined element $T_w = T_{i_1} \cdots T_{i_m} \in \Hf_n(t).$  We will write $\ellen(w) =m$ for the length of $w$. By convention $T_{\id} = 1.$ Note that the set
$\{ T_w | w \in \Sn\}$ forms a basis of $\Hf_n(t).$

\begin{notation} We denote the (unbalanced) quantum integers by 
$[k]_{r} = \frac{r^k - 1}{r - 1}$ 
and quantum factorials by
$[n]_{r}! = [n]_{r} \cdots [2]_{r} [1]_{r}$. 
\end{notation}

The finite Hecke algebra has two one-dimensional representations. Corresponding to the trivial and the sign (anti-symmetric) representation, respectively, we have the idempotents
\begin{gather*}
   \etrivn =  \frac{1}{[n]_{t^2}!} \sum_{w \in \Sn} t^{\ellen(w)} T_w, \qquad     \en = \frac{1}{[n]_{t^{-2}}!} \sum_{w \in \Sn} (-t^{-1})^{\ellen(w)} T_w.
\end{gather*}
These satisfy $(T_i -t) \etrivn = 0$ and $(T_i + t^{-1}) \en = 0$, for $1 \le i < n$. 

\begin{notation} Recall that 
 $\alpha = (\alpha_1,\ldots,\alpha_\ell)$ is  a \defterm{composition}  of $n$, denoted $\alpha \vDash n$, if $n = \alpha_1 + \cdots + \alpha_\ell$  and $\alpha_i \in \Z_{>0}$. We will say $\alpha$ has $\ell$ parts.  We will also consider ``weak" compositions where $\alpha_i \in \Z_{\ge 0}$ for which we tend to ignore the $0$ parts.  If additionally $\alpha_1 \ge \cdots \ge \alpha_\ell$ we call $\alpha $ a \defterm{partition} of $n$, denoted $\alpha \vdash n$.  
\end{notation} 
 Recall that since $t$ is not a root of  unity, the irreducible representations of $\Hfn(t)$ are indexed by $\lambda \vdash n$ and we will denote these by $\Slam$.
 The representation theory of $\Hfn(t)$ coincides with that of $\K[\Sn]$, and we will use its properties liberally, such as that the algebra is semisimple, as well as properties  about induction, restriction, and the Pieri rule. 

\begin{remark}
Compositions of $n$ are in bijection with subsets $J\subset \{1,\ldots,n-1\}$, via
\[
J\left((\alpha_1,\ldots,\alpha_\ell)\right) := \{1,\ldots,n-1\} \backslash \left\{\sum_{i=1}^j\alpha_i \,\,|\,\, j=1,\ldots \ell\right\}.
\]
\end{remark}

\begin{definition} \label{def:parabolic finite}
Given a subset $J\subset \{1,\ldots,n-1\}$ the parabolic subalgebra $\Hf_J(t)\subset\Hfn(t)$ is generated by $T_j$ for $j \in J$.
\end{definition}

We may lighten notation and just write $\Hfn$ for $\Hfn(t)$ when $t$ is understood, similarly for $\Hf_J$.

Given a composition $\alpha = (\alpha_1,\ldots,\alpha_\ell) \vDash n$, we have an algebra embedding,
\[i_\alpha:\Hf_{\alpha_1}\otimes \cdots \otimes \Hf_{\alpha_\ell}\to \Hf_{n},
\]
obtained by mapping the first factor along the evident inclusion, and mapping each $k$th factor $\Hf_{\alpha_k}$ along the evident inclusion with indices shifted by $\sum_{j<k}\alpha_j$.  The image of $i_\alpha$ is clearly $\Hf_{J(\alpha)}$.  When it is clear from context, we will use the notation $\Hf_J$ and $\Hf_\alpha$ interchangeably. 

Note $\Hf_\alpha$ corresponds to the parabolic or Young subgroup $\Sm{\alpha} \subseteq \Sn$.  In particular it has basis $\{T_w \mid w \in \Sm{\alpha} \}$.

Given $\alpha \vDash n$, we may consider the  parabolic sign idempotent $\ek{\alpha} \in \Hfn$ with  
 $\sgna{\alpha}$ the corresponding 1-dimensional sign representation of $\Hf_\alpha$. Note $(T_j + t^{-1}) \ek{\alpha} =0$ for all $j \in J(\alpha)$. 

In the special case 
a composition has the form $(\alpha 1^{r})$, 
we will lighten notation and write 
$\ek{\alpha} = \ek{(\alpha 1^{r})}$. 
More specifically, when $jm \le n$, 
 $\emj$ denotes the following parabolic sign 
 idempotent  
\begin{gather} \label{eq:eNj}
\emj = \frac{1}{([m]_{t^{-2}}!)^j}\sum_{w \in \Sm{m^j 1^{n-jm}}} (-t^{-1})^{\ellen(w)} T_w \,
\in \underbrace{\Hfm \otimes \Hfm \otimes \cdots \otimes \Hfm}_{j} \subset \Hfn 
\end{gather}
When $j=1$ we often denote this idempotent by just $\ek{m}$, and in the case $m=N, j=1$ sometimes just by  $\ee$ in later sections.

\subsection{On  parabolic sign idempotents}
In this section we describe several instances for which $\sgna{\alpha} M  \neq 0$ for a module $M$; this will be an important ingredient for ultimately proving Theorem \ref{thm:idempotent}.

Throughout this subsection, unless otherwise stated, $n,\ell, m, r$ are assumed to be non-negative integers.

Because $\Hfn $ is semisimple, the following useful lemma is a straightforward consequence of Frobenius Reciprocity. 
\begin{lemma}\label{lemma:e vs sgn}
Let $\alpha \vDash n, \, \lambda \vdash n$.
    The representation $\Sblah{\lambda}$ is a composition factor of $\Ind_{\Hf_\alpha}^{\Hf_n} \sgna{\alpha} $ iff
    $\ek{\alpha} \Sblah{\lambda} \neq 0 $. 
\end{lemma}

Recall that \emph{dominance} order on partitions $\lambda, \mu$ for which $|\mu| = |\lambda|$
is given by  $\mu \dom \lambda$
if $\sum_{i=1}^p \mu_i \le \sum_{i=1}^p \lambda_i $ for all $p$.
Note $\mu \dom \lambda$ iff $\mu^T \unrhd \lambda^T$ where $\mu^T$ is the transposed diagram (transposing corresponds to the sign involution). 

The following claims are well-known for the representation theory of the symmetric group (see for instance \cite{Sagan}), and carry over to the finite Hecke algebra. They are usually stated in terms of parabolic trivial as compared to sign idempotents, but are easily modified to our setting given tensoring with the sign representation transposes diagrams and ``reverses" dominance order. 

Let us write $\sort(\alpha)$ for the partition whose parts are those of $\alpha$ in weakly decreasing order.

\begin{claim}\label{claim:composition factors alpha}
Let $\alpha \vDash n$.
The partitions $\lambda$ of $n$ for which  $\Slam$ is a composition factor of $\Ind_{\Hf_\alpha}^{\Hfn} \sgna{\alpha}$ are precisely those   satisfying  $\lambda \dom \sort(\alpha)^T.$  In particular if $\alpha$ has $\ell$ (nonzero) parts, 
then $\lambda_1 \le \ell$.
\end{claim}

\begin{claim}\label{claim:lambdas not killed}
    Suppose  $\alpha \vDash \ell m$ has $\le \ell$ parts. For  $\lambda \vdash \ell m$, 
  if $\ek{\alpha} \Slam \neq 0$ then  $\eml \Slam \neq 0$. 
\end{claim}

This follows from checking $\sort(\alpha)^T \dom (m^\ell)^T$, i.e. $(m^\ell) \dom \sort(\alpha)$ when $\alpha$ has $\le \ell$ parts. 
We prove a stronger but more technical form of this claim below.

 \begin{remark} \label{rem:idempotent killing}
Notice that $\eml \Slam \neq 0$ implies $\ek{m^j} \Slam \neq 0$ for any $j \le \ell$. 
Further $\eN \Slam \neq 0 $ implies $\ek{m} \Slam \neq 0$ for any $m \le N$.  This follows from $\eml$ belonging to the ideal generated by $\emj$, and likewise for $\eN$ and $\ek{m}$, when $ j \le \ell$ and $m \le N$.
\end{remark}

\begin{proposition}
\label{prop:kN+r}
Write $n=\ell m+r = \ell m + r_1 + r_2$ with $0 \le r < m$ and $0 \le r_1, r_2$. Let $\lambda \vdash n$, $\beta \vDash r_2, \alpha \vDash \ell m+r_1$ and suppose $\alpha $ has $\le \ell$ parts. 
    Then $\ek{(\alpha, \beta)} \Slam \neq 0$ implies that  $\ek{(m^{\ell}, 1^{r})} \Slam \neq 0$.  Further $\ek{\alpha} \Slam \neq 0$ implies that  $\ek{m^{\ell}} \Slam \neq 0$, and
consequently $\emj \Slam \neq 0$ for all $0\le j \le \ell$.
        
\end{proposition}
\begin{proof}
   
    Suppose $\ek{(\alpha, \beta)} \Slam \neq 0$.  By Lemma \ref{lemma:e vs sgn} and Claim \ref{claim:composition factors alpha} we have $\lambda \dom \sort(\alpha, \beta)^T$.  Thus also by Claim \ref{claim:composition factors alpha} it suffices to show $\sort(\alpha, \beta)^T \dom (m^\ell, 1^r)^T$, or equivalently that $(m^\ell, 1^r) \dom \sort(\alpha, \beta)$. 

    Write $\mu = \sort(\alpha)$. Since $\mu$ has $\le \ell$ parts but $|\mu| \ge \ell m$, we have $m \le \mu_1$.  Further $\ell m \le \ell m + r_1 = |\mu|$. 
    If $(m^\ell, 1^{r_1}) \centernot\dom \mu$  
    then there exists $1 < p < \ell$ with $(p-1)m \le \mu_1 + \cdots + \mu_{p-1}$ but $pm > \mu_1 + \cdots + \mu_{p-1} + \mu_p$.  Thus $\mu_p = m-s$ for some $0 < s \le m$. Then $\ell m + r_1 = (\mu_1 + \cdots + \mu_p) + (\mu_{p+1} + \cdots + \mu_{\ell}) <  mp + (\ell - p)(m-s) = \ell m -s(\ell -p) < \ell m$, which is a contradiction.  Thus $(m^\ell, 1^{r_1}) \dom \alpha$. 

    Next for $1 \le p \le \ell$ clearly $\sort(\alpha, \beta)_p \ge \mu_p$, and since the nonzero parts of $\beta$ are $\ge 1$, we have $\sort(\alpha, \beta) \unrhd (\mu, 1^{r_2}) = \sort(\alpha, 1^{r_2}) \unrhd (m^{\ell}, 1^r)$. The final statement of the proposition applies by taking $\beta = 1^{r_2}$ as $\ek{\alpha} = \ek{(\alpha, 1^{r_2})}$. 
\end{proof}

\subsection{Affine Hecke algebra and its parabolic subalgebras}
\begin{definition}
The extended affine Hecke algebra $\H_n(t)$ is generated by the algebras $\K[Y_1^{\pm 1},\ldots Y_n^{\pm 1}]$ and $\Hf_n(t)$, with relations: 
\[
T_iY_iT_i=Y_{i+1}, \textrm{for $i=1,\ldots, n-1$},\quad 
T_i Y_j = Y_j T_i, \textrm{ for $j \neq i, i+1$}
\]
\end{definition}

To lighten notation, we will write the parameters $t$ and $n$ only when required, and otherwise abbreviate $\H_n(t)$ by either $\H_n$ or simply $\H$. Similarly for $\Hf$.
 We denote by $\Y$ the $\K$-subalgebra generated by $Y_1^{\pm 1},\ldots Y_n^{\pm 1}$. 
We shall also sometimes denote $\H_n(t)$ by $\HY$ or refer to it as AHA.

Note that the center $Z(\HY) = \K[Y_1^{\pm 1},\ldots Y_n^{\pm 1}]^{\Sn} $ consists of 
symmetric Laurent polynomials. 

\begin{remark} \label{rem:Y inverse}  A common alternative presentation  imposes instead the relations $T_i Y_i^{-1} T_i = Y_{i+1}^{-1}$.  An isomorphism between the presentations may be given by inverting $Y_i$.
\end{remark}

\begin{definition} \label{def:parabolic Hecke}
Given a subset $J\subset \{1,\ldots,n-1\}$ the parabolic subalgebra $\H_J\subset\H$ is generated by $T_j$ for $j \in J$ and by $\Y$.
\end{definition}

Similar to the finite Hecke algebra, for the extended affine Hecke algebra we will use the notation $\H_J$ and $\H_\alpha$ interchangeably when $J = J(\alpha)$ corresponds to the composition $\alpha \vDash n$.

\subsection{Multisegments and the category of $\Y$-finite $\H$-modules}
In this section we review the construction of simple AHA modules from the data of multisegments.

As $Z(\H) = \Y^{\Sn}$, all simple $\H$-modules are finite-dimensional.  For the ease of understanding and parameterizing simple modules, we desire the added assumption that they are absolutely irreducible or equivalently that $\K$ is a splitting field.  For this to hold, it suffices that the eigenvalues of $Y_i$ lie in $\K$.   To this end, we redefine the following notion of $\Y$-finite to suit our purposes. 
\begin{definition}
 \label{def:Yfinite}
 A module $M$ for $\H$ is called \defterm{$\Y$-finite} if for every $v\in M$, $\Y v$ is a finite-dimensional subspace and further the eigenvalues of each $Y_i$ lie in $\K$.
\end{definition}

The category of $\Y$-finite modules is controlled combinatorially by so-called multisegment modules: every simple $\Y$-finite module arises as a simple quotient $L(\Delta)$ of a standard multisegment module $M(\Delta)$, the latter being obtained by inducing from a sign representation for some parabolic subalgebra.  Let us recall all of these constructions, which will play an important role in our proofs.

\begin{definition}
A pair of an element 
$a \in \K^\times$ and a non-negative integer $\ell$ is called a \defterm{segment} $\seg{a}{\ell}$, and $a$ is called its \defterm{start}.  
We will refer to a tuple $\Delta = (\seg{a_1}{\ell_1}, \dots, \seg{a_m}{\ell_m})$ of segments as a \defterm{multisegment}, and write $|\Delta|= \ell_1 + \cdots + \ell_m$.
\end{definition}
\begin{remark}
We caution that the word ``multisegment" typically refers only to a multiset of segments, and not an ordering on its elements.  For our purposes we will require the additional data of the ordering of segments; we abuse terminology and nevertheless call such an ordered tuple of segments a multisegment.
\end{remark}

\begin{definition}
Given a multisegment $\Delta$ and $a\in\K^\times$, the \defterm{area} of $a$ in $\Delta$ is the sum $\sum_{i}\ell_i$ over all segments $\seg{a}{\ell_i}$ appearing in $\Delta$ with start $a$.
\end{definition}

\begin{definition}
Given a multisegment $\Delta = (\seg{a_1}{\ell_1},\cdots,\seg{a_m}{\ell_m})$ of size $n=|\Delta|$, the associated composition $\alpha(\Delta)$ of $n$ is defined by $\alpha(\Delta)= (\ell_1,\ldots,\ell_m)$.
\end{definition}

In this way a multisegment $\Delta$ also determines a subset $J(\alpha(\Delta))\subset \{1,\ldots,|\Delta|-1\}$.  It is sometimes useful to allow some $\ell_j=0$ as placeholders, although such empty segments will ultimately be ignored. 

Given a multisegment $\Delta$, we will abbreviate $\H_\Delta = \H_{\alpha(\Delta)}= \H_{J(\Delta)} $.  

\begin{example}
   For $\Delta = (\seg{a}{2}, \seg{b}{3})$ we have 
   $ \H_{\Delta} = \H_{(2,3)}= \H_{\{1,3,4\}}   \subset \H_{5}$, for any $a,b\in \K^\times$. 
\end{example}

\begin{definition} Given two multisegments $\Delta$
 and $\Gamma$
 we denote their concatenation by
 $\Delta \sqcup \Gamma$.\end{definition}

Note that $\alpha(\Delta \sqcup \Gamma)$ is the concatenation of the compositions
$\alpha(\Delta)$ and $\alpha(\Gamma)$ and that  $|\Delta \sqcup \Gamma| = |\Delta| + |\Gamma|$.  We allow $\Delta$ or $\Gamma$ to be the empty multisegment.

\begin{definition} An  multisegment $\Delta = (\seg{a_1}{\ell_1},\cdots,\seg{a_m}{\ell_m})$ is \defterm{right-ordered} if
\begin{enumerate}
    \item whenever $a_j/a_i = t^{2z}$ for a non-negative integer $z$, then $i\leq j$, and
    \item whenever $a_i=a_j$ and $i<j$, then $\ell_i\leq \ell_j$.
\end{enumerate}
\end{definition}

\begin{definition} Let $a,b, r\in\K^\times$.  We say that $a$ and $b$ are \defterm{in the same $r$-line} if $a/b=r^{z}$ for some $z\in\mathbb{Z}$, and otherwise that they are \defterm{in distinct $r$-lines}.
\end{definition}

\begin{remark} The definition of right-ordered above is slightly different than the one found in \cite{V-multisegments} for two reasons. First, here we allow segments with starts on distinct $t^2$-lines, and so take  a more flexible definition to account for that. Second, the representation-theoretic significance of a segment is  attached to the sign representation here, whereas in \cite{V-multisegments} it was attached to the trivial representation.\end{remark}

 \begin{definition}
Two multisegments $\Delta$ and $\Gamma$ are \defterm{equivalent}, denoted $\Delta \sim \Gamma$, if:
\[
\Delta=(\seg{a_1}{\ell_1},\ldots ,\seg{a_m}{\ell_m}), \textrm{ and }
\Gamma = (\seg{a_{\sigma(1)}}{\ell_{\sigma(1)}},\ldots ,\seg{a_{\sigma(m)}}{\ell_{\sigma(m)}})\]
for some $\sigma\in \Sm{m}$ such that its only inversions occur among distinct $t^2$-lines, i.e. if $i<j$ and $\sigma(i)>\sigma(j)$, then $a_i/a_j\not\in \{t^{2z}\,\,|\,\,z\in\mathbb{Z}\}$.
\end{definition}

Note that if $\Delta$ is right-ordered and $\Delta \sim \Gamma$ then $\Gamma$ is also right-ordered.

\begin{definition}\label{def:narrow}
Let $s\in \mathbb{Q}_{>0}$.  Write $s=\frac{n_0}{N_0}$ with $n_0, N_0>0$, and $\gcd(n_0,N_0)=1$.  A multisegment $\Delta = (\seg{a_1}{\ell_1},\ldots \seg{a_m}{\ell_m})$ is \defterm{$s$-narrow} if whenever $a_i/a_j=t^{\frac{2z}{N_0}}$, with $z\in\mathbb{Z}$, we have $|z|<n_0$.\end{definition}

We note that if $s \in \Z$ then $N_0=1$.  As an important special case, $\Delta$ is $1$-narrow if any two starts of $\Delta$ are either equal or lie on distinct $t^2$-lines.

\begin{def/prop}
Given a multisegment $\Delta = (\seg{a_1}{\ell_1},\ldots \seg{a_m}{\ell_m})$  there exists a unique one-dimensional representation $\r(\Delta)$ of 
$\H_{\Delta}$
such that $Y_i$ acts by 
the $i$th entry of
$$(a_1, a_1 t^{-2},  \ldots, a_1 t^{2(1-\ell_1)}, a_2, a_2 t^{-2},  \ldots, a_2 t^{2(1-\ell_2)}, \ldots,
a_m, a_m t^{-2},  a_m t^{2(1-\ell_m)}
)
$$
and $T_j$ acts by $-1/t$ for $j \in J(\Delta)$.
\end{def/prop}

Note if $m=1$, so that $\Delta$ consists of a single segment, then all the eigenvalues of $Y_i$ on $\r(\Delta)$ lie on the same $t^2$-line.

\begin{definition}
Given a multisegment $\Delta = (\seg{a_1}{\ell_1},\ldots \seg{a_m}{\ell_m})$ with $|\Delta|  = n $, let $M(\Delta)$ denote the induced module,
\[
M(\Delta) = \Ind_{\H_{\Delta}}^{\H} \r(\Delta) \,  :=  \, \H \rt{\H_{\Delta}} \r(\Delta).
\]
\end{definition}

\begin{notation}\label{not:vdelta} Let $M'$ be an $\H$-module.
Any non-zero element $f\in\Hom_{\H}(M(\Delta),M')$ produces a vector $v_\Delta\in M'$ such that $\r(\Delta)=\K v_\Delta.$  By abuse of notation, we call any such non-zero vector $v_\Delta$, even though it is only defined up to a scalar, and furthermore we suppress the dependence on $f$ from the notation.
\end{notation}

\begin{remark}
To avoid confusion, we note that \cite{V-multisegments} constructs $M(\Delta)$ by inducing from the \emph{trivial} representation of $\Hf_J$, whereas we use the \emph{sign} representation.  Hence segments in \cite{V-multisegments} feature increasing powers of $t^2$ whereas in our arguments they are decreasing.  All the proofs however go through in either setting, see \cite[Section 3, 9]{V-multisegments}. 
\end{remark}

\begin{definition}
Let $\Rep_r(\H_n(t))$ be the category of 
$\Y$-finite $\H_n(t)$-modules such that all eigenvalues of the
$Y_i$ lie in $\{r t^{2z} \mid z \in \Z \}$; in other words they are all on the same $r$-line.  Let $\Rep_r$ be the direct sum over all $n \ge 0$ of $\Rep_r(\H_n(t))$.
\end{definition}

\begin{theorem}[\cite{BZ77, Zelevinsky}] \label{thm:RepaRepb}  Suppose $L$  is a simple 
module in $\Rep_a(\H_l)$ and
$M$ is a simple 
module in $\Rep_b(\H_m)$. If $a$ and $b$ lie on different $t^2$-lines, then $\Ind^{\H_{l+m}}_{\H_{(l,m)}}(L\boxtimes M)$ is simple, and moreover isomorphic to 
$\Ind^{\H_{l+m}}_{\H_{(m,l)}}(M\boxtimes L)$.
\end{theorem}

\begin{theorem}\label{thm:equalstarts} Suppose that all starts of $\Delta$ are equal.  Then $M(\Delta)$ is simple, and further we have an isomorphism $M(\Delta) \cong M(\Delta')$ for any permutation $\Delta'$ of $\Delta$.
\end{theorem}

Combining Theorems \ref{thm:RepaRepb} and \ref{thm:equalstarts} gives the following

\begin{corollary}\label{cor:distinctstarts}
Let $\Delta$ be $1$-narrow.  Then $M(\Delta)$ is simple, and $M(\Delta)\cong M(\Delta')$ for any permutation $\Delta'$ of $\Delta$.
\end{corollary}
This corollary follows easily from  Theorems \ref{thm:RepaRepb} and \ref{thm:equalstarts}, recalling that $1$-narrow means that any two starts of $\Delta$ are either equal or lie on distinct $t^2$-lines.

For the following, we assume $\K$ is a splitting field. 

\begin{theorem}[\cite{BZ77, Zelevinsky}]
\label{thm:L=usq}
Simple $\H$-modules are parameterized by right-ordered 
multisegments.  More precisely:
\begin{enumerate}
\item Suppose that $\Delta' \sim \Delta$.
    Then we have $M(\Delta)\cong M(\Delta')$.
    \item  Suppose $\Delta$ is right-ordered. Then the module $M(\Delta)$ has a unique simple quotient which we denote $L(\Delta)$.
    \item Suppose $\Delta$ is right-ordered and that $\Delta' \sim \Delta$.
    Then we have 
    $L(\Delta)\cong L(\Delta')$.
\end{enumerate}
Moreover, given any simple module $L$ of $\H$, there exists a right-ordered multisegment $\Delta$ and an isomorphism $L\cong L(\Delta)$.
\end{theorem}

\begin{remark}\label{remark:Z}
$\SLGL$
If we replace $\H$ with a quotient that imposes a relation of the form $(Y_1 \cdots Y_n)^d = \Zprod$ for
some constant $\Zprod$, the above theorems and constructions still hold. It merely imposes a restriction on the multisegments we consider.  Namely, the constant by which $Y_1 Y_2\cdots Y_n$ acts on $\r(\Delta)$ by must be consistent with the relation. (See Remark \ref{remark:SLGL}.)
\end{remark}

\subsection{DAHA--Schmaha}\label{sec:DAHA defn}
\begin{definition} \label{def:HG}
The $\GLn$ double affine Hecke algebra $\HG$ is
the $\K$-algebra presented by generators:
$$T_0,T_1,\ldots T_{n-1}, \pi^{\pm 1}, Y_1^{\pm 1},\ldots, Y_n^{\pm 1},$$
subject to relations\footnote{As with the affine symmetric group  $\AffSym$, we drop the relations on the 
second line when $n=2$.}:
\begin{align} 
&(T_i-t)(T_i+t^{-1})=0 \quad (i=0,\ldots, n-1),& && \label{HeckeReln}\\
&T_iT_jT_i = T_jT_iT_j\quad (j\equiv i\pm 1 \bmod n),& &T_iT_j = T_jT_i \quad (\textrm{otherwise}),&\label{BraidReln}\\
&\pi T_i\pi^{-1} = T_{i+1} \quad (i=0,\ldots, n-2),& &\pi T_{n-1}\pi^{-1}=T_0,&\nonumber\\
&T_iY_iT_i=Y_{i+1} \quad (i=1,\ldots, n-1),& &T_0Y_nT_0 = q^{-1}Y_1&\nonumber\\
&T_i Y_j = Y_j T_i \quad  (j \not\equiv i, i+1 \bmod n),& &&\nonumber\\
&\pi Y_i\pi^{-1} = Y_{i+1} \quad (i=1,\ldots, n-1),& &\pi Y_{n}\pi^{-1}=
 q^{-1}Y_1&\nonumber
\end{align}
\end{definition}

We recall that $\HG$ has basis $\{  T_w Y^\beta \mid
w \in \AffSym, \beta \in \Z^n\}$ where we identify $\pi$ with $T_\pi$. 

An alternate presentation of $\HG$ drops $T_0$ and $\pi$ but includes generators $X_i^{\pm 1}, 1 \le i \le n$, which are related to the generators above via $X_1 = \pi T_{n-1}^{-1} \cdots T_1^{-1}$ and $X_{i+1} = T_i X_i T_i$.  
Note $X_1 X_2 \cdots X_n = \pi^n$, and  this element $q$-commutes with each $Y_i$.
The $X_i$ generate a Laurent polynomial subalgebra, hence
similar to above, we may write $\cX = \K[X_1^{\pm 1}, \cdots, X_n^{\pm 1}]$.

\begin{notation} \label{not:AHA-GL}
We denote by $\HY$ and $\HX$, respectively, the subalgebras of $\HG$ generated by $\Hf$ and $Y_i^{\pm}$'s (resp, $X_i^{\pm}$'s). Note $\HX$ is also generated by $\Hf$ and $\pi^{\pm 1}$.  Each subalgebra identifies with the AHA  $\H$ as an abstract algebra.
\end{notation} 
\begin{remark} \label{remark:SLGL} Throughout most of the paper, statements and proofs will be made using $G=\GLN$ notation. We will warn the reader in cases where statements or proofs need modification to be correct for $\SLN$, with the symbol $\SLGL$. 
\end{remark}

 \begin{definition}\label{def:schmaha}
The $\SL$ Schmaha $\Schmaha{n}$ is the $\K$-algebra presented by generators:

$$T_0,T_1,\ldots T_{n-1}, \Spi^{\pm 1}, Z_1^{\pm 1},\ldots, Z_n^{\pm 1},$$
subject to relations:
\begin{align*} 
&(T_i-t)(T_i+t^{-1})=0 \quad (i=0,\ldots, n-1),& 
&\Spi^n \text{ is central} ,
&\\
&T_iT_jT_i = T_jT_iT_j\quad (j\equiv i\pm 1 \bmod n),& &T_iT_j = T_jT_i \quad (\textrm{otherwise}),&\\
&\Spi T_i\Spi^{-1} = T_{i+1} \quad (i=0,\ldots, n-2),& &\Spi T_{n-1}\Spi^{-1}=T_0,&\\
&T_iZ_iT_i=Z_{i+1} \quad (i=1,\ldots, n-1),&
&T_i Z_j = Z_j T_i \quad  (j \not\equiv i, i+1 \bmod n),&
\\
&T_0Z_nT_0 = \Sq^{2n}Z_1,&
&Z_1 Z_2 \cdots Z_n   \text{ is central}, &
& &\\
&\Spi Z_i\Spi^{-1} = \Sq^{-2} Z_{i+1} \quad (i=1,\ldots, n-1),& &\Spi Z_{n}\Spi^{-1}=
 \Sq^{2n-2}Z_1.&
\end{align*}
\end{definition}

Let us fix a constant $\Zprod \in\K^\times$.

\begin{definition} \label{def:HS}
The $\SLn$ double affine Hecke algebra $\HS$ is the quotient of $\Schmaha{n}$ by the 
relations $Z_1 Z_2 \cdots Z_n = \Zprod$ and $\Spi^n = 1$.

Notice that the inner automorphism given by conjugation by $\Spi$ has order $n$ for both $\HS$ and $\Schmaha{n}$.

\end{definition}

$\Schmaha{n}$ has analogous subalgebras $\HY$ and $\HX$, which are isomorphic to their $\GL$ counterparts.
$\HS$ has  subalgebras $\HY$ and $\HX$, which by abuse of notation we also refer to by these same names, even though they are each quotients of the affine Hecke algebra (AHA). 
This abuse of notation allows us to make uniform statements across the cases.
All the algebras $\HS$, $\Schmaha{n}$,  or $\HG$ may be referred to as a DAHA.

\section{Zeroth Hochschild homology of Skein algebras}\label{sec:HH skein algebras}

In this section we compute the zero${}^{\mathrm{th}}$ Hochschild homology of the $\GLN$ and $\SLN$ skein algebras.
Our starting point 
is the isomorphism \eqref{eq:Sk = KLamW} between $\SkAlg_G(T^2)$ and the algebra $\K_\omega[\Lambda \oplus \Lambda]^{\SN}$ defined below, as well as the following:

\begin{proposition}\label{prop:morita smash}
We have a Morita equivalence between $\K_\omega[\Lambda \oplus \Lambda]^{\SN}$ and $\K_\omega[\Lambda \oplus \Lambda]\#{\SN}$.

\end{proposition}
\begin{proof}
    By \cite[Theorem 2.3]{Montgomery1980}, the algebra $S:= \K_\omega[\Lambda \oplus \Lambda]\#{\SN}$ is simple. It follows that $S\etriv{N} S=S$, where $\etriv{N} \in \K[\SN]$ is the trivial idempotent,  and thus we have the Morita equivalence by Proposition \ref{prop:Morita}.
\end{proof}

Because $\HHz$ is Morita invariant, it remains only to compute the Hochschild homology of $\SmashD$. This is the focus of the remainder of the section.

\begin{notation}
Let $\Lambda_{\GLN}$ and $\Lambda_{\SLN}$ denote the weight lattices of $\GLN$ and $\SLN$, respectively. For ease of notation below, we identify $\Lambda_{\GLN}$ with $\Z^N$ and its elements are expressed relative to the  orthonormal basis $\{ e_1, \ldots, e_N \}$. For $a\in \Lambda_{\GLN}$ we define its degree $\deg(a)\in\mathbb{Z}$ to be the sum of its entries. 
 We identify $\Lambda_{\SLN}$ with $\Z^N/ \Z \sum_i e_i$.
For $a\in\Lambda_{\SLN}$ $\deg(a)$ is the sum of its entries taken modulo $N$, so that $\deg(a)\in\mathbb{Z}/N\mathbb{Z}$.
The group $\SN$ acts on $\Lambda_{G}$, preserving degree.

Note we may also view $\deg(a)$ as living in $\Lambda_G/Q$, where $Q$ is the root lattice.

Let $\Lambda=\Lambda_{\GLN}$ or $\Lambda_{\SLN}$. 
 For $u,v\in \Lambda$ we denote by $(u,v)$ the Cartan pairing, which is a symmetric bilinear form, and we denote by $\omega$ the canonical symplectic pairing defined on $\Lambda\oplus\Lambda$,
\[
\omega(u\oplus v,x\oplus y) =  (u,y) - (x,v).
\]
The quantum torus $\K_\omega[\Lambda\oplus\Lambda]$ is identified as a $\K$-vector space with the group algebra of the lattice $\Lambda \oplus \Lambda$, but with a multiplication which is twisted by $\omega$.  It inherits the grading by degree from $\Lambda\oplus\Lambda$.  Denoting by $\cX^a$ the basis element labeled by $a\in\Lambda\oplus \Lambda$, the multiplication is given by:
\[
\cX^a \cX^b = q^{-\omega(a,b)/2} \cX^{a+b}.
\]

The  group $\SN$ acts on $\Lambda \oplus \Lambda$ via the diagonal action on $\Lambda$.  This induces an action by algebra automorphisms on $\K_\omega[\Lambda \oplus \Lambda]$ precisely because $\SN$ respects the Cartan pairing on $\Lambda$ and hence the symplectic pairing on $\Lambda\oplus \Lambda$.  We denote by $\K_\omega[\Lambda\oplus\Lambda]^{\SN}$ the subalgebra of invariants for the $\SN$-action, and by $\K_\omega[\Lambda\oplus\Lambda]\#{\SN}$ the smash product algebra.

\end{notation}

Given a group $W$ and $\sigma \in W$, we will denote the conjugacy class of $\sigma$ by $[\sigma]$ and its centralizer by $Z_W(\sigma)$.
It follows from e.g. \cite{Shepler-Witherspoon},\cite{Stefan} (see also {\cite[Theorem 2.4]{Kaygun2021}}, {\cite[Theorem 3.3]{kinnear2023skein}}) that for any algebra $A$ with an action of 
$W$ by automorphisms, the Hochschild homology of the smash product $A\# W$ may be computed as a sum over 
$[\sigma]$ of the $\stab W \sigma$-coinvariants in the Hochschild homology of $A$ with coefficients in the $\sigma$-twisted bimodule $A_\sigma$, on which $A$ acts on the left in the usual way, and on the right twisted by $\sigma$.  In the case at hand, this yields:
\begin{gather}
    \label{eq:HH0 decomposition}
\HHz(\SmashD) \cong \bigoplus_{[\sigma]\in \Conj{\SN}} \HHz(\K_\omega[\Lambda \oplus \Lambda],\K_\omega[\Lambda \oplus \Lambda]_\sigma)_{\stab{\SN} \sigma}. 
\end{gather}
Here $\K_\omega[\Lambda \oplus \Lambda]_\sigma$ denotes $\K_\omega[\Lambda \oplus \Lambda]$ with bimodule action twisted on the right by $\sigma$; the sum is taken over the conjugacy classes in $\SN$, denoted $\Conj{\SN}$ above; and the subscript $\stab{\SN} \sigma $
denotes that we take the coinvariants for the centralizer of $\sigma$ in $\SN$.  More explicitly:
\[
\HHz(\K_\omega[\Lambda \oplus \Lambda],\K_\omega[\Lambda \oplus \Lambda]_\sigma) = \K_\omega[\Lambda \oplus \Lambda] \Bigg/ \left\{ \cX^a \cX^b - \cX^b \cX^{\sigma(a)} \, , \,\, \textrm{ for } a,b \in \Lambda\oplus\Lambda\right\}.\\
\]

Let $_\mathbb{Q}\Lambda = \mathbb{Q}\otimes_\ZZ \Lambda$. 
Given $\sigma\in\SN$,  when we consider it as a  linear map  on $_\mathbb{Q}\Lambda$ (or diagonally on $_\mathbb{Q}\Lambda \oplus {}_\mathbb{Q}\Lambda$),  we will write $\sigma_{\Q}$.
Otherwise we assume its domain is $\Lambda$ (resp. $\Lambda \oplus \Lambda$).
Let  ${}_{\Q} \Uw = \ker(1-\sigma_{\Q}) $ and let ${}_{\Q}\Uwp$ denote its orthogonal complement with respect to the Cartan pairing. 
Write 
$$\Uw = {}_{\Q} \Uw \cap \Lambda \, \text{ and  } \,  \Uwp = {}_{\Q} \Uwp \cap \Lambda.$$
Observe  $\Uw = \Lambda^{\sigma}= \ker(1-\sigma) |_{\Lambda}$ and $ {}_{\Q} \Uwp =\Im(1 - \sigma)$, but we only have $\Im(1-\sigma)|_\Lambda \subseteq \Uwp$ in general.  More precisely, we will compute its index below and show the quotient is a cyclic group.
\begin{remark}In the statement and proof of the following proposition we encounter expressions such as $\omega((1-\sigma_\Q)^{-1}\alpha,\beta)$ with $\alpha, \beta\in \Uwp\oplus \Uwp$.  These are well-defined because $(1-\sigma_\Q)$ surjects onto  $\Uwp \oplus \Uwp$, and because $\ker(1-\sigma_\Q)$ coincides with ${}_\Q\Uw \oplus {}_\Q \Uw$.\end{remark}

\begin{proposition}\label{prop:UU is HH0}
Given $\sigma\in\SN$, we let $\K[(\Uwp\oplus\Uwp)/\Im(1-\sigma)]$ denote the vector space formally spanned by the set of $\Im(1-\sigma)$-cosets in $\Uwp\oplus \Uwp$.
Then the linear homomorphism
\begin{align*}
\psi: \K[(\Uwp\oplus \Uwp)/\Im(1-\sigma)] &\to\HHz(\K_\omega[\Lambda\oplus\Lambda],\K_\omega[\Lambda\oplus\Lambda]_\sigma)\\
 a &\mapsto m_a = q^{-\frac12\omega( (1-\sigma_\Q)^{-1}a,a)} \cX^a
\end{align*}
defines an isomorphism. In the $\SLN$ case, it is compatible with the natural $\Z/N\Z\times\Z/N\Z$-grading.
\end{proposition}

\begin{proof}
First we show that $\psi$ is surjective.  Suppose we have $a\in\Lambda\oplus\Lambda$ which does not lie in $\Uwp \oplus \Uwp$.  Then there exists some $u\in\Uw$ such that either $b=u \oplus 0$ or $b=0 \oplus u$ satisfies $\omega(a,b)\neq 0$.
In particular $(1-q^{-\omega(a,b)}) \neq 0$.
A straightforward computation then yields
\[
[ \cX^b, \cX^{-b}  \cX^{a}]_\sigma = (1-q^{-\omega(a,b)}) \cX^a.
\]
This implies that $\psi$ is surjective.

To establish that $\psi$ is injective, we must determine the relations amongst $\cX^a$ for $a\in \Uwp \oplus \Uwp$.  We begin by remarking that the commutator relations $[\K_\omega[\Lambda \oplus \Lambda], \K_\omega[\Lambda \oplus \Lambda]]_\sigma$ are naturally graded by the quotient $\Lambda\oplus\Lambda/(\Im(1-\sigma))$.  Indeed, we may rewrite the relations:

\[
[\cX^a,\cX^b]_\sigma = q^{-\omega(a,b)/2}\cX^{a+b} - q^{-\omega(b,\sigma(a))/2}\cX^{a+b - (1-\sigma)(a)}
,\]
which are clearly homogeneous modulo $\Im(1-\sigma)$. 

So let us now consider a pair $a, b$ lying in the same   $\Im(1-\sigma)$-coset, and let $c$ be the unique $(1-\sigma)^{-1}$-preimage of $(a-b)$ in $\Uwp\oplus \Uwp$.  We compute the $\sigma$-commutator,
\[
[\cX^c, \cX^{a - c}]_{\sigma}
=q^{\frac12 \omega((1-\sigma_\Q)^{-1}b,a)} \cdot \left(q^{-\frac12\omega( (1-\sigma_\Q)^{-1}a,a)} \cX^a - q^{-\frac12 \omega((1-\sigma_\Q)^{-1}b,b)} \cX^b\right).
\]

Hence if we renormalize by letting $m_a = q^{-\frac12\omega( (1-\sigma_\Q)^{-1}a,a)} \cX^a$ when $a \in \Uwp \oplus \Uwp$, then it follows that the $\sigma$-commutator relations simply amount to identifying any two  elements $m_a, m_b$ with $a, b$ lying in the same coset of $\Im(1-\sigma)$.  We obtain a basis consisting of the generators $m_a$, for $a$ ranging over distinct $\Im(1-\sigma)$-coset representative in $\Uwp \oplus\Uwp$.

\end{proof}

\begin{proposition}\label{prop:lose Z}
For any $\sigma\in\SN$
\begin{align*}
\HHz(\K_\omega[\Lambda\oplus\Lambda],\K_\omega[\Lambda\oplus\Lambda]_\sigma)  \cong
\HHz(\K_\omega[\Lambda\oplus\Lambda],\K_\omega[\Lambda\oplus\Lambda]_\sigma)_{Z_{\SN}(\sigma)}
\end{align*}
\end{proposition}

\begin{proof}
We use the same notation as in Proposition \ref{prop:UU is HH0} and its proof.

Observe that if $\tau \in \stab{\SN} \sigma$ then $\tau$ preserves $\Uwp$, $a$ and $\tau a$ lie in the same $\Im(1-\sigma)$-coset, and further that
\[ \omega(\tau a,(1-\sigma_\Q)^{-1}\tau a) = \omega(\tau a,\tau (1-\sigma_\Q)^{-1}a) = \omega(a,(1-\sigma_\Q)^{-1}a),\]
and so taking $\tau$-coinvariants imposes the relation $m_{\tau a} = m_a$, which is already a relation in $\HHz$.  Hence taking $\stab{\SN} \sigma$-coinvariants does not change the space. 
\end{proof}
\begin{lemma}\label{lem:coset-size}
For $\Lambda=\Lambda_{\GLN}$ we have $\Im(1-\sigma) = \Uwp$.  
For $\Lambda = \Lambda_{\SLN}$, the index of $\Im(1-\sigma)$ in $\Uwp$ is $\dl := \gcd(\lambda_1,\ldots,\lambda_k)$, where $\lambda \vdash N$ denotes the cycle type of $\sigma$.  
In particular, $\Uwp/\Im(1-\sigma)$ is a cyclic group of order $\dl$
with cyclic generator of degree $N/\dl$.
\end{lemma}

\begin{proof}
We note that $\Uwp/ \Im(1-\sigma)$ coincides precisely with torsion subgroup of $\Lambda / \Im(1-\sigma)$.  Hence, the result follows from determination of the Smith normal form associated to the generating set $\{ e_j - e_{\sigma(j)} \mid 1 \le j \le N\}$ of $\Im(1-\sigma)$.  This is with respect  to an orthonormal basis $\{e_i \mid 1 \le i \le N \}$  of $\Lambda_{\GLN}$, where we have identified $\Lambda_{\SLN} $ with $\Lambda_{\GLN}/ \Z \sum_{i=1}^N e_i.$ 
For $\GLat $ we see there is no torsion. We omit this elementary computation for $\SLat$.

A cyclic generator of $\Uwp/\Im(1-\sigma)$ is any vector consisting of $N/\dl$ ones and $N-N/\dl$ zeros, 
distributed so there are $\lambda_i/\dl$ ones
in each $\sigma$-orbit of $\{1,2 \cdots  N\}$. 
In particular, it has degree $N/\dl \in \mathbb{Z}/N\mathbb{Z}$.

\end{proof}

\begin{remark}
We note that when $\lambda = (N)$ has one part, $\Im(1-\sigma)|_\Lambda \subseteq \Lambda$ is precisely the root lattice.  In the case of ${\GLN}$, $\Lambda_{\GLN}/\Im(1-\sigma) = \Z^N/\Im(1-\sigma)$ has no torsion (but is free of rank $1$).  In the case  of ${\SLN}$, $\Uwp$ is the weight lattice $\Lambda_{\SLN}$ for which we know the root lattice has index $N$.
\end{remark}

\begin{lemma}\label{lemma:totient}
Fix positive integers $k$ and $N$. Then
$$\sum_{\lambda \vdash N} \dl^k =  (\cP \star J_k) (N)$$
\end{lemma}
\begin{proof}
    The left hand side counts the number of elements in a union of products of cyclic groups
    $$\underbrace{\Z/\dl \Z \oplus \Z/\dl \Z \oplus \cdots \oplus \Z/\dl \Z }_{k}$$
    as $\lambda$ varies.
    We can count these elements by their order. The $k$th totient function $J_k(d)$ precisely counts how many elements of order $d$ there are in $(\Z/g\Z)^{\oplus k}$ when $d \mid g$. 
    On the other hand $\Z/\dl \Z$ has an element of order $d$ precisely if $d \mid \dl$, and in this case it means the partition $\lambda = d \mu$ where $\mu \vdash N/d$.
    Hence 
    \begin{align*}
        \sum_{\lambda \vdash N} \dl^k
        &= \sum_{d \mid N} \sum_{\lambda \vdash N} \# \text{elements of order $d$ in } (\Z/\dl\Z)^{\oplus k}
        \\
    &=
    \sum_{d \mid N} \sum_{\mu \vdash N/d} J_k(d) \\
    &= \sum_{d \mid N} \cP( N/d) J_k(d)
    = (\cP \star J_k) (N).
    \end{align*}
    We note the above sum may also be written $= \sum_{v \in (\Z/N\Z)^{\oplus k} } 
    \cP(\gcd(v_1, \dots , v_k, N))$
\end{proof}

We can now prove the main theorem of this section.  We note that the $\GLN$ case of the following theorem may be read off from \cite[Theorem 5.1]{Etingof-Oblomkov}, however the formula in the $\SLN$ case appears to be new. 

\begin{theorem} \label{thm:HHz KLam formula}
  For $\Lambda = \Lambda_{\GLN} $ we have   $\dim \HHz(\SmashD) = \cP(N)$.
  
For $\Lambda = \Lambda_{\SLN} $ we have $\dim \HHz(\SmashD) = (\cP \star J_2)(N) = \sum_{\lambda \vdash N} \gcd(\lambda)^2.
$
\end{theorem}
\begin{proof}

The formula for $\GLN$ follows immediately from Proposition \ref{prop:lose Z} and Lemma \ref{lem:coset-size}: each conjugacy class $[\sigma]$ in the symmetric group contributes the one-dimensional summand $\HHz(A,A_\sigma)$, and there are $\cP(N)$ such conjugacy classes.

For the $\SLN$ case, Lemma \ref{lem:coset-size} shows $\HHz(A,A_\sigma)$ has dimension equal to the index
\[\dl^2 = [\Uwp \oplus \Uwp : \Im(1-\sigma) \oplus \Im(1-\sigma)].\]
Together with Lemma \ref{lemma:totient} this concludes the proof in the $\SLN$ case.
\end{proof}

\begin{proof}[Proof of Theorems \ref{thm: HHz of GL skein algebra} \ref{thm: HHz of SL skein algebra}]

The  same formulas as above for the dimensions of the zeroeth Hochschild homology of the skein algebras for $G =\GLN, \SLN$  now follow from the isomorphism
$$
\SkAlg_G \cong \ee \DAHA{N}{N} \ee
$$
from \cite{GJV}, along with the shift isomorphism of \cite{Marshall1999}, that at our specialized parameters yields
$$ \ee \DAHA{N}{N} \ee \cong \ek{+} \SmashD \ek{+},
$$
which by \cite{Montgomery1980} is Morita equivalent to $\SmashD$. 
\end{proof}

\section{Morita equivalences for DAHA} \label{sec:DAHA stuff}

In this section we give criteria for when the idempotents $\emj$ from equation \eqref{eq:eNj} are conservative, hence define Morita equivalences.  These results are analogous to the determination of spherical values of the Cherednik algebras established in \cite{GS, Bezrukavnikov-Etingof}.

\nocite{Cherednik1995} \nocite{CherednikBook} \nocite{Cherednik-Selecta} \nocite{Cherednik1995}

Throughout this section we assume $t$ is 
not a root of unity.

\begin{theorem}\label{thm:idempotent}
Let $k \in \Q_{>0}$,  $n, m, j \in \Z_{>0}$, with $\ell = \lfloor k \rfloor$ and $m \le \frac nk$.

   Suppose $q=t^{-2k}$.   If $j \le k$ then $\emj$ is conservative for $\HG$.

Suppose $\Sq=t^{k/n}$.   If $j \le k$
   then $\emj$ is conservative for $\Schmaha{n}$.
   
   Suppose $\Sq=t^{k/n}$.  
   If $j \le k$
   then $\emj$ is conservative for $\HS$.
\end{theorem}

\begin{remark}
    Observe that parabolic sign idempotents $\emj$ are not conservative for the affine Hecke algebra, nor the finite Hecke algebra.  Similarly they are not conservative for the skein algebra of the annulus or disk.  Conservativity is a property of the torus.
\end{remark}

\begin{example}
Consider the case $k=2$, $n=4$, and let $M$ be the rectangular representation of $\HG$ or $\HS$ attached to the partition $(2,2)$, studied in \cite{Jordan-Vazirani}.  Then we have $\ek{4}M = 0$.  This shows that the restriction on $k$ is non-trivial.
\end{example}

We first need the following simplifying reduction, adapted from \cite{Bezrukavnikov-Etingof}[Theorem 4.1].

\begin{proposition}\label{prop:BEG}
 Let $\ee$ denote any idempotent in $ \Hf_n\subset \HH_n$, where we let $\HH_n$ denote any of $\HG, \HS, \Schmaha{n}$ below.  The functor
\[e: \HH_n\modu \to \ee\HH_{n}\ee\modu\]
is conservative if, and only if, it is conservative on $\Y$-finite modules.
\end{proposition}

\begin{proof}
We will prove the contrapositive statement, using Proposition \ref{prop:Morita}. Suppose 
\[
B = \DAHA{N}{n}/\DAHA{N}{n} \ee\DAHA{N}{n} 
\]
is non-zero. Note that $\DAHA{N}{n}$ and thus $B$ is finitely generated as an $\SymX-\SymY$-bimodule. It follows that there exists a character $\chi:\SymY \to \K$ such that
\[
B_\chi := B \otimes_{\SymY} \r(\chi) \neq 0
\]
is nonzero.  As a $\DAHA{N}{n}$-module, $\B_\chi$ is $\Y$-locally finite and  $\ee B_\chi =0$ as required.
\end{proof}

In Section \ref{sec:apps to skeins} we will show how Theorem \ref{thm:relskeins} follows from the following result, the proof of which comprises the remainder of this section.

As a special case of Theorem \ref{thm:idempotent}, in light of Proposition \ref{prop:BEG}, we state Theorem \ref{theorem:signN neq 0} for the specializations we use for applications to Skein Theory in Section \ref{sec:apps to skeins}.  We will in particular care about the case $j=1$. 

\begin{theorem}
\label{theorem:signN neq 0}
Fix integers $n \ge N \ge 1$, and let $k=\frac{n}{N}$. Suppose that either:
\begin{enumerate}
    \item $q=t^{-2k}$, and $M$ is a simple $\Y$-finite  $\HG$-module, or
    \item $\Sq=\St^{\frac{1}{N}}$, and $M$ is a simple $\Y$-finite $\Schmaha{n}$-module, or
    \item $\Sq=\St^{\frac{1}{N}}$, and $M$ is a simple $\Y$-finite $\HS$-module.
\end{enumerate}
Then we have $\eNj M \neq 0$ for $j \in \Z_{\le k}$.
\end{theorem}

In the proof of Theorem \ref{thm:idempotent}, it will suffice to restrict $M$ to an $\H_n$-module, and to look for $M'$ an $\H_n$-submodule such that $\emj M' \neq 0$.  The submodule $M'$ will be one of the simple quotients $L(\Delta)$ of the induced modules $M(\Delta)$ constructed from a multisegment $\Delta$. 
The conjugation action of $\pi$ (resp. $\Spi$) will be a key tool in finding appropriate $\Delta$. 
Before beginning the proof, we collect a number of preliminary results we will need.
 \begin{lemma} \label{lemma:en not kill M}
     Let $n= |\Delta|$.  Then $\ek{n} M(\Delta) \neq 0$. In particular if $\Delta$ is such that $L(\Delta) = M(\Delta)$ we have $\ek{n} L(\Delta) \neq 0$.
 \end{lemma}
 \begin{proof}
     Let $\alpha = \alpha(\Delta).$ Restricting to $\Hfn$ we have $\Res_{\Hfn}^{\H_n} M(\Delta) = \Ind_{\Hf_\alpha}^{\Hfn} \sgna{\alpha}$. The result follows from Claim \ref{claim:composition factors alpha}.
 \end{proof}

\begin{proposition}\label{prop:signABneq0}
Let $\Delta$ be a right-ordered multisegment of the form $\Delta = A\sqcup B$, such that $A$ is $1$-narrow.  Then for any $0\leq m \leq p=|A|$ we have $\ek{m} L(\Delta)\neq 0$.
\end{proposition}

\begin{proof}
We begin by observing that for any $m\leq p$, and for any multisegment $A$ of size $p=|A|$, $\ek{m} M(A)\neq 0$ by Lemma \ref{lemma:en not kill M}.
By Lemma \ref{cor:distinctstarts}, we have $L(A)=M(A)$, so in particular $\ek{m} L(A)\neq 0$.  We have further that
\[
M(\Delta)=M(A\sqcup B) = \Ind^{\H_n}_{\H_p\otimes \H_{n-p}}(M(A)\boxtimes M(B))
\]
has unique simple quotient $L(\Delta)$, while $M(B)$ has a unique simple quotient $L(B)$ since $A \sqcup B$ is right-ordered.  Hence the canonical map of $\H_p\otimes\H_{n-p}$-modules, $L(A)\boxtimes L(B)\to \Res^{\H_n}_{\H_p\otimes\H_{n-p}}(L(\Delta))$ is nonzero, hence an injection.  It follows therefore that $\ek{m} L(\Delta)$ is nonzero.
\end{proof}

We generalize Proposition \ref{prop:signABneq0}  below. However the proof of Proposition \ref{prop:signABneq0lnarrow} potentially requires replacing the multisegment by an equivalent one.

\begin{proposition}\label{prop:signABneq0lnarrow}  Fix natural numbers $n$ and $\ell$, and suppose $m \in \Z_{>0}$ satisfies $\ell m \le n$. 
Let $\Delta$ be a right-ordered multisegment of the form $\Delta = A\sqcup B$, such that $A$ is $\ell$-narrow and  $\ell m \leq|A|$. Then we have $\ek{m^\ell} L(\Delta)\neq 0$.
\end{proposition}
\begin{proof}
    Write $A = A^{(1)} \sqcup A^{(2)} \sqcup \cdots \sqcup A^{(s)}$ where each  $A^{(i)}$
    is a maximal subcollection of segments of $A$ lying on a given $t^2$-line.
    Note each $A^{(i)}$ is $\ell$-narrow.  Let $a_i$ be the first start of $A^{(i)}$.
    Because it is $\ell$-narrow, each $A^{(i)}$ can be further decomposed as $A^{(i)}=A^{(i,1)} \sqcup \cdots \sqcup A^{(i,\ell)}$ where  $A^{(i,j+1)}$ are those segments with start $a_i t^{2j}$. 
    Then set 
    $A^{(*,j)}=A^{(1,j)} \sqcup \cdots \sqcup A^{(s,j)}$, which by construction
    is $1$-narrow (or possibly empty), i.e. any two starts  are either equal or lie on distinct $t^2$-lines, and further
we have preserved being right-ordered, so that
$\Gamma = A^{(*,1)} \sqcup \dots \sqcup A^{(*,\ell)}$ is right-ordered and equivalent to $A$.

Let $\alpha = (|A^{(*,1)}|, \dots, |A^{(*,\ell)}|)$ which is a composition with $\le \ell$ nonzero parts. 
We may write $|\alpha|  = \ell m + r_1$ and $|B| = r_2$ with $\ell m +r_1+r_2 = n$.
Since each 
$A^{(*,j)}$    is $1$-narrow, $\ek{\alpha} ( L(A^{(*,1)})  \boxtimes \cdots \boxtimes L(A^{(*,\ell)})) \neq 0 $
by Lemma \ref{lemma:en not kill M}.   
As  $\Gamma$ is right-ordered, we have
$L(A^{(*,1)})  \boxtimes \cdots \boxtimes L(A^{(*,\ell)}) \boxtimes L(B) \subseteq \Res L(\Gamma) \cong L(\Delta)$, and so $\ek{\alpha} L(\Delta) \neq 0$.  Finally, (restricting further to $\Hf$) Proposition \ref{prop:kN+r} implies $\ek{m^{\ell}} L(\Delta) \neq 0$ 
\end{proof}

\begin{lemma}
\label{lemma:narrow-linkage}
Let $\Gamma$ be a multisegment which is $s$-narrow.  Let $L$ be any simple subquotient of $M(\Gamma)$.  Write $L=L(\Delta)$ for some right-ordered multisegment $\Delta$.  Then $\Delta$ is also $s$-narrow.
\end{lemma}

\begin{proof}
 We shall show that the starts of $\Delta$  form a subset of the starts of $\Gamma$, hence if $\Gamma$ is $s$-narrow, then so is $\Delta$.
 We will only prove it for the case that $\Gamma$ consists of two segments, and then by transitivity of induction, the argument extends to a multisegment with an arbitrary number of segments.
 
 Let $\Gamma=(\seg{a}{\ell},\seg{b}{p})$ and
 $\Gamma'=(\seg{b}{p},\seg{a}{\ell})$.  Since we will consider both multisegments, and in fact $M(\Gamma)$ and $M(\Gamma')$ have the same composition factors, let us assume $\Gamma$ is right-ordered. 
Recall from \cite{V-multisegments} that $M(\Gamma)$ is simple and further $M(\Gamma) \cong M(\Gamma')$ except in the special case $(\ast)$ that $b= a t^{2z}$ for some positive integer $z \in \Z_{>0}$ such that  $p \ge z \ge p-\ell+1$, as pictured in Figure \ref{fig:linkage}. 

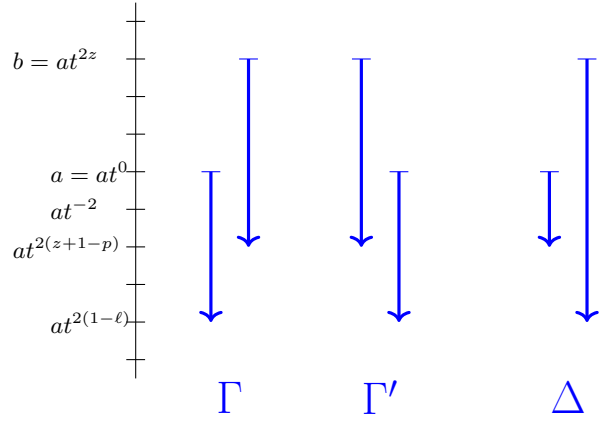
\begin{figure}
    \centering
    \begin{tikzpicture}[scale=0.5]
\begin{scope}[shift={(0,0)}]
\draw[black] (0,-0.5) -- (0, 9.5);
\draw[black] (-0.25,0) -- (0.25, 0);
\draw[black] (-0.25,1) -- (0.25, 1);
\draw[black] (-0.25,2) -- (0.25, 2);
\draw[black] (-0.25,3) -- (0.25, 3);
\draw[black] (-0.25,4) -- (0.25, 4);
\draw[black] (-0.25,5) -- (0.25, 5);
\draw[black] (-0.25,6) -- (0.25, 6);
\draw[black] (-0.25,7) -- (0.25, 7);
\draw[black] (-0.25,8) -- (0.25, 8);
\draw[black] (-0.25,9) -- (0.25, 9);
{\scalefont{.8}
\draw[anchor=west] (-2.5,5) node  {$a=a t^{0}$};
\draw[anchor=west] (-2.5,4) node  {$a t^{-2}$};
\draw[anchor=west] (-2.5,1) node  {$a t^{2(1-\ell)}$};
\draw[anchor=west] (-3.5,8) node  {$b=a t^{2z}$};
\draw[anchor=west] (-3.5,3) node  { $a t^{2(z+1-p)}$};
}
\end{scope}

\begin{scope}[shift={(0,0)}]
\draw[->, very thick, blue] (2,5) -- (2, 1);
\draw[->, very thick, blue] (3,8) -- (3, 3);
\draw[blue] (1.75,5) -- (2.25, 5);
\draw[blue] (2.75,8) -- (3.25, 8);
\draw (2.5,-1) node[blue]  {{\scalefont{1.5}$\Gamma$}};
\end{scope}
\begin{scope}[shift={(4,0)}]
\draw[->, very thick, blue] (3,5) -- (3, 1);
\draw[->, very thick, blue] (2,8) -- (2, 3);
\draw[blue] (2.75,5) -- (3.25, 5);
\draw[blue] (1.75,8) -- (2.25, 8);
\draw (2.5,-1) node[blue]  {{\scalefont{1.5}$\Gamma'$}};
\end{scope}

\begin{scope}[shift={(5,0)}]
\draw[->, very thick, blue] (6,5) -- (6, 3);
\draw[->, very thick, blue] (7,8) -- (7, 1);
\draw[blue] (5.75,5) -- (6.25, 5);
\draw[blue] (6.75,8) -- (7.25, 8);

\draw (6.5,-1) node[blue]  {{\scalefont{1.5}$\Delta$}};
\end{scope}

\end{tikzpicture}
    \caption{$\Gamma$ and $\Delta$ are right-ordered, but $\Gamma'$ is not. Note the starts of $\Delta$ are a subset of those of $\Gamma$.}
    \label{fig:linkage}
\end{figure}
Then we have  short exact sequences,
\begin{gather*}
0 \to L(\Delta) \to M(\Gamma) \to L(\Gamma) \to 0,
\\
0 \to L(\Gamma) \to M(\Gamma') \to L(\Delta) \to 0
\end{gather*}
where $\Delta$ is  right-ordered, with either
$ \Delta = (\seg{b}{\ell+p})$
or
$ \Delta = (\seg{a}{p_1},\seg{b}{p_2}),$
  for some $p_1, p_2$ with $p_1+p_2=\ell_1+\ell_2$
  (specifically $p_1 = p-z$, $p_2 = z + \ell$). 
  Since $\Delta$ does not satisfy $(\ast)$, it holds that $M(\Delta)=L(\Delta)$; this is used in the general induction argument which we have omitted.  Finally, we observe that the starts of $\Delta$ are a subset of $\{a,b\}$.
\end{proof}

\begin{proposition}
    
\label{prop:k-narrow}
Fix an integer $n \geq 1$, and let $k\in \Q_{>0}$.  Suppose that either:
\begin{enumerate}
    \item $q = t^{-2k}$, and $M$ is a simple and $\Y$-finite $\HG$-module, or
     \item $\Sq = \St^{\frac{k}{n}}$, and $M$ is a simple and $\Y$-finite  $\Schmaha{n}$-module,
    \item $\Sq = \St^{\frac{k}{n}}$, and $M$ is a simple and $\Y$-finite  $\HS$-module,
\end{enumerate}
Then $\Res^{\HH}_{\H}M$ contains a simple submodule of the form $L(\Delta)$, where $\Delta$ is a $k$-narrow right-ordered multisegment.

\end{proposition}

\begin{proof}
We will treat the $\GL_n$ case first, and afterwards outline the changes for the Schmaha and  $\SLn$ case.  We will abbreviate $\HH=\HG$, and $\H=\H_{n}(t) = \HY$, the AHA regarded as a subalgebra generated by the subalgebras $\Hf,\Y \subset \HH$.

Since $M$ is $\Y$-finite, we may find 
simple submodules of $\Res^\HH_{\H(\Y)} M$, each of which is of the form $L(\Delta)$ for some right-ordered multisegment $\Delta$.  If $\Delta$ is $k$-narrow, then we are done. If not, we will produce a succession of multisegments until our final one has the desired properties and by abuse of notation we will eventually also call that one $\Delta$.

Write $k= \frac{n_0}{N_0}$ with $\gcd(n_0,N_0)=1$ and $n_0, N_0 > 0$.

Assume otherwise,
i.e., that $\Delta $ is $s/{N_0}$-narrow for some $s>n_0$,

We will now describe an inductive procedure to produce another simple submodule of $\Res M$ of the form $L(\widetilde{\Delta})$ such that $\widetilde{\Delta}$ is  right-ordered and $\frac{s-1}{N_0}$-narrow.

Replacing $\Delta$ with an equivalent multisegment, we may write $\Delta=\Delta_{(a)}\sqcup \Delta^c$, where
$\Delta_{(a)}$ is not $k$-narrow, and furthermore, for some $a \in \K^\times$
all starts of $\Delta_{(a)}$ are on the same $t^{2/N_0}$-line as $a$ but are on distinct $t^{2/N_0}$-lines from all starts of $\Delta^c$.  
We may further assume that
\[\Delta_{(a)}=(\seg{a}{\ell_1},\seg{a t^{\frac{2z_2}{N_0}}}{\ell_2}, \ldots, \seg{a t^{\frac{2z_m}{N_0}}}{\ell_m}),\]
with $0 = z_1 \le z_i \le z_{i+1}$ for $i=1,\ldots, m-1$ because $\Delta$ is right-ordered. 
We note that $\Delta_{(a)}$ is $\frac{z_m+1}{N_0}$-narrow, so by assumption we have $z_m\geq n_0$.  Let $b$ denote the area of the start $a$, i.e., $b=\ell_1 + \cdots + \ell_p$, where $z_2=\cdots =z_p=0$, but $z_{p+1}\neq 0$.

\begin{figure}
    \centering
     \begin{tikzpicture}[scale=0.4]

\begin{scope}[shift={(0,0)}]
\draw[black] (0,-0.5) -- (0, 9.5);
\draw[black] (-0.25,0) -- (0.25, 0);
\draw[black] (-0.25,1) -- (0.25, 1);
\draw[black] (-0.25,2) -- (0.25, 2);
\draw[black] (-0.25,3) -- (0.25, 3);
\draw[black] (-0.25,4) -- (0.25, 4);
\draw[black] (-0.25,5) -- (0.25, 5);
\draw[black] (-0.25,6) -- (0.25, 6);
\draw[black] (-0.25,7) -- (0.25, 7);
\draw[black] (-0.25,8) -- (0.25, 8);
\draw[black] (-0.25,9) -- (0.25, 9);
{\scalefont{.8}
\draw[anchor=west] (-2.7,8) node  {$a t^{\frac{2z_m}{N_0}}$};
\draw[anchor=west] (-2.7,6) node[blue]  {$a t^{\frac{2n_0}{N_0}}$};
\draw[anchor=west] (-2.7,2) node  { $a t^{\frac 4{N_0}}$};
\draw[anchor=west] (-2.7,1) node  { $a t^{\frac 2{N_0}}$};
\draw[anchor=west] (-2.7,0) node[blue]  {$a $};
}

\end{scope}

\begin{scope}[shift={(0,0)}]
\draw[->, very thick, blue] (2,0) -- (2, -1);
\draw[->, very thick, blue] (3,0) -- (3, -3);
\draw[->, very thick, blue] (4,0) -- (4, -3);
\draw[->, very thick, blue] (5,0) -- (5, -4);

\draw[->, very thick] (6,1) -- (6, -4);
\draw[->, very thick] (7,1) -- (7, -4);
\draw[->, very thick] (8,6) -- (8, -1);
\draw[->, very thick] (9,6) -- (9, 1);
\draw[->, very thick] (10,8) -- (10, 7);
\draw[->, very thick] (11,8) -- (11, 5);
\draw[->, very thick] (12,8) -- (12, 4);

\draw (6.5,-7) node[black]  {{\scalefont{1.3}$\Delta_{(a)}$}};
\draw (3.5,-5) node[blue]  {$\underbrace{\hspace{1.2cm}}_{\scalefont{1.5}p}$};
\end{scope}

\begin{scope}[shift={(13,0)}]
\draw[->, very thick, blue] (13,6) -- (13, 5);
\draw[->, very thick, blue] (14,6) -- (14, 3);
\draw[->, very thick, blue] (15,6) -- (15, 3);
\draw[->, very thick, blue] (16,6) -- (16, 2);

\draw[->, very thick] (6,1) -- (6, -4);
\draw[->, very thick] (7,1) -- (7, -4);
\draw[->, very thick] (8,6) -- (8, -1);
\draw[->, very thick] (9,6) -- (9, 1);
\draw[->, very thick] (10,8) -- (10, 7);
\draw[->, very thick] (11,8) -- (11, 5);
\draw[->, very thick] (12,8) -- (12, 4);

\draw (14.5,1) node[blue]  {$\underbrace{\hspace{1.2cm}}_{\scalefont{1.5}p}$};
\draw (10.5,-7) node[black]  {{\scalefont{1.3}$\Gamma'_{(a)}$}};
\end{scope}

\end{tikzpicture}

    \caption{We apply $\pi^{-b}$ where $b$ is the start area of ${\color{blue}a}$. We see how this transforms a $\frac{z_m+1}{N_0}$-narrow multisegment to one that is   $\frac{z_m}{N_0}$-narrow if $z_m$ is too large. The result may not be right-ordered. 
    }
    \label{fig:multisegment not narrow}
\end{figure}
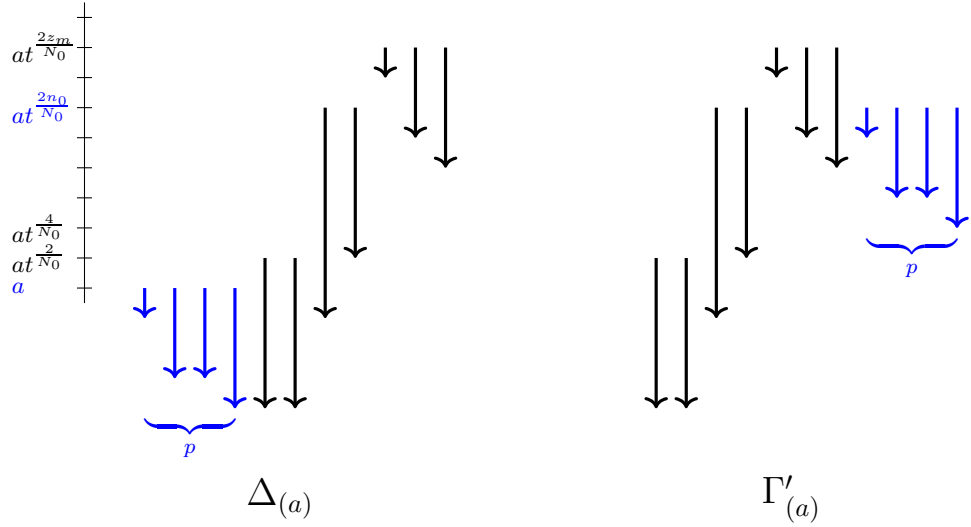

Since $\Hom_{\H}(M(\Delta),\Res^{\HH}_{\H} (M)) \neq 0$, let us pick a $v_\Delta\in M$ as in Notation \ref{not:vdelta}.  The vector $v_\Gamma = \pi^{-b} v_\Delta \in M$ is non-zero, and hence gives rise to a non-zero element of $\Hom_{\H}(M({\Gamma}),M)$, where
\begin{align} 
{\Gamma} &= (\seg{a_{p+1}}{\ell_{p+1}},\ldots,  \seg{a_m}{\ell_m}) \sqcup \Delta^c \sqcup (\seg{q^{-1}a}{\ell_1},\ldots,\seg{q^{-1}a}{\ell_p}) \notag\\
&=(\seg{t^{2\frac{z_{p+1}}{N_0}}a_{}}{\ell_{p+1}},\ldots,  \seg{t^{2\frac{z_m}{N_0}}a}{\ell_m}) \sqcup \Delta^c \sqcup (\seg{t^{2\frac{n_0}{N_0}}a}{\ell_1},\ldots,\seg{t^{2\frac{n_0}{N_0}}a}{\ell_p})\label{eq:SL Gamma}
\end{align}
Observe that $\Gamma \sim \Gamma'$, where:
\[\Gamma' = \Gamma'_{(a)}\sqcup \Delta^c, \textrm{ with }\]
\[\Gamma'_{(a)}= (\seg{a_{p+1}}{\ell_{p+1}},\ldots,\seg{a_m}{\ell_m}) \sqcup (\seg{a t^{2\frac{n_0}{N_0}}}{\ell_1},\ldots,\seg{a t^{2\frac{n_0}{N_0}}}{\ell_p}).\]
We have now that ${\Gamma'_{(a)}}$ is $\frac{z_m}{N_0}$-narrow. (See Figure \ref{fig:multisegment not narrow} for a sketch of $\Delta_{(a)}$ versus ${\Gamma'_{(a)}}$.)  Moreover, we have produced a nonzero homomorphism from $M(\Gamma')$ to $\Res^{\HH}_{\H} M$.

Consider any such nonzero homomorphism and pick a simple $\H$-submodule $L(\Delta')$ in its image.  By definition $\Delta'$ is right-ordered.  Write $\Delta'=\Delta'_{(a)} \sqcup (\Delta')^c$, so that $\Delta'_{(a)}$ consists of precisely those segments of $\Delta'$ with starts on the same $t^{2/N_0}$-line as $a$.  By Lemma \ref{lemma:narrow-linkage}, $\Delta'_{(a)}$ is also $\frac{z_m}{N_0}$-narrow.

Repeating this process, we eventually produce a  simple submodule $L({\Delta})$ of $\Res^{\HH}_{\H} M$ where ${\Delta}$ is right-ordered and $k$-narrow, concluding the proof in the $\GL_n$ case.

\medskip 

$\SLGL$ Let us now outline the changes necessary to make the proof work in the Schmaha or  $\SLn$ case.  We will abbreviate $\HH=\Schmaha{n}$ or $\HS$, and $\H=\HY$, regarded as a subalgebra generated by the subalgebras $\Hf,\cZ \subset \HH$.
The only modifications $\SLGL$ concern the action of $\Spi$ when we set $v_{\Gamma} = \Spi^{-b} v_\Delta$. More specifically it is the conjugation action by $\Spi$ that is modified from that of $\pi$ (but those actions agree for $\HS$ and $\Schmaha{n}$). 
The multisegment  we obtain is
$ t^{-\frac{2kb}{n}}\Gamma$; it differs 
 from Equation \eqref{eq:SL Gamma} 
in that all of its starts are uniformly scaled by $t^{-\frac{2kb}{n}}= t^{-\frac{2 n_0 b}{n N_0}}$ (recalling $\frac{n_0}{N_0}=k$), as below:
\begin{gather}\label{eq:SL Gamma'} 
t^{\frac{-2kb}{n}} \Gamma  = (\seg{at^{\frac{2z_{p+1}}{N_0} -\frac{2kb}{n}}}{\ell_{p+1}},\ldots, \seg{at^{\frac{2z_m}{N_0} -\frac{2kb}{n}}}{\ell_m}) \sqcup t^{\frac{-2kb}{n}}\Delta^c \sqcup (\seg{at^{\frac{2n_0}{N_0}- \frac{2kb}{n} }}{\ell_1},\ldots,\seg{at^{\frac{2n_0}{N_0}- \frac{2kb}{n}}}{\ell_p}),
\end{gather}
where we have written $t^{\frac{-2kb}{n}}\Delta^c$ for the corresponding multisegment that has its starts scaled by $t^{-\frac{2kb}{n}}$. 
So while specific 
$t^{\frac 2{N_0}}$-lines have changed, depending on $\gcd(b, n_0)$, the property of belonging to the same $t^{\frac 2{N_0}}$-line, and of being $\frac{z_m}{N_0}$-narrow, is preserved. 
The rest of the proof continues as in the $\GL$ case.

\end{proof}

\begin{corollary} \label{cor:k=1}
Suppose we are in any of the  cases of Proposition \ref{prop:k-narrow} and that 
$k=1$, 
and let $M$ be a simple and $\Y$-finite $\HH$-module.  Then we have $\ek{n} M\neq 0$.
\end{corollary}
\begin{proof}
By Proposition \ref{prop:k-narrow}, $\Res^{\HH}_{\H}M$ contains an $\H$-submodule $L(\Delta)$, with $\Delta$ right-ordered and $1$-narrow.  Therefore the claim follows from Proposition \ref{prop:signABneq0}. 
\end{proof}

\begin{notation}\label{not:sequence}
A right-ordered multisegment all of whose starts lie on the same $t^{\frac{2}{N_0}}$-line is equivalent to one of the form,
\[(\seg{a}{\ell_1},\seg{at^{\frac{2z_2}{N_0}}}{\ell_2},\ldots, \seg{at^{\frac{2z_m}{N_0}}}{\ell_{m}}),\]
with $0\leq z_2\leq \cdots \leq z_m$ and $z_i\in \mathbb{Z}$, and some $a\in\K^{\times}$.
Its \defterm{sequence of start areas} refers to the sequence $(b_1,\ldots, b_{n_0})$ where each $b_{j+1}$ is the sum of the areas of all starts of the form $at^{\frac{2j}{N_0}}$ (in particular, the sequence contains zeroes for start areas that do not occur).
\end{notation}

\begin{proposition}\label{prop:l-narrow}
    Suppose we are in any of the  cases of Proposition \ref{prop:k-narrow} and that
    $\ell = \lfloor k \rfloor \neq 0$ and $m\in \Z_{>0}$ satisfies $m \le \frac nk$.
    Suppose 
     $L({\Delta}) \subseteq \Res^{\HH}_{\H} M$ with ${\Delta}$  right-ordered and $k$-narrow.
     Then $\Res^{\HH}_{\H} M \supseteq L(\Gamma)$ where $\Gamma = A \sqcup B$  is right-ordered with $|A| \ge \ell m$ and $A$ $\ell$-narrow.
\end{proposition}
\begin{proof}
    We will  treat the $\GL_n$ case in detail first, and then explain the required modifications $\SLGL$ for the  Schmaha or $\SLn$ case.  
Recall we write $k=\frac{n_0}{N_0}$, where $\gcd(n_0,N_0)=1$ and $n_0, N_0>0$
so that $n_0 = N_0 \ell + r_0$ with $0 \le r_0 < N_0$.

 Possibly replacing $\Delta$ by an equivalent right-ordered multisegment,
 we may write
$\Delta = \Delta^{(1)} \sqcup \Delta^{(2)} \sqcup \cdots \sqcup \Delta^{(s)},$
where each $\Delta^{(i)}$ is here a maximal subcollection of segments of $\Delta$ lying on a given $t^{\frac{2}{N_0}}$-line.  Note each $\Delta^{(i)}$ is $k$-narrow.  Let us denote $m_i=|\Delta^{(i)}|$. Observe $\sum_i m_i = n$.  Below we will construct another right-ordered multisegment with a decomposition,
\[
 \Gamma = 
 A^{(1)} \sqcup  \cdots \sqcup A^{(s)} \sqcup B^{(1)}\sqcup \dots \sqcup  B^{(s)}
 \sim A^{(1)} \sqcup B^{(1)} \sqcup \cdots \sqcup A^{(s)} \sqcup   B^{(s)}  ,
\]
with $\Hom(M(\Gamma),\Res^{\HH}_{\H}M ) \neq 0$, and satisfying the following three properties, the third property being what we require to complete the proof:
\begin{enumerate}
  \item \label{item:distinct} $A^{(i)}$ and $A^{(p)}$ occupy distinct $t^{\frac{2}{N_0}}$-lines (hence in particular distinct $t^2$-lines), for any pair $i \neq p$.
    \item \label{item:cyclic} The sequence of start areas of each $A^{(i)}\sqcup B^{(i)}$ is some cyclic shift of the sequence of  start areas of $\Delta^{(i)}$. In particular $m_i=|A^{(i)} \sqcup B^{(i)}|$.
 
    \item \label{item:1narrow}  $A = A^{(1)} \sqcup \cdots \sqcup A^{(s)}$ is right-ordered and $\ell$-narrow, with $|A| \ge \ell m$. 
\end{enumerate}

\medskip

Let us now construct the desired $\Gamma$ from the given $\Delta$. 
Notice if $r_0=0$, i.e., $\ell=k$, then we may let $A=\Gamma=\Delta$ and $B=\emptyset$.
If each $\Delta^{(i)}$ can already be decomposed as $\Delta^{(i)}=A^{(i)}\sqcup B^{(i)}$ satisfying condition \eqref{item:1narrow}, then we are done, noting that \eqref{item:distinct} holds automatically from how the $\Delta^{(i)}$ were defined.  Assuming otherwise, we may suppose, replacing $\Delta$ with an equivalent multisegment if necessary, that $\Delta^{(1)}$ violates  \eqref{item:1narrow}.  Let us rename,
\[\Delta^{(1)} = \Delta_{(a)}, \qquad \Delta^c = \Delta^{(2)} \sqcup \cdots \sqcup \Delta^{(s)},\]
so that $\Delta = \Delta_{(a)} \sqcup \Delta^c$, and all starts of $\Delta_{(a)}$ lie on the same $t^{\frac{2}{N_0}}$-line as $a$, and all starts on $\Delta^c$ lie on distinct $t^{\frac{2}{N_0}}$-lines from $a$.

Because $\Delta$ is right-ordered and $k$-narrow, we have $z_m < n_0$, where we write 

\[\Delta_{(a)} = (\seg{a}{\ell_1},\seg{at^{\frac{2z_2}{N_0}}}{\ell_2},\ldots, \seg{at^{\frac{2z_m}{N_0}}}{\ell_m}),\]
with $0\leq z_2\leq \cdots \leq z_m$  as in Notation \ref{not:sequence}.
 Let $(b_1, b_2, \ldots, b_{n_0})$ be the sequence of start areas of $\Delta_{(a)}$. 
Notice that if we take the segments corresponding to $\ell N_0$ consecutive start areas (including possible empty ones), the resulting multisegment will be $\ell$-narrow.  Our aim is to construct one of large enough size.

Pick a $v_\Delta$ as in Notation \ref{not:vdelta}.
Let $p$  be maximal such that $z_1=z_2 = \cdots = z_p=0$, so that $b_1 = \ell_1 + \cdots + \ell_p$.
We therefore obtain a non-zero $v_{\Delta'} = \pi^{-b_1} v_\Delta \in M$,
where
\begin{gather}\label{eq:GL Delta'}
\Delta' = (\seg{at^{\frac{2z_{p+1}}{N_0}} }{\ell_{p+1}},\ldots, \seg{at^{\frac{2z_m}{N_0}}}{\ell_m} \sqcup \Delta^c \sqcup (\seg{at^{\frac{2n_0}{N_0}}}{\ell_1},\ldots,\seg{at^{\frac{2n_0}{N_0}}}{\ell_p}).
\end{gather}
  We have \[
\Delta'\sim \Gamma' = \underbrace{(\seg{at^{\frac{2z_{p+1}}{N_0}}}{\ell_p+1},\ldots,\seg{at^{\frac{2z_m}{N_0}}}{\ell_m}) \sqcup (\seg{a t^{\frac{2n_0}{N_0}}}{\ell_1},\ldots,\seg{at^{\frac{2n_0}{N_0}}}{\ell_p})}_{\Gamma'_{(a)}}\sqcup \Delta^c.
\]
Observe that $\Gamma'$ is still right-ordered and $k$-narrow but the sequence of start areas  of $\Gamma'_{(a)}$ is cycled to $(b_2,b_3,\ldots,b_{n_0},b_1)$.  
(This is similar to the process depicted in Figure \ref{fig:multisegment not narrow}, but \emph{preserves} $k$-narrowness.)  Iterating this process for each $t^{\frac{2}{N_0}}$-line, we produce a nonzero $v_\Gamma \in M$, with $\Gamma$ right-ordered, $k$-narrow and such that the sequence of start areas in any $\Gamma^{(i)}$ is an arbitrary cyclic shift from those of $\Delta^{(i)}$. 

We claim that there exists such a shifted $\Gamma$ so that the start areas $(b_{g+1},b_{g+2},\ldots, b_{n_0},b_1,\ldots b_g)$ of each $\Gamma^{(i)}$ satisfy the inequality,
\begin{equation}
     \sum_{d=g+  1}^{g+\ell N_0} b_d \geq m_i  \frac{\ell}{k},
    \label{eqn:key-inequality}
\end{equation}
taking indices $\bmod\, n_0$, and recalling $m_i= |\Delta^{(i)}| = |\Gamma^{(i)}|$.  
Indeed, consider the sum over all such expressions as $g$ varies.  We have:
\[\sum_{g=1}^{n_0}\sum_{d=g+1}^{g+ \ell N_0} b_d = \ell N_0\sum_{d=1}^{n_0}b_d = \ell N_0 m_i.\]
It follows that at least one of the $n_0$ summands $\sum_{d=g+1}^{g+\ell N_0} b_d$ must be greater than or equal to $\frac{\ell N_0 m_i}{n_0} = \frac{\ell m_i}{k},$ as claimed.  

Having made this shift, we put these larger start areas at the leftmost places, and write this cyclic shift of $\Delta^{(i)}$ as $A^{(i)} \sqcup B^{(i)}$, so   $|A^{(i)}| \ge m_i \frac {\ell}{k}$ and having just taken $\ell N_0$ consecutive start areas, $A^{(i)}$ is $\ell$-narrow. 
Taking $A$ as in \eqref{item:1narrow} and  summing the inequality \eqref{eqn:key-inequality} over $i$ yields $|A| \ge n \frac {\ell}{k} \ge \ell m$. 
Finally setting $B= B^{(1)}\sqcup \cdots \sqcup B^{(s)}$ implies $\Gamma$ satisfies conditions \eqref{item:cyclic} and \eqref{item:1narrow}.  
Hence 
the proof is complete in the $\GL_n$ case.

\medskip 

$\SLGL$  For the $\SLn$ case, the only modification concerns the action of $\Spi$ when we set \mbox{$v_{\Delta'} = \Spi^{-b_1} v_\Delta$}.

The $\Delta'$ we obtain differs from Equation \eqref{eq:GL Delta'} in that all of its starts are uniformly scaled by $t^{-\frac{2b_1k}{n}}$. 
So while specific 
$t^{\frac 2{N_0}}$-lines have changed (depending on $\gcd(b_1, n_0)$), the property of belonging to the same $t^{\frac 2{N_0}}$-line is preserved. 
Thus the sequence of start areas is still cycled to
$(b_2, \ldots, b_{n_0}, b_1)$.  The rest of the proof of Theorem \ref{theorem:signN neq 0} proceeds as in the $\GL$ case.

\end{proof}

\begin{proof}[Proof of Theorem  \ref{thm:idempotent}]   

Putting together the propositions from this section, we now prove Theorem  \ref{thm:idempotent},
which has Theorem \ref{mainthm:morita-DAHA-eNj} as a special case and is a key ingredient of Theorem \ref{mainthm:morita-DAHA}.

 Given simple $M$, by Proposition \ref{prop:k-narrow}, we may find some $k$-narrow right-ordered multisegment $\Delta$ such that $\Hom(M(\Delta),\Res^{\HH}_{\H} M)\neq 0$. 
 From this we construct $\Gamma$ as in Proposition \ref{prop:l-narrow} that satisfies the
 hypotheses of Proposition \ref{prop:signABneq0lnarrow} as well as $\Hom(M(\Gamma),\Res^{\HH}_{\H} M)\neq 0$.
Then by Proposition \ref{prop:signABneq0lnarrow}   we have $\ek{m^\ell} L(\Gamma)\neq 0$
and thus $\ek{m^{\ell}} M \neq 0$.
\end{proof}
 
\subsection{Modification for other  specializations}
Theorem \ref{thm:idempotent} begs the following interesting question: for precisely which values of $k$ is $\emj$ conservative? 
The following proposition -- which is not needed in the rest this paper -- gives a partial answer. 

 \begin{proposition}
     \label{prop:signn neq 0}
Fix an integer $n  \ge 1$, and let $k \not\in \Q$ (by which we mean $t^a = q^b$ for $a,b \in \Z$ implies $a=b=0$)  or $k \in \Q_{< 0}$. Suppose that either:
\begin{enumerate}
    \item $q=t^{-2k}$, and $M$ is a simple $\Y$-finite  $\HG$-module, or
    \item $\Sq=\St^{\frac{k}{n}}$, and $M$ is a simple $\Y$-finite $\Schmaha{n}$-module, or
    \item $\Sq=\St^{\frac{k}{n}}$, and $M$ is a simple $\Y$-finite $\HS$-module.
\end{enumerate}
Then we have $\ek{n} M \neq 0$.  Consequently $\ek{\alpha} M \neq 0$ for $\alpha \vDash n$. 
 \end{proposition}

\begin{proof}[Proof sketch]
    If $k \not\in \Q$ it is easy to use the methods in Proposition \ref{prop:k-narrow} to produce a composition factor of $\Res_{\H}^{\HH} M$ of the form $L(\Delta) = M(\Delta)$ with $\Delta$ $1$-narrow, since $a$ and $q^{-1}a$ (or indeed $q^{z}a$ for any $z \in \Z\setminus \{0\}$) are on different $t^2$-lines. Thus by Lemma \ref{lemma:en not kill M}, $\ek{n} M \neq 0$.

    If $k \in \Q_{< 0}$ we take a different approach to produce a composition factor $M(\Delta)$ of $\Res_{\H}^{\HH} M$ that satisfies $M(\Delta) = L(\Delta)$.  Then by Lemma \ref{lemma:en not kill M} we will have $\ek{n} M \neq 0$.  For this, we employ another combinatorial property of multisegments that implies $M(\Delta)$ is irreducible.  In the case of two segments, possibly on the same $t^2$-line, but differing by a suffiently high power of $t^2$, we have by \cite{BZ77, Zelevinsky, V-multisegments} that the module
        $M(\seg{a}{\ell}, \seg{a t^{2s}}{r})$ is irreducible for all $\ell, s, r \in \Z_{\ge 1}$, $a \in \K^\times$  such that $r<s$.  With more than two segments a similar result applies.

    Instead of using the action of $\pi$ (or $\Spi$) to reduce the power of $t^{2/N_0}$ separating distinct starts of   segments (as in Figure \ref{fig:multisegment not narrow}), we increase it instead, making use of the ideas in 
     Lemma \ref{lemma:narrow-linkage}, until we can apply the irreducibility criterion above.  
\end{proof}

\begin{remark} \label{rem:triv vs sign}
    Results for specialization of parameters depending on $k$ for parabolic sign idempotents also hold at $-k$ for parabolic trivial idempotents.  One could furthermore use similar ideas
    as above to see $\ek{n}^+$ is conservative for $\SmashD$-modules. 
\end{remark}

\section{Applications to skein theory: the case $G=\SL_2$} \label{sec:SL2}
In this section, we give a complete description of the $\SL_2$-skein category of $T^2$.  We give direct proofs in this case which bypass the combinatorics of Section \ref{sec:DAHA stuff}.

\begin{theorem}\label{thm:SL2}
    Let $G=\SL_2$. Then
    \begin{enumerate}
        \item For any $n \ge 2$, the idempotent $\ee = \ek{2}$ satisfies
        \[
        \SkAlgNn{\SL_2}{n}\cdot \ee \cdot \SkAlgNn{\SL_2}{n} = \SkAlgNn{\SL_2}{n}.
        \]
        \item The objects $\Dist$ and $\Dist_V$ are orthogonal for the $\Hom$ pairing.
        \item Their sum $\Dist \oplus \Dist_V$ is a compact generator of $\SkCat_{\SL_{2}}(T^2)\modu$.
        \item We have isomorphisms,
        \[\End(\Dist) \cong \KLam^{\Sm{2}},\quad \End(\Dist_V) \cong \K^4.\] 
  
    \end{enumerate}
    In particular, we have an equivalence of categories:
    \[
    \SkCat_{\SL_{2}}(T^2)\modu \simeq \KLam \# \Sm{2}\modu \oplus \Vect_\K^{\oplus 4} \simeq \KLam^{\Sm{2}}\modu \oplus \Vect_\K^{\oplus 4} 
    \]
\end{theorem}

\begin{proof} Claim (1) is proved as Proposition \ref{prop:good-idempotent} below. Claim (2) is clear for degree parity reasons.  Claim (3) follows from Claim (1) and Proposition \ref{prop:Morita}, as we may inductively reduce the number of strands in any braid on the torus by 2, applying (1), as $\ek{2} \SkAlgNn{N}{n} \ek{2} \cong \SkAlgNn{N}{n-2}$.  The first isomorphism of Claim (4) is a celebrated result of Frohman and Gelca \cite{Frohman-Gelca}, while the second isomorphism is proved as Lemma \ref{lem:1-strand} below. \end{proof}

\begin{remark}
    The four simple objects appearing above may be understood as cuspidal $\q$-character sheaves (see the discussion in Section \ref{sec:quantum character}). 
\end{remark}

\begin{proposition}\label{prop:good-idempotent}
    Let $n \ge 2$ and $\ee = \ek{2}$. Then the 2-sided ideal $\HS \cdot \ee \cdot \HS = \HS$ unless $\Sq^{2n} = t^{2n}$.  Further $\Schmaha{n} \cdot \ee \cdot \Schmaha{n} = \Schmaha{n}$ unless $\Sq^{2n} = t^{2n}$.
    Consequently $\SkAlgNn{N}{n} \cdot  \ee \cdot \SkAlgNn{N}{n} = \SkAlgNn{N}{n}$.
\end{proposition}

\begin{proof}
Given the homomorphism $\SWdaff$ of \eqref{eqn:SW} at the specialization $\Sq=t^{1/N}=t^{1/2}=\q^{1/2}$ sends $\ee$ to $\ee$, the final statement follows from Proposition \ref{prop:idems}.

The proof below only relies on the conjugation action by $\Spi$ (and not its order), so the same proof works in $\HS$ or $\Schmaha{n}$.  Recall $\ee = \frac 1{1+t^{-2}}(\id - t^{-1} T_1) $. 

We separate out the case $n=2$ as it differs slightly from $n>2$.

First take $n=2$. 
Direct algebraic computation gives the equality:

$$
(1+t^{-2})\left((\ee - Z_1^{-1} \ee Z_2) -t^{-2}\Sq^4 \Spi(\ee - Z_2^{-1} \ee Z_1)\Spi^{-1} \right) = 1-t^{-4}\Sq^4,
$$
exhibiting 
$1 \in \HStwo \cdot \ee\cdot \HStwo$ as long as $t^4 \neq \Sq^4$, so in particular
$1 \in \HH_2(2)\cdot \ee\cdot \HH_2(2)$.

For $n > 2$, we compute that 
\[(t+t^{-1})\frac{Z_3}{Z_2}(\frac{Z_1}{Z_3} \ee - \ee\frac{Z_2}{Z_3}) = 
\frac{Z_3}{Z_2} (
 \frac{Z_1}{Z_3} (t -  T_1)   - (t -  T_1)\frac{Z_2}{Z_3})
= \frac{Z_3}{Z_2}(t^2 \frac{Z_1}{Z_3} - \frac{Z_2}{Z_3})
= t^2 \frac{Z_1}{Z_2} - 1.\]
Hence the right-hand-side lies in the ideal 
$\Schmaha{n} \ee \Schmaha{n}$.
Moreover for each $i=0,\ldots n-2$, we then have that
\begin{align}\label{eqn:Z1Zi}
\frac{Z_1}{Z_{i}}\Spi^i(t^2 \frac{Z_1}{Z_2} - 1)\Spi^{-i}
\end{align}
lies in the ideal.  The telescoping sum of terms \eqref{eqn:telescope} gives that
\begin{gather} \label{eqn:telescope}
(t^{2(n-1)} \frac{Z_1}{Z_n} - t^{2(n-2)} \frac{Z_1}{Z_{n-1}}) + \cdots + (t^4 \frac{Z_1}{Z_3} - t^2 \frac{Z_1}{Z_{2}})+ (t^2 \frac{Z_1}{Z_2} - 1) = t^{2(n-1)} \frac{Z_1}{Z_n} -1,
\end{gather}
lies in the ideal.  Finally, we compute 
\begin{gather} \label{eqn:pi conjugate Zs}
t^{2(n-1)}\frac{Z_1}{Z_n} \Spi (t^2 \frac{Z_{n-1}}{Z_{n}} - 1) \Spi^{-1} = 
t^{2(n-1)}\frac{Z_1}{Z_n}(t^2 \frac{Z_{n}}{\Sq^{2n} Z_{1}} - 1)
= t^{2n}/\Sq^{2n} - t^{2(n-1)}\frac{Z_1}{Z_n},
\end{gather}
so that again the righthand side lies in the ideal.  Adding together equations \eqref{eqn:telescope} and \eqref{eqn:pi conjugate Zs} shows that 
$\frac{t^{2n}}{\Sq^{2n}} -1 $ lies in the ideal, hence, $1$ lies in the ideal 
unless $t^{2n}=\Sq^{2n}$, as required.

\end{proof}

One could give a purely diagrammatic proof of the last statement in Proposition \ref{prop:good-idempotent}, following the same specialized calculations under the assignment:

\smallskip

\begin{tikzpicture}
\begin{scope}[shift={(0,0)}, scale=1]
    \node at (-1.3,.5) {$T_1 = \quad \q^{\frac 12}$}; 
    \node[bV] at (.5,1){}; \node[bV] at (0.5,0){}; 
        \node[bV] at (0,1){}; \node[bV] at (0,0){}; 
        \draw[very thick]  (0, 1) -- (0.5,0);
    \draw[white, fill= white] (0.25,.5) circle (4pt);
    \draw[very thick]  (0, 0) -- (0.5,1);
          \node[bV] at (1,1){}; \node[bV] at (1,0){}; 
          \node at (1.5,.5) {$\cdots$};
        \node[bV] at (2,1){}; \node[bV] at (2,0){}; 
         \draw[line width=2pt] (1,1) -- (1,0);
        \draw[line width=2pt] (2,1) -- (2,0);
\end{scope}

   \begin{scope}[shift={(4.8,0)}, scale=1]
      \node at (-1.25,0.5) {$Z_1 = \,  \q^{\frac{-3}{2}}$};
     \node[bV] at (.5,1){}; \node[bV] at (0.5,0){}; 
        \node[bV] at (0,1){}; \node[bV] at (0,0){}; 
          \node[bV] at (1,1){}; \node[bV] at (1,0){}; 
          \node at (1.5,.5) {$\cdots$};
        \node[bV] at (2,1){}; \node[bV] at (2,0){}; 
         \draw[line width=2pt] (.5,1) -- (.5,0);
         \draw[line width=2pt] (1,1) -- (1,0);
        \draw[line width=2pt] (2,1) -- (2,0);
        \draw[very thin, blue]  (-.2,0.15) -- (2.4,0.15);
        \draw[very thin, gray] (-.4,-0.15) -- (-.2,0.15);
        \draw[very thin, gray]   (2.4,0.15) -- (2.2,-0.15) ;
        \draw[thin, red]  (-.4,-0.15) -- (2.2,-0.15);
        \draw[thin, blue]  (-.2,1.15) -- (2.4,1.15);
        \draw[very thin, gray] (-.4,0.85) -- (-.2,1.15);
        \draw[very thin, gray]   (2.4,1.15) -- (2.2,0.85) ;
        \draw[thin, red]  (-.4,0.85) -- (2.2,0.85);

        \draw[line width=3pt, blue!50!white] (.18,.36) -- (.18,.43);
         \draw[line width=2pt] (0,1) .. controls (-0.1,.65) .. (-.2,.6) ;
         \draw[line width=2pt] (0,0) .. controls (.1,.35) .. (.2,.4) ;
        \draw[line width=2pt, red] (-.2,.57) -- (-.2,.64);
\end{scope}
   
    \begin{scope}[shift={(9,0)}, scale=1]
    \node at (-.5,0.5) {$\Spi= \quad $};
    \node[bV] at (.5,1){}; \node[bV] at (0.5,0){}; 
        \node[bV] at (0,1){}; \node[bV] at (0,0){}; 
          \node[bV] at (1,1){}; \node[bV] at (1,0){}; 
          \node at (1.5,.5) {$\cdots$};
        \node[bV] at (2,1){}; \node[bV] at (2,0){}; 
         \draw[line width=2pt] (1,0) -- (.5, 1);
        \draw[line width=2pt]  (2,0) -- (1.9, .2);
          \draw[line width=2pt]  (1,1) -- (1.1, .8);
        \draw[line width=2pt] (.5,0) -- (0, 1);

        \draw[line width=2pt] (0,0) .. controls (-.04, .4) .. (-.2, .5);
        \draw[line width=2pt] (2,1) .. controls (2.04, .6) .. (2.2, .5);
        \draw[line width=2.5pt, gray] (-.2,.45) -- (-.2,.53); 
        \draw[line width=2.5pt, gray] (2.2,.47) -- (2.2,.55); 
\end{scope}
\end{tikzpicture}

\noindent 
where we recall $\Sq = \q^{1/2},$ $ t = \q$. 
Notice that $\q^{3/2} = \q^{\frac{N^2-1}{N}} = \q^{\langle  \lambda + 2 \rho,\lambda \rangle}$ for $N=2$ and $\lambda$ the highest weight of the defining representation $V$. Hence we could have drawn $Z_1$ as having an extra loop in it as it wound around the torus instead of rescaling the picture above. 
Observe we also draw

\begin{center}

\begin{tikzpicture}
      
   \begin{scope}[shift={(4.5,0)}, scale=1]
     \node at (-1.9,0.5) {$Z_2 = \, \q^{\frac 2N}*\q^{\frac{1-N^2}{N}} $};
     \node at (4.5,0.5) { $\quad = \q \, \Spi Z_1 \Spi^{-1} = T_1 Z_1 T_1$.};
     \node[bV] at (.5,1){}; \node[bV] at (0.5,0){}; 
        \node[bV] at (0,1){}; \node[bV] at (0,0){}; 
          \node[bV] at (1,1){}; \node[bV] at (1,0){}; 
          \node at (1.5,.5) {$\cdots$};
        \node[bV] at (2,1){}; \node[bV] at (2,0){}; 
         \draw[line width=2pt] (0,1) -- (0,0);
         \draw[line width=2pt] (1,1) -- (1,0);
        \draw[line width=2pt] (2,1) -- (2,0);
        \draw[very thin, blue]  (-.2,0.15) -- (2.4,0.15);
        \draw[very thin, gray] (-.4,-0.15) -- (-.2,0.15);
        \draw[very thin, gray]   (2.4,0.15) -- (2.2,-0.15) ;
        \draw[thin, red]  (-.4,-0.15) -- (2.2,-0.15);
        \draw[thin, blue]  (-.2,1.15) -- (2.4,1.15);
        \draw[very thin, gray] (-.4,0.85) -- (-.2,1.15);
        \draw[very thin, gray]   (2.4,1.15) -- (2.2,0.85) ;
        \draw[thin, red]  (-.4,0.85) -- (2.2,0.85);

        \draw[line width=3pt, blue!50!white] (.68,.36) -- (.68,.43);
         \draw[line width=2pt] (0.5,1) .. controls (0.4,.65) .. (.3,.6) ;
         \draw[line width=2pt] (0.5,0) .. controls (.6,.35) .. (.7,.4) ;
        \draw[line width=2pt, red] (.3,.57) -- (.3,.64);
\end{scope}
\end{tikzpicture}
    
\end{center}

We do not draw $T^2 \times \{0\}$ or $T^2 \times \{1\}$ in most of the pictures, unless we want to emphasize something that differs for the annulus or rectangle. However for $Z_1$ we
do, to emphasize the $1$-strand starts from the bottom  and exits the back {\color{blue} blue} wall
after which it
enters the board at the front {\color{red} red} wall then connects to the top; although of course these two walls are identified in the torus.   For $\Spi$, a strand enters/exits from the (gray) sides.

For $\SL_2$ one may also use the following diagram for $\ek{2}$

\begin{tikzpicture}[scale = .7]
    \node at (-2,.5) {$\ek{2} = \quad \frac{-1}{\q + \q^{-1}}$};
      \draw[line width=2pt] (0,1) .. controls (.15,.55) and (.35, .55)   .. (0.5,1);
        \draw[line width=2pt] (0,0) .. controls (0.15,0.45) and (0.35, 0.45)  .. (0.5,0);
        \node[bV] at (.5,1){}; \node[bV] at (0.5,0){}; 
        \node[bV] at (0,1){}; \node[bV] at (0,0){}; 
          \node[bV] at (1,1){}; \node[bV] at (1,0){}; 
          \node at (1.5,.5) {$\cdots$};
        \node[bV] at (2,1){}; \node[bV] at (2,0){}; 
        \draw[line width=2pt] (1,1) -- (1,0);
        \draw[line width=2pt] (2,1) -- (2,0);
\end{tikzpicture}.

Hence the equation $T_1 = \q \id -(\q+\q^{-1}) \ek{2}$ (taken locally) lets us replace any skein diagram with one that is crossing-less.

\smallskip

The argument given below for the relations $X^2=Y^2=1$  was explained to us many years ago by Peter Samuelson, we thank him for the explanation.

\begin{lemma}\label{lem:1-strand}
We have an isomorphism,
\[
\phi:\K[X,Y]/\langle X^2-1, Y^2-1\rangle \to \SkAlgNn{\SL_2}{1}
\]
\end{lemma}
\begin{proof}
First, let us construct a homomorphism.  Let $x$ denote the location of $V$ on $T^2$.  Since $\pi_1(T^2,x)=\mathbb{Z}\times \mathbb{Z}$, 
we obtain a canonical homomorphism,
\[
\phi: \K[X^{\pm1},Y^{\pm1}]\to \SkAlgNn{\SL_2}{1},
\]
regarding the source as the group algebra of $\Z\times \Z$. In order to deduce the relations $X^2=Y^2=1$, we compute:

\begin{tikzpicture}[line width=1pt]
\begin{scope}[shift={(0,0)}, scale=1]
    \Under[.5,.5][1.5,.5]
        \Over[1,0][1,1]
        \node[bV] at (1,1){}; \node[bV] at (1,0){}; 

    \draw[very thin,  blue] (.4,-0.1) -- (1.4,-0.1);
    \draw[very thin,   blue](1.6,1.1) -- (.6, 1.1)  ;
    \draw[very thin,  blue]  (1.6,0.1) -- (.6, 0.1) ;
    \draw[very thin,   blue](.4,0.9) -- (1.4,0.9) ;
    
    \draw[very thin] (1.4,-0.1) -- (1.6,0.1) ;
    \draw[very thin]  (.6, 1.1) -- (.4,0.9) ;
    \draw[very thin]  (.6, 0.1) -- (.4,-0.1);
    \draw[very thin](1.4,0.9) -- (1.6,1.1) ;
\end{scope}
\begin{scope}[shift={(2,0)}, scale=1]
   \Under[1,0][1,1]
        \Over[.5,.5][1.5,.5]
        \node[bV] at (1,1){}; \node[bV] at (1,0){}; 

    \draw[very thin,  blue] (.4,-0.1) -- (1.4,-0.1);
    \draw[very thin,   blue](1.6,1.1) -- (.6, 1.1)  ;
    \draw[very thin,  blue]  (1.6,0.1) -- (.6, 0.1) ;
    \draw[very thin,   blue](.4,0.9) -- (1.4,0.9) ;
    
    \draw[very thin] (1.4,-0.1) -- (1.6,0.1) ;
    \draw[very thin]  (.6, 1.1) -- (.4,0.9) ;
    \draw[very thin]  (.6, 0.1) -- (.4,-0.1);
    \draw[very thin](1.4,0.9) -- (1.6,1.1) ;
\end{scope}
\begin{scope}[shift={(0,0)}]
       \node at (0,.5) {$0=$};
         \node at (2,.5) {$-$};
        \node at (4.8,.5) {$ = \, (\q^{\frac 12} - \q^{-\frac12})  $};
         \node[scale=2.5] at (6.25,.5) {$( $};
          \node[scale=2.5] at (9.75,.5) {$)$};
        \node at (8,.5) {$-$};
     \node at (12,.5) {$ = \, (\q^{\frac 12} - \q^{-\frac 12})(X-X^{-1}),$};
\end{scope}
\begin{scope}[shift={(6,0)}, scale=1]     
        \draw[line width=2pt] (.5,.5) .. controls (.8,.45)   .. (1,0);
        \draw[line width=2pt] (1,1) .. controls (1.2,0.55)  .. (1.5,.5);
        \node[bV] at (1,1){}; \node[bV] at (1,0){}; 

    \draw[very thin,  blue] (.4,-0.1) -- (1.4,-0.1);
    \draw[very thin,   blue](1.6,1.1) -- (.6, 1.1)  ;
    \draw[very thin,  blue]  (1.6,0.1) -- (.6, 0.1) ;
    \draw[very thin,   blue](.4,0.9) -- (1.4,0.9) ;
    
    \draw[very thin] (1.4,-0.1) -- (1.6,0.1) ;
    \draw[very thin]  (.6, 1.1) -- (.4,0.9) ;
    \draw[very thin]  (.6, 0.1) -- (.4,-0.1);
    \draw[very thin](1.4,0.9) -- (1.6,1.1) ;
\end{scope}
\begin{scope}[shift={(8,0)}, scale=1]     
        \draw[line width=2pt] (.5,.5) .. controls (.8,.55)   .. (1,1);
        \draw[line width=2pt] (1,0) .. controls (1.2,0.45)  .. (1.5,.5);
        \node[bV] at (1,1){}; \node[bV] at (1,0){}; 

    \draw[very thin,  blue] (.4,-0.1) -- (1.4,-0.1);
    \draw[very thin,   blue](1.6,1.1) -- (.6, 1.1)  ;
    \draw[very thin,  blue]  (1.6,0.1) -- (.6, 0.1) ;
    \draw[very thin,   blue](.4,0.9) -- (1.4,0.9) ;
    
    \draw[very thin] (1.4,-0.1) -- (1.6,0.1) ;
    \draw[very thin]  (.6, 1.1) -- (.4,0.9) ;
    \draw[very thin]  (.6, 0.1) -- (.4,-0.1);
    \draw[very thin](1.4,0.9) -- (1.6,1.1) ;
\end{scope}
\end{tikzpicture}

And similarly for $Y$, going the other direction around the torus.

To see the map is surjective, we note that in the standard projection we can always apply Kauffman skein relations to reduce any tangle to one which has no crossings.  Any such tangle (with only one top and bottom endpoint, that must then be connected)
is clearly isotopic to the disjoint (and unlinked) union of a braid, and some number of unknots.  The skein relations reduce unknots to scalars, and we left only with the braid components.

Let us now show that $\phi$ is an injection, hence an isomorphism.  We note that the source of $\phi$, hence (by surjectivity) the target of $\phi$ are both commutative; hence they are canonically isomorphic as vector spaces to their zeroth Hochschild homologies.  As a consequence of Part (3) of Theorem \ref{thm:SL2}, the zeroth Hochschild homology of $\SkAlgNn{\SL_2}{1}$ is the unique contribution to the skein module of $T^3$ in degree $(\ast,\ast,1)$ for the natural $(\mathbb{Z}/2\Z)^{ \oplus 3}$ grading: by definition $\HHz(\End(\Dist))$ lives only in degrees $(\ast,\ast,0)$.  However, by \cite{mclendon-traces},
we have $\HHz (\SkAlgN{\SL_N}(T^2))$ is 5-dimensional, comprised of 2-dimensions in degree $(0,0)$, and one-dimension in each degree $(1,0)$, $(0,1)$, $(1,1)$.  Noting that the mapping class group $\SL_3(\Z)$
acts transitively on  $(\mathbb{Z}/2\Z)^{ \oplus 3}\setminus \{(0,0,0)\}$, we see the remaining degrees cannot be zero.  See Figure \ref{fig:cube}.  If $\phi$ had a kernel, then we would have a surjection from a vector space of smaller dimension, which is obviously a contradiction.

\begin{figure}
    \centering
 
\begin{tikzpicture}
                [cube/.style={very thick,black},
                        grid/.style={very thin,gray},
                        axis/.style={->,blue,thick}]

 \begin{scope}[shift={(0,0)}]

        \draw[cube,blue] (0,2,0) -- (2,2,0) -- (2,0,0) ;
        \draw[cube, very thick, dotted, blue] (0,0,0) -- (2,0,0);
        \draw[cube, very thick, dotted, blue] (0,0,0) -- (0,2,0);
        \draw[cube, blue] (0,2,0) -- (2,2,0) -- (2,0,0) ;
        \draw[cube, blue] (0,0,2) -- (0,2,2) -- (2,2,2) -- (2,0,2) -- cycle;
        
        \draw[cube, very thick, dotted, blue] (0,0,0) -- (0,0,2);
        \draw[cube, blue] (0,2,0) -- (0,2,2);
        \draw[cube, blue] (2,0,0) -- (2,0,2);
        \draw[cube, blue] (2,2,0) -- (2,2,2);
        
        \draw (0,0,0)  node[scale = 0.4, anchor=north]{$\left[ \begin{array}{c} 0\\0\\0 \end{array} \right]
$};
        \draw (2,0,0)  node[scale = 0.4, anchor=west]{$\left[ \begin{array}{c} 0\\1\\0 \end{array} \right]
$};
        \draw (0,2,0)  node[scale = 0.4, anchor=south]{$\left[ \begin{array}{c} 0\\0\\1 \end{array} \right]
$};
        \draw (0,0,2)  node[scale = 0.4, anchor=north]{$\left[ \begin{array}{c} 1\\0\\0 \end{array} \right]
$};

\end{scope}

 \begin{scope}[shift={(6,0)}]
        \draw[cube,blue] (0,2,0) -- (2,2,0) -- (2,0,0) ;
        \draw[cube, very thick, dotted, blue] (0,0,0) -- (2,0,0);
        \draw[cube, very thick, dotted, blue] (0,0,0) -- (0,2,0);
        \draw[cube, blue] (0,2,0) -- (2,2,0) -- (2,0,0) ;
        \draw[cube, blue] (0,0,2) -- (0,2,2) -- (2,2,2) -- (2,0,2) -- cycle;
        
        \draw[cube, very thick, dotted, blue] (0,0,0) -- (0,0,2);
        \draw[cube, blue] (0,2,0) -- (0,2,2);
        \draw[cube, blue] (2,0,0) -- (2,0,2);
        \draw[cube, blue] (2,2,0) -- (2,2,2);
        
        \draw (0,0,0)  node[scale = 1.4, anchor=north]{$2 $};
        \draw (2,0,0)  node[scale = 1.4, anchor=north]{$1$};
        \draw (0,2,0)  node[scale = 1.4, anchor=south]{$ 1$};
        \draw (0,0,2)  node[scale = 1.4, anchor=north]{$ 1$};
        \draw (2,2,0)  node[scale = 1.4, anchor=south]{$1$};
        \draw (0,2,2)  node[scale = 1.4, anchor=south]{$ 1$};
        \draw (2,0,2)  node[scale = 1.4, anchor=north]{$ 1$};
        \draw (2,2,2)  node[scale = 1.4, anchor=south]{$ 1$};
\end{scope}
\end{tikzpicture}
   \caption{We label $\left[\begin{array}{c} x\\y\\z \end{array}\right] \in (\Z/2\Z)^{\oplus 3}$ with $\gcd(x,y,z,2)$.  The bottom face of the resulting cube corresponds to $\HHz(\SkAlgN{2}(T^2))$, and remaining entries on the top face are determined by compatibility with rotation in $\SL_{3}(\Z/2\Z)$.}
    \label{fig:cube}
\end{figure}

Finally, we note that $\K[X,Y]/\langle X^2-1, Y^2-1\rangle$ is isomorphic to $\K^{\times 4}$ as an algebra, hence its category of modules is $\Vect_\K^4$ as claimed.
\end{proof}

For relative skein algebra $\SkAlgNn{\SL_2}{2}$, the topology of the torus makes it easy to see $\Spi^2$ is central.  The skein relations also make it easy to see we then have a homomorphism 
$\Schmaham{2} \to \SkAlgNn{\SL_2}{2}$.  Corollary \ref{cor:SkAlgN} tells us we expect an isomorphism $\DAHA{2}{2} \to \SkAlgNn{\SL_2}{2}$, so in particular $\Spi^2 = 1$. 
Below we give a diagrammatic proof of this relation, relying on Proposition \ref{prop:good-idempotent}, which also generalizes to larger $N$.

\begin{proposition} \label{prop:SL2 pi squared}

    In $\SkAlgNn{\SL_2}{2}$, the relation $\Spi^2=1$ holds.
  
\end{proposition}
\begin{proof}
    The diagrammatic proof below shows that $\ek{2} (\Spi^2-1) \ek{2} = 0$. Recall that for $G=\SL_2$ we depict $-(\q+1/\q) \ek{2}$ by \cupcap.  So the picture below depicts
    $ (-\q-1/\q) \ek{2} * \Spi^2 * (-\q-1/\q) \ek{2} = (\q+1/\q)^2 \ek{2}$. 
    Thus by Proposition \ref{prop:good-idempotent} and Proposition \ref{prop:central}, we conclude $\Spi^2=1$.
    
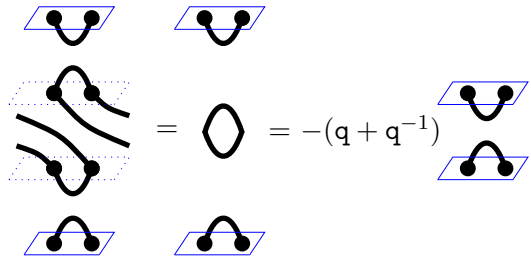
\begin{center}
\begin{tikzpicture}[line width=2pt]
\begin{scope}[shift={(2,2)}, scale=1]
        \draw[line width=2pt] (0,1) .. controls (.15,.55) and (.35, .55)   .. (0.5,1);
        \draw[line width=2pt] (0,0) .. controls (0.15,0.45) and (0.35, 0.45)  .. (0.5,0);
        \node[bV] at (.5,1){}; \node[bV] at (0.5,0){}; 
        \node[bV] at (0,1){}; \node[bV] at (0,0){}; 
        \draw[very thin, blue] (-.4,0.85) -- (-.2,1.15) -- (0.8,1.15) -- (0.6,0.85) ;
        \draw[very thin, blue] (-.4,0.85) --  (0.6,0.85) ;
\end{scope}
\begin{scope}[shift={(2,0)}, scale=1]
        \draw[line width=2pt] (0,1) .. controls (.15,.55) and (.35, .55)   .. (0.5,1);
        \draw[line width=2pt] (0,0) .. controls (0.15,0.45) and (0.35, 0.45)  .. (0.5,0);
        \draw[very thin, blue] (-.4,-0.15) -- (-.2,0.15) -- (0.8,0.15) -- (0.6,-0.15) ;
        \node[bV] at (.5,1){}; \node[bV] at (0.5,0){}; 
        \node[bV] at (0,1){}; \node[bV] at (0,0){}; 
        \draw[very thin, blue] (-.4,-0.15) --  (0.6,-0.15) ;
\end{scope}

\begin{scope}[shift={(2,1)}, scale=1]     
        \draw[line width=2pt] (-.5,.3) .. controls (-.2,.20)   .. (0,0);
        \draw[line width=2pt]  (-.5, .7) .. controls  (0.1,.45) ..  (0.5,0);
        \draw[line width=2pt] (1.0,.7).. controls (0.7,0.8)  ..  (0.5,1);
        \draw[line width=2pt] (1.0, .3) .. controls (.4,.55)   ..  (0,1);
        \draw[thin, blue, dotted] (-.6,-0.15) -- (-.4,0.15) -- (1.0,0.15) -- (0.8,-0.15) ;
        \node[bV] at (.5,1){}; \node[bV] at (0.5,0){}; 
        \node[bV] at (0,1){}; \node[bV] at (0,0){}; 
        \draw[thin, blue, dotted] (-.6,-0.15) --  (0.8,-0.15) ;
        \draw[thin, blue, dotted] (-.6,0.85) -- (-.4,1.15) -- (1.0,1.15) -- (0.8,0.85) ;
        \draw[thin, blue, dotted] (-.6,0.85) --  (0.8,0.85) ;
\end{scope}

\begin{scope}[shift={(4,2)}, scale=1]
        \draw[line width=2pt] (0,1) .. controls (.15,.55) and (.35, .55)   .. (0.5,1);
        \node[bV] at (.5,1){};
        \node[bV] at (0,1){};
        \draw[very thin, blue] (-.4,0.85) -- (-.2,1.15) -- (0.8,1.15) -- (0.6,0.85) ;
        \draw[very thin, blue] (-.4,0.85) --  (0.6,0.85) ;
\end{scope}

\begin{scope}[shift={(4,.5)}, scale=1]
     \draw[line width=2pt] (0,0+0.97) .. controls (0.15,0.45+1) and (0.35, 0.45+1) .. (0.5,0+0.97);
      \draw[line width=2pt] (0,1) .. controls (.15,.55) and (.35, .55)   .. (0.5,1);
\node at (-0.5,1) {=};
\end{scope}

\begin{scope}[shift={(4,0)}, scale=1]     
        \draw[line width=2pt] (0,0) .. controls (0.15,0.45) and (0.35, 0.45)  .. (0.5,0);
        \draw[very thin, blue] (-.4,-0.15) -- (-.2,0.15) -- (0.8,0.15) -- (0.6,-0.15) ;
\node[bV] at (0.5,0){}; 
\node[bV] at (0,0){}; 
        \draw[very thin, blue] (-.4,-0.15) --  (0.6,-0.15) ;
\end{scope}

\begin{scope}[shift={(7.5,1)}, scale=1]
\node at (-1.5,.5) {=  $-(\q+\q^{-1})$};
        \draw[line width=2pt] (0,1) .. controls (.15,.55) and (.35, .55)   .. (0.5,1);
        \draw[line width=2pt] (0,0) .. controls (0.15,0.45) and (0.35, 0.45)  .. (0.5,0);
        \draw[very thin, blue] (-.4,-0.15) -- (-.2,0.15) -- (0.8,0.15) -- (0.6,-0.15) ;
        \node[bV] at (.5,1){}; \node[bV] at (0.5,0){}; 
        \node[bV] at (0,1){}; \node[bV] at (0,0){}; 
        \draw[very thin, blue] (-.4,-0.15) --  (0.6,-0.15) ;
        \draw[very thin, blue] (-.4,0.85) -- (-.2,1.15) -- (0.8,1.15) -- (0.6,0.85) ;
        \draw[very thin, blue] (-.4,0.85) --  (0.6,0.85) ;
\end{scope}

\end{tikzpicture}
\end{center}

Observe that if we rotate the pictures 90 degrees  and rescale appropriately, we will recover that 
$Z_1 Z_2 = \Zprod$. 
\end{proof}

\section{Applications to Skein Theory: The cases $G = \GLN, \SLN$}
\label{sec:apps to skeins}

The kind of direct computation in Section \ref{sec:SL2} is not practical for higher $N$, so we instead rely on the representation theory of the DAHA developed in Section \ref{sec:DAHA stuff}.   In this section we explain how to apply those results to the relative skein algebra, then to the skein category of $T^2$, and eventually to the skein module of $T^3$. After fixing conventions for parameters, and recalling the setup of elliptic Schur-Weyl duality, we proceed to prove all theorems stated in the introduction, as corollaries or special cases of the results proved in the preceding sections.

\subsection{Elliptic Schur-Weyl duality}
Each of the finite, affine, and double affine Hecke algebras arise naturally in skein theory, simply because they are presented as quotients of the group algebra of the appropriate braid group by quadratic relations, and these quadratic relations hold in the $\GLN$ and $\SLN$ skein categories.  Let us now make this more precise and fix some notation.

\begin{notation}
Throughout this section we fix a positive integer $N$ and a field $\cK$ containing
$\Q(\q^{\frac 1N})$.
We specialise the various parameters $t, \q, q, \Sq, \Zprod$ in the DAHA and Schmaha as follows:

\label{not:parameters-all-together}
 
\begin{itemize}
\item The quantum group parameter is $\q = (\q^{\frac{1}{N}})^N$.

\item The quadratic parameter in Hecke algebras is $t = \q= (\q^{\frac{1}{N}})^N$, so in particular $t^{\frac 1N} = \q^{\frac 1N}$.  Taking $t= \q $ ensures compatibility with Schur-Weyl duality.
\item The loop parameter for the rank $n$ $\GL$ DAHA $\HG$ is specialized $q = t^{-2n/N} = \q^{-2n/N}$.  
\item The loop parameter for the rank $n$ Schmaha  $\Schmaha{n}$ is specialized $\Sq = t^{1/N} =\q^{1/N}$.
The loop parameter for the rank $n$ $\SL$ DAHA  $\HS$ is specialized $\Sq = t^{1/N} =\q^{1/N}$, and further we take $\Zprod = \Sq^{n(n-N^2)} = t^{n(n/N-N)}$. 
\end{itemize}
To lighten notation we will use the abbreviations $\DAHAN{n}$ or $\SchmahaN{n}$ for the above specialisation of parameters, or simply
$\DAHA{N}{n}$, $\Schmahan$ when the group $\GLN$ or $\SLN$ 
is clear from context.

\end{notation}

 \begin{proposition}[Schur-Weyl duality homomorphism]
Let $G= \GLN$ or $\SLN$.  The natural homomorphisms from the disk, annulus, and torus braid groups to the respective skein algebra, each descend to the finite, affine, and double affine Hecke algebras, 
to give algebra homomorphisms,
\begin{align*}
\SWfin: \Hf_n &\longrightarrow \SkAlgNn{G}{n}(\DD)\\
\SWaff: \H_n &\longrightarrow \SkAlgNn{G}{n}(\Ann)\\
\SWdaff: \HG &\longrightarrow \SkAlgNn{\GLN}{n}(T^2)\\
\SWdaff: \HSweak &\longrightarrow \SkAlgNn{\SLN}{n}(T^2)
\end{align*}
with $q, \Sq, t$ specialized as above.
\end{proposition}

Let $\eN\in \Hf_N$ denote the sign idempotent as in \ref{eq:eNj}.  In \cite{GJV} we have proved the following ``Schur-Weyl duality" isomorphism:
\begin{theorem}[\cite{GJV}, Corollary 1.12]
Let $G=\GLN$ or $\SLN$, and let $n=N$ above. Then the restricted homomorphism,
\[
\eN\SWdaff\eN: \eN\HH_N(N)\eN \to \SkAlgN{G}(T^2)
\]
is an isomorphism.
\end{theorem}


Theorem \ref{mainthm:morita-DAHA} follows as a special case $j=1$ of the following theorem, which is itself a specialisation of the parameters of Theorem \ref{thm:idempotent}. 

\begin{theorem}\label{mainthm:morita-DAHA-eNj} Suppose that the 
quantum parameter $\q$ is is not a root of unity.  Let $n, N, j \in \Z_{> 0}$. Then for $j\leq n/N$ we have   
\[
\DAHA{N}{n}\cdot \eNj\cdot \DAHA{N}{n} = \DAHA{N}{n},
\quad \textrm{ and } \quad\Schmahan\cdot \eNj\cdot \Schmahan = \Schmahan.\]
 Consequently, we also obtain Morita equivalences, between $\DAHA{N}{n}$ and 
 $\eNj  \DAHA{N}{n} \eNj$, and between $\Schmahan$ and 
 $\eNj  \Schmahan \eNj$.
\end{theorem}

Theorem \ref{thm:relskeins} follows by applying Proposition \ref{prop:idems} in the case $\ee = \eN$, $B= \SkAlgNn{G}{n}$ and $\varphi=\SWdaff$, with $A = \DAHAN{n}$ for $G=\GLN$, or $A = \SchmahaN{n}$ for $G=\SLN$.  


\medskip 

Corollary \ref{cor:SkAlgN} now follows: the Schur-Weyl homomorphism $\SWdaff$ is a homomorphism of $\DAHAN{n}$-bimodules, which becomes an isomorphism upon application of the idempotent $\ek{N}$.  This means the original homomorphism $\SWdaff$ is an isomorphism, since $\ek{N}$ is conservative.

\medskip 

Corollary \ref{cor:SkAlgGLNn} now follows: by Proposition \ref{prop:idems}, the surjectivity of the Schur-Weyl map $\SWdaff$ is equivalent to the surjectivity of $\eNk\SWdaff\eNk$.
We show it is surjective (on the level of vector spaces) as follows.
Note $\eNk \SkAlg_{G, kN}(T^2) \eNk \cong \SkAlg_{G,0}(T^2)$ and we have shown above $\eN\DAHA{N}{N}\eN$ surjects to the latter.  Diagrammatically, this means we can represent any tangle in $T^2 \times I$, modulo skein relations, as a linear combination  of braids on $N$ strands that are all then ``capped off" with a sign idempotent at $T^2 \times \{0\}$ and $T^2 \times \{1\}$.   To show surjectivity, it
suffices to represent such a capped off braid as a braid on $kN$ strands, where strands are now capped off in $k$ groups of $N$.  We merely cut out a little disk $\D \subseteq T^2$ so that $\D \times I$ does not meet the braids, and then fill the disk with $(k-1)N$ strands.  Now we cap off these $(k-1)N$ strands using $\ek{N^{k-1}}$ in the obvious way, which at worst results in the original picture times a scalar (and at best augments the special $\det_q^{-k}$ strand in the $\GLN$ case appropriately). 
Note that this does not imply that $\eN \DAHA{N}{N} \eN$ is a subalgebra of $\eNk \DAHA{N}{kN} \eNk$ as the above map is not a homomorphism.

Theorem \ref{mainthm:generators} now follows, with a little more work than the preceding results. 
 Recall that $\Repq(G)$ is graded by the character lattice of the center of $G$, which is $\mathbb{Z}$ for $\GLN$ and $\mathbb{Z}/N$ for $\SLN$: 
 given an indecomposable object in $\Repq(G)$, all of its weights lie in the same coset of the root lattice, hence we can use this to alternatively define its degree to be the appropriate coset of $P/Q$.  For $\GLN$ we identify $P/Q$ with $\Z$ and for $\SLN$ with $\Z/N \Z$.
 
\begin{lemma}\label{lem:GLdeg}
Let $G=\GLN$, and let $X\in\Repq(G)$ be simple.  Then the object $\Dist_{X}$ is zero in $\SkCat_G(T^2)$ unless $X$ has degree zero. 
\end{lemma}
\begin{proof}

In the category $\Repq (\GLN)$ we have that the braiding $\det_q(V) \otimes X \to X \otimes \det_q(V)$ differs from its inverse by a factor of $\q^{\deg(X)}$. In the skein category, consider the relation

\begin{center}

\begin{tikzpicture}[line width=1pt]

       \node at (-.4,.5) {$0=$};
       \node at (2,.5) {$-$};
       \node at (5,.5) {$= \, (1-\q^{\deg(X)})$};
       \node at (9.2,.5) {$=(1-\q^{\deg(X)})$};

\begin{scope}[shift={(0,0)}, scale=1]
        \node[bV] at (1,1){}; \node[bV] at (1,0){}; 
    \draw[very thin,  blue] (.4,-0.1) -- (1.4,-0.1);
    \draw[very thin,   blue](1.6,1.1) -- (.6, 1.1)  ;
    \draw[very thin,  blue]  (1.6,0.1) -- (.6, 0.1) ;
    \draw[very thin,   blue](.4,0.9) -- (1.4,0.9) ;
    
    \draw[very thin] (1.4,-0.1) -- (1.6,0.1) ;
    \draw[very thin]  (.6, 1.1) -- (.4,0.9) ;
    \draw[very thin]  (.6, 0.1) -- (.4,-0.1);
    \draw[very thin](1.4,0.9) -- (1.6,1.1) ;
    \draw[thick, red]  (.5, .4) -- (1.5,.4);
    \draw[thick, green!50!black]  (.5, .6) -- (1.5,.6);
    \draw[white, fill= white] (1,.4) circle (2pt);
    \draw[white, fill= white] (1,.6) circle (2pt);
    \draw[very thick]  (1, 0) -- (1,1);
\end{scope}

\begin{scope}[shift={(2,0)}, scale=1]
        \node[bV] at (1,1){}; \node[bV] at (1,0){}; 
    \draw[very thin,  blue] (.4,-0.1) -- (1.4,-0.1);
    \draw[very thin,   blue](1.6,1.1) -- (.6, 1.1)  ;
    \draw[very thin,  blue]  (1.6,0.1) -- (.6, 0.1) ;
    \draw[very thin,   blue](.4,0.9) -- (1.4,0.9) ;
    
    \draw[very thin] (1.4,-0.1) -- (1.6,0.1) ;
    \draw[very thin]  (.6, 1.1) -- (.4,0.9) ;
    \draw[very thin]  (.6, 0.1) -- (.4,-0.1);
    \draw[very thin](1.4,0.9) -- (1.6,1.1) ;
    \draw[thick, green!50!black]  (.5, .6) -- (1.5,.6);
    \draw[white, fill= white] (1,.6) circle (2pt);
    \draw[very thick]  (1, 0) -- (1,1);
    \draw[white, fill= white] (1,.4) circle (2pt);
    \draw[thick, red]  (.5, .4) -- (1.5,.4);
\end{scope}
\begin{scope}[shift={(6,0)}, scale=1]
        \node[bV] at (1,1){}; \node[bV] at (1,0){}; 
    \draw[very thin,  blue] (.4,-0.1) -- (1.4,-0.1);
    \draw[very thin,   blue](1.6,1.1) -- (.6, 1.1)  ;
    \draw[very thin,  blue]  (1.6,0.1) -- (.6, 0.1) ;
    \draw[very thin,   blue](.4,0.9) -- (1.4,0.9) ;
    
    \draw[very thin] (1.4,-0.1) -- (1.6,0.1) ;
    \draw[very thin]  (.6, 1.1) -- (.4,0.9) ;
    \draw[very thin]  (.6, 0.1) -- (.4,-0.1);
    \draw[very thin](1.4,0.9) -- (1.6,1.1) ;
    \draw[thick, red]  (.5, .4) -- (1.5,.4);
    \draw[thick, green!50!black]  (.5, .6) -- (1.5,.6);
    \draw[white, fill= white] (1,.4) circle (2pt);
    \draw[white, fill= white] (1,.6) circle (2pt);
    \draw[very thick]  (1, 0) -- (1,1);
\end{scope}
\begin{scope}[shift={(10,0)}, scale=1]
        \node[bV] at (1,1){}; \node[bV] at (1,0){}; 
    \draw[very thin,  blue] (.4,-0.1) -- (1.4,-0.1);
    \draw[very thin,   blue](1.6,1.1) -- (.6, 1.1)  ;
    \draw[very thin,  blue]  (1.6,0.1) -- (.6, 0.1) ;
    \draw[very thin,   blue](.4,0.9) -- (1.4,0.9) ;
    
    \draw[very thin] (1.4,-0.1) -- (1.6,0.1) ;
    \draw[very thin]  (.6, 1.1) -- (.4,0.9) ;
    \draw[very thin]  (.6, 0.1) -- (.4,-0.1);
    \draw[very thin](1.4,0.9) -- (1.6,1.1) ;
    \draw[very thick]  (1, 0) -- (1,1);
\end{scope}

\begin{scope}[shift={(0,-2)}, scale=1]
    \draw[thick, red]  (1, 1) -- (1.5,1);
    \draw[thick, green!50!black]  (1, .5) -- (1.5,.5);
    \draw[very thick]  (4, 0.5) -- (4,1);
       \node[red] at (2.5,1) {$=\det_{\q}$};
       \node[green!50!black] at (2.5,.5) {$=\det_{\q}^{-1}$};
       \node[black] at (4.6,0.75) {$=X$};

\end{scope}

\end{tikzpicture}
\end{center}
\end{proof}

Let $G=\GLN$ or $\SLN$, and let $M\in\SkCat_{G}(T^2)\modu$ be an arbitrary object.  Because the tensor powers $V^{\ot n}$ of the defining representation collectively generate $\Repq(G)$, it follows that the objects $\Dist_{V^{\ot n}}$, for $n\geq 0$, collectively generate the skein category.  In other words there exists some non-negative integer $n$ with
\[
\Hom_{T^2}(\Dist_{V^{\ot n}},M) \neq 0.
\]
If it happens that $n\leq N-1$, then we are done.  Otherwise, we may apply Theorem \ref{thm:relskeins}, to conclude that
\[\Hom_{T^2}(\Dist_{V^{\ot (n-N)}},M)  = \Hom_{T^2}(\Dist_{\eN\cdot V^{\ot n}},M) =
\eN\cdot \Hom_{T^2}(\Dist_{V^{\ot n}},M) \neq 0.
\]
In this way we can keep reducing the power $n$ until it lies in the range $0,\ldots N-1$. 
In the $G=\SLN$ case, we  note that the  generators $\Dist_{V^{\ot m}}$ for $m=0, 1,\ldots, N-1$ are pairwise orthogonal for the $\Hom$ pairing, due to the grading by degree. 
Further, in  $G=\GLN$ case, Lemma \ref{lem:GLdeg} implies that $\Dist_{V^{\ot m}}=0$ for $m=1,2,\ldots, N-1$, leaving only $\Dist$.  This proves Theorem \ref{mainthm:generators}.

Theorem \ref{mainthm:dimensionsGLN} (the $\GLN$ case) follows easily from Theorems \ref{thm: HHz of GL skein algebra} and by Theorem \ref{mainthm:generators}: since  $\Dist$ is a compact projective generator, the zeroth Hochschild homology of the category is equal to the zeroth Hochschild homology of $\End(\Dist)$, which is the skein algebra.

Theorem \ref{mainthm:dimensionsSLN} (the $\SLN$ case) also follows, but with a little more work, due to the fact that in the $\SLN$ case, $\Dist$ is not a generator, hence the Hochschild homology of the skein algebra yields only the degree $(\ast,\ast,0)$ part of $\SkMod_{\SLN} T^3$, with respect to the natural grading by $H_1(T^3;\Z/N\Z) = (\Z/N\Z)^{\oplus 3}$. Stated more concretely, the Hochschild homology of the skein algebra involves only those skeins on $T^3$ which do not wrap on the $S^1$ factor under the decomposition $T^3=T^2\times S^1$.

In order to determine the remaining degrees, we exploit the action of the mapping class group $\operatorname{Map}(T^3)=\SL_3(\mathbb{Z})$, which acts on the skein module compatibly with the $H_1$-grading by $(\Z/N\Z)^{\oplus 3}$, via the surjection $\SL_3(\mathbb{Z})\to\SL_3(\Z/N\Z)$. 
Indeed, an arbitrary element of $(\Z/N\Z)^{\oplus 3}$ is a translate by $\SL_3(\Z/N\Z)$ of a weight of the  form $(d,0,0)$, for some divisor $d$ of $N$.  Hence, we have:
\[
\dim\SkMod_{\SLN}(T^3) = \sum_{d | N} \cP(d)J_3(N/d).
\]
Together with Lemma \ref{lemma:totient} this completes the proof of Theorem \ref{mainthm:dimensionsSLN}.

\begin{proposition} \label{prop:SLN pi N}

    In $\SkAlgNn{\SLN}{N}$, the relation $\Spi^N=1$ holds.
  
\end{proposition}
\begin{proof}
    The diagrammatic proof below shows that $\ek{N} (\Spi^N-1) \ek{N} = 0$, very similar to our proof of Proposition \ref{prop:SL2 pi squared}. Recall that for $G=\SL_N$ we introduce the picture

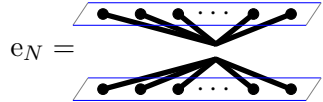
\begin{center}

    \begin{tikzpicture}
 \begin{scope}[shift={(0,0)}, scale=1]
          \node at (-.8,.5) {$\eN = $};
     \node[bV] at (.5,1){}; \node[bV] at (0.5,0){}; 
        \node[bV] at (0,1){}; \node[bV] at (0,0){}; 
          \node[bV] at (1,1){}; \node[bV] at (1,0){}; 
          \node at (1.5,1) {$\cdots$};
          \node at (1.5,0) {$\cdots$};
        \node[bV] at (2,1){}; \node[bV] at (2,0){}; 
        \node[bV] at (2.5,1){}; \node[bV] at (2.5,0){}; 
         \draw[line width=2pt] (0,1) -- (1.5,0.6);
         \draw[line width=2pt] (.5,1) -- (1.5,0.6);
         \draw[line width=2pt] (1,1) -- (1.5,0.6);
        \draw[line width=2pt] (2,1) -- (1.5,0.6);
        \draw[line width=2pt] (2.5,1) -- (1.5,0.6);
         \draw[line width=2pt] (0,0) -- (1.5,0.4);
         \draw[line width=2pt] (.5,0) -- (1.5,0.4);
         \draw[line width=2pt] (1,0) -- (1.5,0.4);
        \draw[line width=2pt] (2,0) -- (1.5,0.4);
        \draw[line width=2pt] (2.5,0) -- (1.5,0.4);

        \node[V] at (1.5,.4){}; \node[V] at (1.5,0.6){}; 
        \draw[very thin, blue]  (-.2,0.15) -- (2.9,0.15);
        \draw[very thin, gray] (-.4,-0.15) -- (-.2,0.15);
        \draw[very thin, gray]   (2.9,0.15) -- (2.7,-0.15) ;
        \draw[thin, blue]  (-.4,-0.15) -- (2.7,-0.15);
        \draw[thin, blue]  (-.2,1.15) -- (2.9,1.15);
        \draw[very thin, gray] (-.4,0.85) -- (-.2,1.15);
        \draw[very thin, gray]   (2.9,1.15) -- (2.7,0.85) ;
        \draw[thin, blue]  (-.4,0.85) -- (2.7,0.85);

\end{scope}
\end{tikzpicture}
    
\end{center}
 
 \noindent
 not rescaling it by $\pm [N]_{\q}$ as is the custom in the case $N=2$.
    Hence we get the local  relation

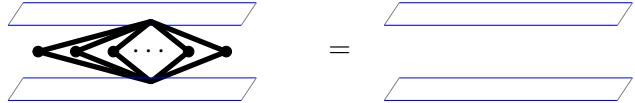
\begin{center}
\begin{tikzpicture}

 \begin{scope}[shift={(5,-0.5)}, scale=1]
     \node[bV] at (.5,1){};  
        \node[bV] at (0,1){}; 
          \node[bV] at (1,1){}; 
          \node at (1.5,1) {$\cdots$};
        \node[bV] at (2,1){}; 
        \node[bV] at (2.5,1){}; 
         \draw[line width=2pt] (0,1) -- (1.5,0.6);
         \draw[line width=2pt] (.5,1) -- (1.5,0.6);
         \draw[line width=2pt] (1,1) -- (1.5,0.6);
        \draw[line width=2pt] (2,1) -- (1.5,0.6);
        \draw[line width=2pt] (2.5,1) -- (1.5,0.6);
\node[V] at (1.5,0.6){}; 
\end{scope}

 \begin{scope}[shift={(5,0)}, scale=1]
        \draw[thin, blue]  (-.2,1.15) -- (2.9,1.15);
        \draw[very thin, gray] (-.4,0.85) -- (-.2,1.15);
        \draw[very thin, gray]   (2.9,1.15) -- (2.7,0.85) ;
        \draw[thin, blue]  (-.4,0.85) -- (2.7,0.85);
        \draw[very thin, blue]  (-.2,0.15) -- (2.9,0.15);
        \draw[very thin, gray] (-.4,-0.15) -- (-.2,0.15);
        \draw[very thin, gray]   (2.9,0.15) -- (2.7,-0.15) ;
        \draw[thin, blue]  (-.4,-0.15) -- (2.7,-0.15);

\end{scope}

 \begin{scope}[shift={(5,0.5)}, scale=1]
         \draw[line width=2pt] (0,0) -- (1.5,0.4);
         \draw[line width=2pt] (.5,0) -- (1.5,0.4);
         \draw[line width=2pt] (1,0) -- (1.5,0.4);
        \draw[line width=2pt] (2,0) -- (1.5,0.4);
        \draw[line width=2pt] (2.5,0) -- (1.5,0.4);

        \node[V] at (1.5,.4){};
\end{scope}

 \begin{scope}[shift={(10,0)}, scale=1]
          \node at (-.8,.5) {$ = \quad  $};
        \draw[thin, blue]  (-.2,1.15) -- (2.9,1.15);
        \draw[very thin, gray] (-.4,0.85) -- (-.2,1.15);
        \draw[very thin, gray]   (2.9,1.15) -- (2.7,0.85) ;
        \draw[thin, blue]  (-.4,0.85) -- (2.7,0.85);
        \draw[very thin, blue]  (-.2,0.15) -- (2.9,0.15);
        \draw[very thin, gray] (-.4,-0.15) -- (-.2,0.15);
        \draw[very thin, gray]   (2.9,0.15) -- (2.7,-0.15) ;
        \draw[thin, blue]  (-.4,-0.15) -- (2.7,-0.15);

\end{scope}
    
\end{tikzpicture}
\end{center}

    \noindent
    So the picture below depicts
    $ \ek{N} \cdot \Spi^N \cdot  \ek{N} =  \ek{N}$ implying $\eN(\Spi^N -1)\eN = 0$. 
    Thus by Proposition \ref{prop:good-idempotent} and Proposition \ref{prop:central}, we conclude $\Spi^N=1$.

    \medskip
    
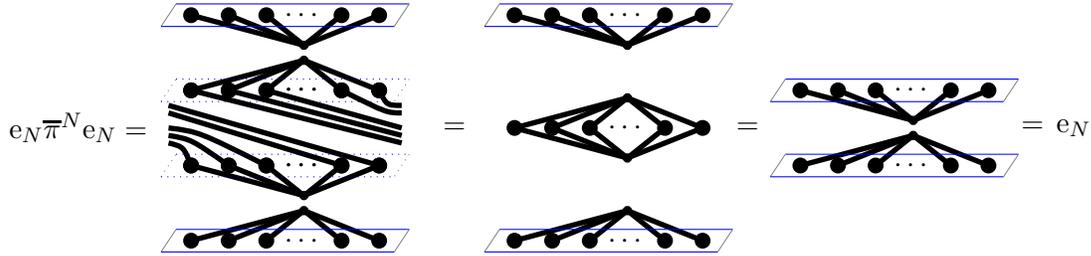
\begin{tikzpicture}[line width=2pt, scale=1.0]
 \begin{scope}[shift={(0,2)}, scale=1]
     \node[bV] at (.5,1){}; \node[bV] at (0.5,0){}; 
        \node[bV] at (0,1){}; \node[bV] at (0,0){}; 
          \node[bV] at (1,1){}; \node[bV] at (1,0){}; 
          \node at (1.5,1) {$\cdots$};
          \node at (1.5,0) {$\cdots$};
        \node[bV] at (2,1){}; \node[bV] at (2,0){}; 
        \node[bV] at (2.5,1){}; \node[bV] at (2.5,0){}; 
         \draw[line width=2pt] (0,1) -- (1.5,0.6);
         \draw[line width=2pt] (.5,1) -- (1.5,0.6);
         \draw[line width=2pt] (1,1) -- (1.5,0.6);
        \draw[line width=2pt] (2,1) -- (1.5,0.6);
        \draw[line width=2pt] (2.5,1) -- (1.5,0.6);
         \draw[line width=2pt] (0,0) -- (1.5,0.4);
         \draw[line width=2pt] (.5,0) -- (1.5,0.4);
         \draw[line width=2pt] (1,0) -- (1.5,0.4);
        \draw[line width=2pt] (2,0) -- (1.5,0.4);
        \draw[line width=2pt] (2.5,0) -- (1.5,0.4);

        \node[V] at (1.5,.4){}; \node[V] at (1.5,0.6){}; 
        \draw[thin, blue, dotted]  (-.2,0.15) -- (2.9,0.15);
        \draw[thin, gray, dotted] (-.4,-0.15) -- (-.2,0.15);
        \draw[thin, gray,dotted]   (2.9,0.15) -- (2.7,-0.15) ;
        \draw[thin, blue,dotted]  (-.4,-0.15) -- (2.7,-0.15);
        \draw[thin, blue]  (-.2,1.15) -- (2.9,1.15);
        \draw[very thin, gray] (-.4,0.85) -- (-.2,1.15);
        \draw[very thin, gray]   (2.9,1.15) -- (2.7,0.85) ;
        \draw[thin, blue]  (-.4,0.85) -- (2.7,0.85);
\end{scope}

 \begin{scope}[shift={(0,1)}, scale=1]
          \node at (-1.5,.5) {$\eN \Spi^N \eN = $};
          \node at (3.5,.5) {$=$};
     \node[bV] at (.5,1){}; \node[bV] at (0.5,0){}; 
        \node[bV] at (0,1){}; \node[bV] at (0,0){}; 
          \node[bV] at (1,1){}; \node[bV] at (1,0){}; 
          \node at (1.5,1) {$\cdots$};
          \node at (1.5,0) {$\cdots$};
        \node[bV] at (2,1){}; \node[bV] at (2,0){}; 
        \node[bV] at (2.5,1){}; \node[bV] at (2.5,0){}; 

         \draw[line width=2pt] (0,1) -- (2.8,0.3);
         \draw[line width=2pt] (.5,1) -- (2.8,0.4);
         \draw[line width=2pt] (1,1) -- (2.8,0.5);
        \draw[line width=2pt] (2,1) .. controls (2.6,.7) .. (2.8,0.7);
        \draw[line width=2pt] (2.5,1) .. controls (2.6,.8) .. (2.8,0.8);
         \draw[line width=2pt] (-0.3,0.3) .. controls (-0.1,.25) .. (0,0);
         \draw[line width=2pt] (-0.3,0.4) .. controls (-0.0,.35) .. (.5,0);
         \draw[line width=2pt] (-0.30,0.5) .. controls (0.1,.45) .. (1.0,0);
         \draw[line width=2pt] (-0.30,0.7) -- (2.0,0);
        \draw[line width=2pt] (-0.30,0.8) -- (2.5,0);

\end{scope}
 \begin{scope}[shift={(0,0)}, scale=1]
     \node[bV] at (.5,1){}; \node[bV] at (0.5,0){}; 
        \node[bV] at (0,1){}; \node[bV] at (0,0){}; 
          \node[bV] at (1,1){}; \node[bV] at (1,0){}; 
          \node at (1.5,1) {$\cdots$};
          \node at (1.5,0) {$\cdots$};
        \node[bV] at (2,1){}; \node[bV] at (2,0){}; 
        \node[bV] at (2.5,1){}; \node[bV] at (2.5,0){}; 
         \draw[line width=2pt] (0,1) -- (1.5,0.6);
         \draw[line width=2pt] (.5,1) -- (1.5,0.6);
         \draw[line width=2pt] (1,1) -- (1.5,0.6);
        \draw[line width=2pt] (2,1) -- (1.5,0.6);
        \draw[line width=2pt] (2.5,1) -- (1.5,0.6);
         \draw[line width=2pt] (0,0) -- (1.5,0.4);
         \draw[line width=2pt] (.5,0) -- (1.5,0.4);
         \draw[line width=2pt] (1,0) -- (1.5,0.4);
        \draw[line width=2pt] (2,0) -- (1.5,0.4);
        \draw[line width=2pt] (2.5,0) -- (1.5,0.4);

        \node[V] at (1.5,.4){}; \node[V] at (1.5,0.6){}; 
        \draw[very thin, blue]  (-.2,0.15) -- (2.9,0.15);
        \draw[very thin, gray] (-.4,-0.15) -- (-.2,0.15);
        \draw[very thin, gray]   (2.9,0.15) -- (2.7,-0.15) ;
        \draw[thin, blue]  (-.4,-0.15) -- (2.7,-0.15);
        \draw[thin, blue, dotted]  (-.2,1.15) -- (2.9,1.15);
        \draw[thin, gray,dotted] (-.4,0.85) -- (-.2,1.15);
        \draw[thin, gray,dotted]   (2.9,1.15) -- (2.7,0.85) ;
        \draw[thin, blue,dotted]  (-.4,0.85) -- (2.7,0.85);

\end{scope}

 \begin{scope}[shift={(4.3,2)}, scale=1]
     \node[bV] at (.5,1){};
        \node[bV] at (0,1){};
          \node[bV] at (1,1){};
          \node at (1.5,1) {$\cdots$};
        \node[bV] at (2,1){};
        \node[bV] at (2.5,1){};
         \draw[line width=2pt] (0,1) -- (1.5,0.6);
         \draw[line width=2pt] (.5,1) -- (1.5,0.6);
         \draw[line width=2pt] (1,1) -- (1.5,0.6);
        \draw[line width=2pt] (2,1) -- (1.5,0.6);
        \draw[line width=2pt] (2.5,1) -- (1.5,0.6);
\node[V] at (1.5,0.6){}; 
        \draw[thin, blue]  (-.2,1.15) -- (2.9,1.15);
        \draw[very thin, gray] (-.4,0.85) -- (-.2,1.15);
        \draw[very thin, gray]   (2.9,1.15) -- (2.7,0.85) ;
        \draw[thin, blue]  (-.4,0.85) -- (2.7,0.85);

\end{scope}

 \begin{scope}[shift={(4.3,0)}, scale=1]
 \node[bV] at (0.5,0){}; 
 \node[bV] at (0,0){}; 
 \node[bV] at (1,0){}; 
          \node at (1.5,0) {$\cdots$};
 \node[bV] at (2,0){}; 
 \node[bV] at (2.5,0){}; 
         \draw[line width=2pt] (0,0) -- (1.5,0.4);
         \draw[line width=2pt] (.5,0) -- (1.5,0.4);
         \draw[line width=2pt] (1,0) -- (1.5,0.4);
        \draw[line width=2pt] (2,0) -- (1.5,0.4);
        \draw[line width=2pt] (2.5,0) -- (1.5,0.4);

        \node[V] at (1.5,.4){};
        \draw[very thin, blue]  (-.2,0.15) -- (2.9,0.15);
        \draw[very thin, gray] (-.4,-0.15) -- (-.2,0.15);
        \draw[very thin, gray]   (2.9,0.15) -- (2.7,-0.15) ;
        \draw[thin, blue]  (-.4,-0.15) -- (2.7,-0.15);
\end{scope}
 \begin{scope}[shift={(4.3,.5)}, scale=1]
     \node[bV] at (.5,1){};  
        \node[bV] at (0,1){}; 
          \node[bV] at (1,1){}; 
          \node at (1.5,1) {$\cdots$};
        \node[bV] at (2,1){}; 
        \node[bV] at (2.5,1){}; 
         \draw[line width=2pt] (0,1) -- (1.5,0.6);
         \draw[line width=2pt] (.5,1) -- (1.5,0.6);
         \draw[line width=2pt] (1,1) -- (1.5,0.6);
        \draw[line width=2pt] (2,1) -- (1.5,0.6);
        \draw[line width=2pt] (2.5,1) -- (1.5,0.6);
\node[V] at (1.5,0.6){}; 
\end{scope}
 \begin{scope}[shift={(4.3,1.5)}, scale=1]
         \draw[line width=2pt] (0,0) -- (1.5,0.4);
         \draw[line width=2pt] (.5,0) -- (1.5,0.4);
         \draw[line width=2pt] (1,0) -- (1.5,0.4);
        \draw[line width=2pt] (2,0) -- (1.5,0.4);
        \draw[line width=2pt] (2.5,0) -- (1.5,0.4);

        \node[V] at (1.5,.4){};
\end{scope}

 \begin{scope}[shift={(8.1,1)}, scale=1]
          \node at (3.4,.5) {$=\, \eN  $};
          \node at (-0.7,.5) {$=$};
     \node[bV] at (.5,1){}; \node[bV] at (0.5,0){}; 
        \node[bV] at (0,1){}; \node[bV] at (0,0){}; 
          \node[bV] at (1,1){}; \node[bV] at (1,0){}; 
          \node at (1.5,1) {$\cdots$};
          \node at (1.5,0) {$\cdots$};
        \node[bV] at (2,1){}; \node[bV] at (2,0){}; 
        \node[bV] at (2.5,1){}; \node[bV] at (2.5,0){}; 
         \draw[line width=2pt] (0,1) -- (1.5,0.6);
         \draw[line width=2pt] (.5,1) -- (1.5,0.6);
         \draw[line width=2pt] (1,1) -- (1.5,0.6);
        \draw[line width=2pt] (2,1) -- (1.5,0.6);
        \draw[line width=2pt] (2.5,1) -- (1.5,0.6);
         \draw[line width=2pt] (0,0) -- (1.5,0.4);
         \draw[line width=2pt] (.5,0) -- (1.5,0.4);
         \draw[line width=2pt] (1,0) -- (1.5,0.4);
        \draw[line width=2pt] (2,0) -- (1.5,0.4);
        \draw[line width=2pt] (2.5,0) -- (1.5,0.4);

        \node[V] at (1.5,.4){}; \node[V] at (1.5,0.6){}; 
        \draw[very thin, blue]  (-.2,0.15) -- (2.9,0.15);
        \draw[very thin, gray] (-.4,-0.15) -- (-.2,0.15);
        \draw[very thin, gray]   (2.9,0.15) -- (2.7,-0.15) ;
        \draw[thin, blue]  (-.4,-0.15) -- (2.7,-0.15);
        \draw[thin, blue]  (-.2,1.15) -- (2.9,1.15);
        \draw[very thin, gray] (-.4,0.85) -- (-.2,1.15);
        \draw[very thin, gray]   (2.9,1.15) -- (2.7,0.85) ;
        \draw[thin, blue]  (-.4,0.85) -- (2.7,0.85);

\end{scope}
\end{tikzpicture}

    \end{proof}

    \begin{remark}
        We also conjecture that $\Spi^{\LCM(N,n)}=1$ in $\SkAlgNn{\SLN}{n}$.   For instance the above proof can be easily modified to cover the case $N \mid n$, i.e., $n=kN$, using $\eNk$ for $k \in \Z_{>0}$.
        We have a diagrammatic proof of this conjecture (not included) for $n=1$ and all $N$.
    \end{remark}
    
\mvhide{can also sketch diagrammatic proof $\Spi^N=1$ for the $n=1$ $\SLN$ relative skalg. should  i?}

 \section{Applications to quantum character theory} \label{sec:quantum character}

In this section, we record some consequences of our results for the category $\Dqstr{G}$ of strongly equivariant $\Dq$-modules. We refer to \cite{GJV} for definitions. 

Recall that there is an equivalence of categories
\[
\SkCat_{G}(T^2)\modu \simeq \Dqstr{G}
\]
under which the objects $\Dist_W$ for $W \in \Repq(G)$, as defined in this paper correspond to the objects $\Dist_W$ as defined in \cite{GJV}. In 
 \cite{GJV} we also used the notation $\HKuniv$ to refer to the object $\Dist = \Dist_{\trivial}$ in $\Dqstr{G}$.

The following result (stated as Theorem 1.6 in \cite{GJV}) is a direct translation of Theorem \ref{mainthm:generators} (1) into the language of quantum character theory.

\begin{theorem}\label{thm:HK generates (GJVY)}
 Let $G=\GLN$. The object $\HKuniv = \Dist$ is a (compact, projective) generator of $\Dqstr{G}$. 
\end{theorem}
Combining Theorem \ref{thm:HK generates (GJVY)} above with Theorem 1.1 in \cite{GJV}, we deduce that there is an equivalence of categories
 
   \begin{gather}
   \label{eq:Dq equiv eHe}
    \Dqstr{G} \simeq \ee \cdot \DN \cdot \ee\modu \simeq \DqH^W\modu
    \end{gather} 
for $\ee = \eN$. In fact, by Theorem \ref{mainthm:morita-DAHA} the bimodule $\DN \cdot \ee $ defines a Morita equivalence between $\DN$ and $\ee \cdot \DN \cdot \ee$, 
 so one may replace $\ee \cdot \DN \cdot \ee$ by $\DN$ in \eqref{eq:Dq equiv eHe} above.  Using a similar, well-known Morita equivalence for the symmetric idempotent $\etriv{N}$, one may replace $\DqH^{\SN} = \Coinv$ by $\DqH\# \SN = \SmashD$ as in Proposition \ref{prop:morita smash}.

Part (2) of Theorem \ref{mainthm:generators} translates into the following statement. Let $\Dqstr{G}_{\bar n}$ denote the full subcategory of $\Dqstr{\SLN}$ for which $Z(\SLN) \cong \mu_N$ acts via the character $\bar n$. Note that the category $\Dqstr{\SLN}$ decomposes as an orthogonal direct sum
    \[
    \Dqstr{\SLN} = \bigoplus_{\bar n \in \Z/N\Z} \Dqstr{\SLN}_{\bar n}.
    \]
\begin{theorem}\label{thm:DqstrSLN}
    Let $G=\SLN$. Then $\Dist_{V^{\otimes n}}$ is a (compact, projective) generator for $\Dqstr{G}_{\bar{n}}$.   In particular 
   \[\Dist \oplus \Dist_{V} \oplus \ldots \oplus \Dist_{V^{\otimes (N-1)}}\]  
   is a (compact, projective) generator for $\Dqstr{G}$.
\end{theorem}


The following result is an improvement on Theorem 1.3 from \cite{GJV}, in which we showed the $M_{\lambdatup}$ appearing below were indecomposable. 

\begin{theorem}
    Let $G=\GLN$ or $\SLN$. In the decomposition from Theorem 1.3 of \cite{GJV}
    \[
    \HKe = \bigoplus_{\lambdatup} M_\lambdatup \boxtimes S^\lambdatup
    \]
    the summands $M_\lambdatup$ are simple $\q$-character sheaves. 
\end{theorem}

\begin{proof}
    It has already been shown in \cite{GJV} that $F_N(M_\lambda)$ is a simple $\DN$-module. In the $\GLN$-case, $F_N$ is an equivalence of categories, from which it follows that $M_\lambda$ is simple as claimed. 
    In the $\SLN$-case, the $M_\lambdatup$ are all contained in the subcategory $\Dqstr{G}_{\overline{0}}$ (as they are all quotients of $\HKuniv$) on which the functor $F_N$ is an equivalence, so the argument still applies. 
\end{proof}



\printbibliography

\end{document}